\documentclass[11pt]{amsart}
\usepackage{mathptmx}
\usepackage[margin=1in]{geometry}

\usepackage{amsmath}
\usepackage{amscd}
\usepackage{tikz-cd}
\usepackage{amssymb}
\usepackage[mathscr]{euscript}
\usepackage{enumerate}
\usepackage{hyperref}

\DeclareMathOperator{\Hom}{Hom}
\DeclareMathOperator{\End}{End}
\DeclareMathOperator{\Soc}{Soc}

\DeclareMathOperator{\Ext}{Ext}
\DeclareMathOperator{\Spec}{Spec}
\DeclareMathOperator{\Ann}{Ann}
\DeclareMathOperator{\hgt}{ht}
\DeclareMathOperator{\gr}{gr}
\DeclareMathOperator{\im}{im}
\DeclareMathOperator{\coker}{coker}
\DeclareMathOperator{\Supp}{Supp}
\DeclareMathOperator{\ch}{char}

\renewcommand{\l}{l}
\newcommand{\D}{\mathcal{D}}

\theoremstyle{plain}
\newtheorem{thm}{Theorem}[section]
\newtheorem{bigthm}{Theorem}
\newtheorem{prop}[thm]{Proposition}
\newtheorem{lem}[thm]{Lemma}
\newtheorem{cor}[thm]{Corollary}

\theoremstyle{definition}
\newtheorem{definition}[thm]{Definition}
\newtheorem{conjecture}[thm]{Conjecture}

\theoremstyle{remark}
\newtheorem{remark}[thm]{Remark}
\newtheorem{example}[thm]{Example}
\newtheorem{question}[thm]{Question}
\newtheorem{convention}[thm]{Convention}

\makeatletter
\def\l@section{\@tocline{1}{0pt}{1pc}{}{}}
\def\l@subsection{\@tocline{2}{0pt}{1pc}{4.6em}{}}
\def\l@subsubsection{\@tocline{3}{0pt}{1pc}{7.6em}{}}
\renewcommand{\tocsection}[3]{%
\indentlabel{\@ifnotempty{#2}{\makebox[2.3em][l]{%
\ignorespaces#1 #2.\hfill}}}#3}
\renewcommand{\tocsubsection}[3]{%
\indentlabel{\@ifnotempty{#2}{\hspace*{2.3em}\makebox[2.3em][l]{%
\ignorespaces#1 #2.\hfill}}}#3}
\renewcommand{\tocsubsubsection}[3]{%
\indentlabel{\@ifnotempty{#2}{\hspace*{4.6em}\makebox[3em][l]{%
\ignorespaces#1 #2.\hfill}}}#3}
\makeatother

\begin{document}

\title[De Rham homology and cohomology]{On the de Rham homology and cohomology of a complete local ring in equicharacteristic zero}
\author{Nicholas Switala}
\address{Department of Mathematics, Statistics, and Computer Science \\ University of Illinois at Chicago \\ 322 SEO (M/C 249) \\ 851 S. Morgan Street \\ Chicago, IL 60607}
\email{nswitala@uic.edu}
\thanks{NSF support through grants DMS-0701127 and DMS-1604503 is gratefully acknowledged.}
\subjclass[2010]{Primary 13D45, 14F40}
\keywords{Local cohomology, algebraic de Rham cohomology, Matlis duality, $\D$-modules}

\begin{abstract}
Let $A$ be a complete local ring with a coefficient field $k$ of characteristic zero, and let $Y$ be its spectrum.  The de Rham homology and cohomology of $Y$ have been defined by R. Hartshorne using a choice of surjection $R \rightarrow A$ where $R$ is a complete regular local $k$-algebra: the resulting objects are independent of the chosen surjection.  We prove that the Hodge-de Rham spectral sequences abutting to the de Rham homology and cohomology of $Y$, beginning with their $E_2$-terms, are independent of the chosen surjection (up to a degree shift in the homology case) and consist of finite-dimensional $k$-spaces.  These $E_2$-terms therefore provide invariants of $A$ analogous to the Lyubeznik numbers.  As part of our proofs we develop a theory of Matlis duality in relation to $\D$-modules that is of independent interest. Some of the highlights of this theory are that if $R$ is a complete regular local ring containing $k$ and $\D = \D(R,k)$ is the ring of $k$-linear differential operators on $R$, then the Matlis dual $D(M)$ of any left $\D$-module $M$ can again be given a structure of left $\D$-module, and if $M$ is a holonomic $\D$-module, then the de Rham cohomology spaces of $D(M)$ are $k$-dual to those of $M$.
\end{abstract}

\maketitle

\tableofcontents

\section{Introduction}\label{intro} 

In \cite{derham}, Hartshorne constructs local and global algebraic de Rham homology and cohomology theories for schemes over a field $k$ of characteristic zero.  The global theories, defined for any scheme $Y$ of finite type over $k$, are defined using a choice of embedding $Y \hookrightarrow X$ into a smooth scheme over $k$ (or a local system of such embeddings if a global embedding does not exist): one computes the hypercohomology of certain complexes of sheaves on $X$ or on the formal completion of $Y$ in $X$.  Hartshorne's primary interest \cite[Remark, p. 70]{derham} in constructing the local theories is the case where $Y$ is the spectrum of a complete local ring.  For technical reasons, he defines the local theories for a larger class of schemes (``category $\mathcal{C}$'': see \cite[p. 65]{derham}), and proves the corresponding finiteness and duality results by reducing to the global case and using resolution of singularities.

There is a sketch in \cite[pp. 70--71]{derham} of a failed attempt to prove that local algebraic de Rham homology and cohomology are dual using Grothendieck's local duality theorem.  This sketch is the inspiration for the present paper, as its ideas can now be profitably pursued using Lyubeznik's work on the $\D$-module structure of local cohomology \cite{lyubeznik}.  This structure allows us to speak of the de Rham complex of a local cohomology module, and knowledge of such complexes in turn enables us to better understand the early terms of the spectral sequences appearing in Hartshorne's work.  We are able to give a purely local proof for both the embedding-independence and the finiteness of local de Rham homology and cohomology, replacing the global methods of algebraic geometry (including resolution of singularities) with the theory of algebraic $\D$-modules over a formal power series ring in characteristic zero.  Along the way, we will define a new set of invariants for complete local rings, analogous to the Lyubeznik numbers.  Our proof shows more than is contained in \cite{derham}, namely that the entire Hodge-de Rham spectral sequences for homology and cohomology (with the exception of the first term) are embedding-independent (up to a degree shift in the homology case) and consist of finite-dimensional $k$-spaces.

We now give more detail, recalling Hartshorne's results and stating ours.  Let $A$ be a complete local ring with coefficient field $k$ of characteristic zero (that is, $k$ is the residue field of $A$, and $A$ contains a field isomorphic to $k$).  We view $A$ as a $k$-algebra via this coefficient field.  By Cohen's structure theorem, there exists a surjection of $k$-algebras $\pi: R \rightarrow A$ where $R$ is a complete regular local $k$-algebra, which must take the form $R = k[[x_1, \ldots, x_n]]$ for some $n$.  Let $I \subset R$ be the kernel of this surjection.  We have a corresponding closed immersion $Y \hookrightarrow X$ where $Y = \Spec(A)$ and $X = \Spec(R)$.  In \cite{derham}, the \emph{de Rham homology} of the local scheme $Y$ is defined as $H_i^{dR}(Y) = \mathbf{H}_Y^{2n-i}(X, \Omega_X^{\bullet})$, the hypercohomology (supported at $Y$) of the complex of \emph{continuous} differential forms on $X$.  The differentials in this complex are merely $k$-linear, so the $H_i^{dR}(Y)$ are $k$-spaces.  

Now let $\widehat{X}$ be the formal completion of $Y$ in $X$ \cite[\S II.9]{hartshorneAG}, that is, the topological space $Y$ equipped with the structure of a locally ringed space via the sheaf of rings $\varprojlim \mathcal{O}_X/\mathcal{I}^n$ where $\mathcal{I} \subset \mathcal{O}_X$ is the quasi-coherent ideal sheaf defining the chosen embedding $Y \hookrightarrow X$ (every $\mathcal{O}_X/\mathcal{I}^n$ is supported at $Y$ and thus can be viewed as a sheaf of rings on $Y$).  The differentials in the complex $\Omega_X^{\bullet}$ are $\mathcal{I}$-adically continuous and thus pass to $\mathcal{I}$-adic completions.  We obtain in this way a complex $\widehat{\Omega}^{\bullet}_X$ of sheaves on $\widehat{X}$, the formal completion of $\Omega_X^{\bullet}$, whose differentials are again merely $k$-linear.  In \cite{derham}, the \emph{(local) de Rham cohomology} of the local scheme $Y$ is defined as $H^i_{P, dR}(Y) = \mathbf{H}^i_P(\widehat{X}, \widehat{\Omega}^{\bullet}_X)$, where $P$ is the closed point of $Y$.  After making these definitions, Hartshorne establishes the following properties:

\begin{thm}\cite[Thm. III.1.1, Thm. III.2.1]{derham}\label{hartind}
Let $A$ be a complete local ring with coefficient field $k$ of characteristic zero and $Y = \Spec(A)$.
\begin{enumerate}[(a)]
\item The de Rham homology spaces $H_i^{dR}(Y)$ and cohomology spaces $H^i_{P, dR}(Y)$ as defined above are independent of the surjection of $k$-algebras $R \rightarrow A$ used in their definitions.
\item For all $i$, $H_i^{dR}(Y)$ and $H^i_{P, dR}(Y)$ are finite-dimensional $k$-spaces.
\item For all $i$, the $k$-spaces $H_i^{dR}(Y)$ and $H^i_{P, dR}(Y)$ are $k$-dual to each other.
\end{enumerate}
\end{thm}

The de Rham homology and cohomology of $Y = \Spec(A)$ are both defined using hypercohomology.  As is well-known, there are in general two spectral sequences converging to the hypercohomology of a complex (\cite[11.4.3]{ega31}; also see subsection \ref{specseq}).  If $K^{\bullet}$ is a complex of sheaves of Abelian groups on a topological space $Z$, the first of these spectral sequences begins $E_1^{p,q} = H^q(Z, K^p)$ and has abutment $\mathbf{H}^{p+q}(Z, K^{\bullet})$.  In our case, this takes the form of the \emph{Hodge-de Rham homology spectral sequence}, which begins $E_1^{n-p,n-q} = H^{n-q}_Y(X, \Omega_X^{n-p})$ and has abutment $H_{p+q}^{dR}(Y)$, as well as the \emph{Hodge-de Rham cohomology spectral sequence}, which begins $\tilde{E}_1^{p,q} = H^q_P(\widehat{X}, \widehat{\Omega}^p_X)$ and has abutment $H^{p+q}_{P, dR}(Y)$.  \emph{A priori}, these spectral sequences depend on the choice of surjection $R \rightarrow A$ from a complete regular local $k$-algebra.  We prove stronger versions of Hartshorne's results for these spectral sequences.  Our theorem for de Rham homology is:

\begin{bigthm}\label{mainthmA}
Let $A$ be a complete local ring with coefficient field $k$ of characteristic zero.  Viewing $A$ as a $k$-algebra via this coefficient field, let $R \rightarrow A$ be a choice of $k$-algebra surjection from a complete regular local $k$-algebra.  Associated with this surjection we have a Hodge-de Rham spectral sequence for homology as above.
\begin{enumerate}[(a)]
\item Beginning with the $E_2$-term, the isomorphism class of the homology spectral sequence with its abutment is independent of the choice of regular $k$-algebra and surjection $R \rightarrow A$, up to a degree shift.
\item The $k$-spaces $E_2^{p,q}$ appearing in the $E_2$-term of the homology spectral sequence are finite-dimensional.
\end{enumerate}
\end{bigthm}

The meaning of ``up to a degree shift'' in the statement of part (a) is the following: given two surjections $R \rightarrow A$ and $R' \rightarrow A$ from complete regular local $k$-algebras, where $\dim(R) = n$ and $\dim(R') = n'$, we obtain two Hodge-de Rham spectral sequences $E_{\bullet, R}^{\bullet, \bullet}$ and $\mathbf{E}_{\bullet, R'}^{\bullet, \bullet}$.  Part (a) asserts that there is a morphism $E_{\bullet, R}^{\bullet, \bullet} \rightarrow \mathbf{E}_{\bullet, R'}^{\bullet, \bullet}$ of bidegree $(n'-n,n'-n)$ between these spectral sequences which is an isomorphism on the objects of the $E_2$- (and later) terms (see subsection \ref{specseq} for the precise definitions of the terms used here).

We also have the analogue (without a degree shift) for de Rham cohomology:

\begin{bigthm}\label{mainthmB}
Let $A$ and $R$ be as in Theorem \ref{mainthmA}.  Associated with this surjection we also have a local Hodge-de Rham spectral sequence for cohomology.
\begin{enumerate}[(a)]
\item Beginning with the $E_2$-term, the isomorphism class of the cohomology spectral sequence with its abutment is independent of the choice of regular $k$-algebra and surjection $R \rightarrow A$.
\item The $k$-spaces $\tilde{E}_2^{pq}$ appearing in the $E_2$-term of the cohomology spectral sequence are finite-dimensional.
\end{enumerate}
\end{bigthm}

\begin{remark}\label{ogus}
In subsection \ref{specseq}, the notion of an \emph{isomorphism} of spectral sequences is defined.  As described in this subsection, the ingredients of a spectral sequence are the objects and differentials in the $E_r$-term for all $r$, the abutment objects, and the filtrations on the abutment objects, together with the isomorphisms relating the terms with their successors and with the abutment.  The assertion of Theorems \ref{mainthmA} and \ref{mainthmB} is that \emph{all} of these ingredients (except for the $E_1$-terms) are independent of the chosen regular $k$-algebra and surjection.  Therefore parts (a) and (b) of Theorem \ref{hartind} are subsumed by Theorems \ref{mainthmA} and \ref{mainthmB}, which provide more information: the isomorphism classes of $H_{p+q}^{dR}(Y)$ and $H^{p+q}_{P, dR}(Y)$ as \emph{filtered objects} are independent of the surjection.  We obtain many numerical invariants of $(A, k)$ from the spectral sequence: the (finite) dimensions of the kernels and cokernels of the differentials $d_r$ for all $r \geq 2$, together with the dimensions of the filtered pieces of the abutment.
\end{remark}

The proof of Theorem \ref{mainthmB} requires the development of a theory of Matlis duality for $\D$-modules.  This theory is worked out in sections \ref{klinear}, \ref{dmod}, and \ref{mittleff}, which form a unit of independent interest.  If $R$ is a complete local ring with coefficient field $k$ of characteristic zero, we can consider the ring $\D(R,k)$ of $k$-linear differential operators on $R$.  Given any left module over this ring, we can define its de Rham complex and speak of its de Rham cohomology spaces.  We prove that the Matlis dual of a left $\D(R,k)$-module has a natural structure of right $\D(R,k)$-module.  Specializing to the case of a complete \emph{regular} local ring, we are able to obtain information about the de Rham cohomology of a Matlis dual.  In this case, the dual of a left $\D(R,k)$-module can again be viewed as a left $\D(R,k)$-module.  The following is the main result of our theory of Matlis duality for $\D$-modules (its two assertions are proved separately below as Proposition \ref{regleft} and Theorem \ref{dualcoh}):

\begin{bigthm}\label{ddualcoh}
Let $k$ be a field of characteristic zero, let $R = k[[x_1, \ldots x_n]]$ be a formal power series ring over $k$, and let $\D = \D(R,k)$ be the ring of $k$-linear differential operators on $R$.  If $M$ is a left $\D$-module, the Matlis dual $D(M)$ of $M$ with respect to $R$ can also be given a natural structure of left $\D$-module.  We write $H^{*}_{dR}(M)$ for the de Rham cohomology of a left $\D$-module.  If $M$ is a \emph{holonomic} left $\D$-module, then for every $i$, we have an isomorphism of $k$-spaces
\[
(H^i_{dR}(M))^{\vee} \simeq H^{n-i}_{dR}(D(M))
\]
where $\vee$ denotes $k$-linear dual.
\end{bigthm}

\begin{remark}
If $M$ is holonomic, its de Rham cohomology spaces are known to be finite-dimensional (see Theorem \ref{findimdR}), and so it follows from Theorem \ref{ddualcoh} that $D(M)$ also has finite-dimensional de Rham cohomology.  Since $D(M)$ is not, in general, a holonomic $\D$-module (see Remark \ref{hellus} below), this is not clear \emph{a priori}.
\end{remark}

When applied to the $E_1$- and $E_2$-terms of the homology and cohomology spectral sequences, Theorem \ref{ddualcoh} has the following consequence:

\begin{bigthm}\label{dualE2}
Let $A$, $k$, and $R$ be as in the statement of Theorem \ref{mainthmA}.  The objects in the $E_2$-terms of the homology and cohomology spectral sequences are $k$-dual to each other: for all $p$ and $q$, $E_2^{n-p,n-q} \simeq (\tilde{E}_2^{pq})^{\vee}$.
\end{bigthm}

We conjecture that the entire spectral sequences, beginning with $E_2$, should be $k$-dual to each other (see section \ref{openqs}), but at present we are able only to prove the preceding statement.

As was already clear to Hartshorne \cite[p. 71]{derham}, the Matlis duals of the objects appearing in the $E_1$-term of the homology spectral sequence are exactly the corresponding $E_1$-objects for de Rham cohomology.  What is much less clear is the relationship between the differentials.  The $E_1$-differentials for de Rham homology are merely $k$-linear, so the usual definition of Matlis duality over $R$ cannot be applied to them.  A large part of this paper is devoted to the problem of applying Matlis duality to these $k$-linear maps and establishing that, in our setting, Matlis duality at the $E_1$-term gives rise to $k$-duality at the $E_2$-term.

In more detail, the outline of the paper is as follows.  In section \ref{homology}, after reviewing some preliminary material on differential operators, de Rham complexes, and spectral sequences, we prove Theorem \ref{mainthmA} using local algebra.  In the course of this proof, we define a new set of invariants for complete local rings with coefficient fields of characteristic zero.  In section \ref{klinear}, Matlis duality for local rings containing a field $k$ is interpreted in terms of $k$-linear maps, after SGA2 \cite{SGA2}.  This allows us to dualize continuous maps between finitely generated modules over such rings.  We pass to direct limits in order to dualize $k$-linear maps between arbitrary modules that satisfy certain finiteness and continuity conditions.  In section \ref{dmod}, we consider the case of $\D$-modules for \emph{complete} local rings containing $k$; we describe a natural right $\D$-module structure on the Matlis dual $D(M)$ of a left $\D$-module $M$.  In the special case of a formal power series ring, we can regard $D(M)$ as a left $\D$-module as well using a simple ``transpose'' operation, and thus define its de Rham complex.  We determine the cohomology of this complex in section \ref{mittleff} in the case of holonomic $M$, completing the proof of Theorem \ref{ddualcoh}. The specific case of local cohomology is considered next.  In section \ref{formal}, we work out precisely what happens to the action of derivations on a local cohomology module.  Finally, in section \ref{cohom}, we give a self-contained proof of Theorem \ref{mainthmB}(a) and combine the results of sections \ref{homology}, \ref{mittleff}, and \ref{formal} to prove Theorem \ref{mainthmB}(b) and Theorem \ref{dualE2}. 

\begin{remark}\label{hellus}
After we identify the objects in the $E_2$-term of the homology spectral sequence with de Rham cohomology spaces of local cohomology modules in section \ref{homology}, part (b) of Theorem \ref{mainthmA} follows from Lyubeznik's result (\cite[2.2(d)]{lyubeznik}; \emph{cf.} \cite{mebkhout}) that local cohomology modules are holonomic $\D$-modules, as it is known that holonomic $\D$-modules have finite-dimensional de Rham cohomology.  However, the objects in the $E_2$-term of the cohomology spectral sequence are not de Rham cohomology spaces of holonomic $\D$-modules: as shown in section \ref{formal}, they are de Rham cohomology spaces of \emph{Matlis duals} of local cohomology modules.  Hellus has shown \cite[Cor. 2.6]{hellus} that Matlis duals of local cohomology modules have, in general, infinitely many associated primes, implying that they need not be holonomic $\D$-modules (which always have finitely many associated primes \cite[Cor. 3.6(c)]{lyubeznik}).  It is therefore surprising that the cohomology $E_2$-objects, which are, in general, de Rham cohomology spaces of \emph{non}-holonomic $\D$-modules, are still finite-dimensional.  For this reason, the proof of Theorem \ref{mainthmB}(b) is significantly more difficult than the proofs of the other parts of our main theorems.
\end{remark}

The results in this paper are from the author's thesis, supervised by Gennady Lyubeznik.  We would like to thank Lyubeznik, not only for suggesting this topic, but also for providing constant guidance and insight.  We also thank William Messing for profitable conversations on algebraic de Rham cohomology and Lemma \ref{sscomp}, as well as an exchange concerning the lemmas in section \ref{mittleff}; Joel Gomez for asking a question that led to Proposition \ref{doubleddual}; and Paul Garrett for an exchange concerning Lemma \ref{finddual}.  Finally, we thank Robin Hartshorne, for reading a preliminary version of this paper and making numerous suggestions that improved the exposition.

\section{The de Rham homology of a complete local ring}\label{homology}

In this section, after reviewing some preliminary material on modules over rings of differential operators and spectral sequences, we recall Hartshorne's definition of algebraic de Rham homology \cite[Ch. III]{derham} in the case of the spectrum of a complete local ring in equicharacteristic zero, and examine the associated Hodge-de Rham spectral sequence.  We give proofs that this de Rham homology is intrinsically defined and finite-dimensional which are purely local; in fact, we prove more, namely that up to a degree shift, the entire Hodge-de Rham spectral sequence (with the exception of the first term) is intrinsically defined and has finite-dimensional objects.  As a byproduct of our proof, we obtain a new set of invariants for complete local rings analogous to the \emph{Lyubeznik numbers} \cite[Thm.-Def. 4.1]{lyubeznik}.

\subsection{Preliminaries on $\D$-modules}\label{dmodprelim} 

Let $R$ be a commutative ring and $k \subset R$ a commutative subring.  The ring $\D(R,k)$ of $k$-linear differential operators on $R$, a subring of $\End_k(R)$, is defined recursively as follows \cite[\S 16]{EGAIV}.  A differential operator $R \rightarrow R$ of order zero is multiplication by an element of $R$.  Supposing that differential operators of order $\leq j-1$ have been defined, $d \in \End_k(R)$ is said to be a differential operator of order $\leq j$ if, for all $r \in R$, the commutator $[d,r] \in \End_k(R)$ is a differential operator of order $\leq j-1$, where $[d,r] = dr - rd$ (the products being taken in $\End_k(R)$).  We write $\D^j(R)$ for the set of differential operators on $R$ of order $\leq j$ and set $\D(R,k) = \cup_j \D^j(R)$.  Every $\D^j(R)$ is naturally a left $R$-module.  If $d \in \D^j(R)$ and $d' \in \D^{\l}(R)$, it is easy to prove by induction on $j + \l$ that $d' \circ d \in \D^{j+\l}(R)$, so $\D(R,k)$ is a ring.

We consider now the special case in which $k$ is a field of characteristic zero and $R = k[[x_1, \ldots, x_n]]$ is a formal power series ring over $k$.  A standard reference for facts about the ring $\D = \D(R,k)$ and left modules over $\D$ in this case is \cite[Ch. 3]{bjork}; we summarize some of these facts now.  The ring $\D$, viewed as a left $R$-module, is freely generated by monomials in the partial differentiation operators $\partial_1 = \frac{\partial}{\partial x_1}, \ldots, \partial_n = \frac{\partial}{\partial x_n}$ (\cite[Thm. 16.11.2]{EGAIV}: here the characteristic-zero assumption is necessary).  This ring has an increasing filtration $\{\D(\nu)\}$, called the \emph{order filtration}, where $\D(\nu)$ consists of those differential operators of order $\leq \nu$ (the order of an element of $\D$ is the maximum of the orders of its summands, and the order of a single summand $\rho \partial_1^{a_1} \cdots \partial_n^{a_n}$ with $\rho \in R$ is $\sum a_i$: this notion of order coincides with the one defined in the previous paragraph).  The associated graded object $\gr(\D) = \oplus \D(\nu)/\D(\nu-1)$ with respect to this filtration is isomorphic to $R[\zeta_1, \ldots, \zeta_n]$ (a commutative ring), where $\zeta_i$ is the image of $\partial_i$ in $\D(1)/\D(0) \subset \gr(\D)$.  

If $M$ is a finitely generated left $\D$-module, there exists a \emph{good filtration} $\{M(\nu)\}$ on $M$, meaning that $M$ becomes a filtered left $\D$-module with respect to the order filtration on $\D$ \emph{and} $\gr(M) = \oplus M(\nu)/M(\nu-1)$ is a finitely generated $\gr(\D)$-module.  We let $J$ be the radical of $\Ann_{\gr(\D)} \gr(M) \subset \gr(\D)$ and set $d(M) = \dim \gr(\D)/J$ (Krull dimension).  The ideal $J$, and hence the number $d(M)$, is independent of the choice of good filtration on $M$.  By \emph{Bernstein's theorem}, if $M \neq 0$ is a finitely generated left $\D$-module, we have $n \leq d(M) \leq 2n$.  In the case $d(M) = n$ we say that $M$ is \emph{holonomic}.  It is known that submodules and quotients of holonomic $\D$-modules are holonomic, an extension of a holonomic $\D$-module by another holonomic $\D$-module is holonomic, holonomic $\D$-modules are of finite length over $\D$, and holonomic $\D$-modules are cyclic (generated over $\D$ by a single element). 

Given any left $\D(R,k)$-module $M$, we can define its \emph{de Rham complex}.  This is a complex of length $n$, denoted $M \otimes \Omega_R^{\bullet}$ (or simply $\Omega_R^{\bullet}$ in the case $M = R$), whose objects are $R$-modules but whose differentials are merely $k$-linear.  It is defined as follows \cite[\S 1.6]{bjork}: for $0 \leq i \leq n$, $M \otimes \Omega^i_R$ is a direct sum of $n \choose i$ copies of $M$, indexed by $i$-tuples $1 \leq j_1 < \cdots < j_i \leq n$.  The summand corresponding to such an $i$-tuple will be written $M \, dx_{j_1} \wedge \cdots \wedge dx_{j_i}$.  

\begin{convention}\label{tensorsubscript}
The subscript $R$ in $\Omega_R^{\bullet}$ indicates over which ring the tensor products of objects are being taken.  To simplify notation, we will follow this convention when de Rham complexes over different rings are being simultaneously considered.
\end{convention}

The $k$-linear differentials $d^i: M \otimes \Omega_R^i \rightarrow M \otimes \Omega_R^{i+1}$ are defined by 
\[
d^i(m \,dx_{j_1} \wedge \cdots \wedge dx_{j_i}) = \sum_{s=1}^n \partial_s(m)\, dx_s \wedge dx_{j_1} \wedge \cdots \wedge dx_{j_i},
\]
with the usual exterior algebra conventions for rearranging the wedge terms, and extended by linearity to the direct sum.  The cohomology objects $h^i(M \otimes \Omega_R^{\bullet})$, which are $k$-spaces, are called the \emph{de Rham cohomology spaces} of the left $\D$-module $M$, and are denoted $H^i_{dR}(M)$.  In the case of a holonomic module, van den Essen proved that these spaces are finite-dimensional:

\begin{thm} \label{findimdR} \cite[Prop. 2.2]{essen}
If $M$ is a holonomic left $\D$-module, its de Rham cohomology $H^i_{dR}(M)$ is a finite-dimensional $k$-space for all $i$.
\end{thm}

\begin{remark}\label{lcdmod}
An important family of examples of holonomic $\D$-modules is that of \emph{local cohomology modules}.  Our basic references for facts about local cohomology modules are Brodmann and Sharp \cite{brodmann} and SGA2 \cite{SGA2}.  If $R$ is a commutative ring and $I \subset R$ is an ideal, the functor $\Gamma_I$ of sections with support at $I$ is a left-exact functor on the category of $R$-modules (for an $R$-module $M$, $\Gamma_I(M)$ consists of those $m \in M$ annihilated by some power of $I$).  The local cohomology modules $H^i_I(M)$ are the right derived functors of $\Gamma_I = H^0_I$ evaluated at $M$.  There is a more general sheaf-theoretic formulation, given in \cite{SGA2}: if $X$ is a topological space, $Y \subset X$ is a locally closed subset, and $\mathcal{F}$ is a sheaf of Abelian groups on $X$, the local cohomology groups $H^i_Y(X, \mathcal{F})$ are obtained by evaluating at $\mathcal{F}$ the right derived functors of $\Gamma_Y$, the functor of global sections supported at $Y$.  The relationship between these two definitions in the case of an affine scheme $X$ is given below in Lemma \ref{E2shape}.  In the case of the ring $R = k[[x_1, \ldots, x_n]]$, Lyubeznik proved that for any ideal $I$, the local cohomology modules $H^i_{I}(R)$ have a natural structure of left $\D$-module \cite{lyubeznik}; \emph{cf.} \cite{mebkhout}.  Indeed, they are \emph{holonomic} $\D$-modules \cite[2.2(d)]{lyubeznik}, a fact which will repeatedly prove crucial for us.
\end{remark}

In addition to the standard long exact cohomology sequence for derived functors, there is another useful long exact sequence for local cohomology: the \emph{Mayer-Vietoris sequence} with respect to two ideals.

\begin{prop}\cite[Thm. 3.2.3]{brodmann}\label{mv}
Let $I$ and $J$ be ideals of a commutative ring $R$.  There is a long exact sequence of $R$-modules
\[
\cdots \rightarrow H^{i-1}_{I \cap J}(M) \rightarrow H^i_{I+J}(M) \rightarrow H^i_I(M) \oplus H^i_J(M) \rightarrow H^i_{I \cap J}(M) \rightarrow H^{i+1}_{I+J}(M) \rightarrow \cdots
\]
for every $R$-module $M$, functorial in $M$.
\end{prop}

Finally, we recall the well-known fact that the de Rham complex of a $\D$-module is independent of the coordinates $x_1, \ldots, x_n$ for $R$.  (See \cite[Prop. 2.5]{nsexpos} for a proof.)

\begin{prop}\label{dRind}
If $R = k[[x_1, \ldots, x_n]]$ and $\D = \D(R,k)$, the de Rham complex of any left $\D$-module $M$ is independent of the chosen regular system of parameters $x_1, \ldots, x_n$ for $R$.
\end{prop}

\subsection{Preliminaries on spectral sequences}\label{specseq} 

As we will be working with morphisms of spectral sequences, we collect some basic facts and definitions in this subsection concerning them.  References for this material include Weibel \cite[Ch. 5]{weibel} and EGA \cite[\S 11]{ega31}.  We will not need to consider convergence issues for unbounded spectral sequences and hence make no mention of such issues here.

\begin{definition}\label{ss}
Let $\mathcal{C}$ be an Abelian category.  A (cohomological) \emph{spectral sequence} consists of the following data: a family $\{E_r^{p,q}\}$ of objects of $\mathcal{C}$ (where $p,q \in \mathbb{Z}$ and $r \geq 1$ or $\geq 2$; with $r$ fixed and $p,q$ varying, we obtain the \emph{$E_r$-term} of the spectral sequence), and morphisms (the \emph{differentials}) $d_r^{p,q}: E_r^{p,q} \rightarrow E_r^{p+r,q-r+1}$ for all $p,q,r$ such that $d_r^{p,q} \circ d_r^{p-r,q+r-1} = 0$ and $\ker(d_r^{p,q})/\im(d_r^{p-r,q+r-1}) \xrightarrow{\sim} E_{r+1}^{p,q}$: a family of such isomorphisms, denoted $\alpha_r^{p,q}$, is part of the data of the spectral sequence.
\end{definition}

Let $E$ be a spectral sequence in an Abelian category $\mathcal{C}$, and suppose that for all $\l$ and for all $r$, there are only finitely many nonzero objects $E_r^{p,q}$ with $p+q=\l$.  Such a spectral sequence is called \emph{bounded}.  (For example, this occurs if $E_r^{p,q} = 0$ whenever $p$ or $q$ is negative, in which case $E$ is called a \emph{first-quadrant} spectral sequence.)  If $E$ is a bounded spectral sequence, for every pair $(p,q)$, there exists $r_0$ such that for all $r \geq r_0$, $d_r^{p,q}$ has zero target, $d_r^{p-r,q+r-1}$ has zero source, and so $E_{r+1}^{p,q} \simeq E_r^{p,q}$.  We denote this stable object by $E_{\infty}^{p,q}$.  We can now define the \emph{abutment} of such a spectral sequence.

\begin{definition}
Let $E$ be a bounded spectral sequence in an Abelian category $\mathcal{C}$.  Suppose we are given a family $E^m$ of objects of $\mathcal{C}$, all endowed with a finite decreasing filtration $E^m = E^m_s \supset E^m_{s+1} \supset \cdots \supset E^m_t = 0$, and for all $p$, an isomorphism $\beta^{p,m-p}: E_{\infty}^{p,m-p} \xrightarrow{\sim} E^m_p/E^m_{p+1}$.  Then we say that the spectral sequence \emph{abuts} or \emph{converges} to $\{E^m\}$ (the \emph{abutment}), and write $E_1^{p,q} \Rightarrow E^m$ or $E_2^{p,q} \Rightarrow E^m$.
\end{definition}  

For example, if $E$ is a first-quadrant spectral sequence with abutment $\{E^m\}$, every $E^m$ has a filtration of length $m+1$ (we take $s = 0$ and $t = m+1$ in the definition above), with $E^m/E^m_1 \simeq E_{\infty}^{0,m}$ and $E^m_m \simeq E_{\infty}^{m,0}$. 

Given two spectral sequences, there is a natural notion of a morphism between them, which consists of morphisms between the objects in the $E_r$-terms for all $r$, each of which induces its successor on cohomology.  There is also a natural notion of morphisms between bounded spectral sequences with given abutments.

\begin{definition}\label{sshom}
Let $E$ and $E'$ be two spectral sequences in $\mathcal{C}$ with respective differentials $d$ and $d'$.  A \emph{morphism} $u: E \rightarrow E'$ is a family of morphisms $u_r^{p,q}: E_r^{p,q} \rightarrow E_r^{'p,q}$ such that the $u_r^{p,q}$ are compatible with the differentials ($d_r^{'p,q} \circ u_r^{p,q} = u_r^{p+r,q-r+1} \circ d_r^{p,q}$ for all $p,q,r$) and the morphisms 
\[
\overline{u}_r^{p,q}: \ker(d_r^{p,q})/\im(d_r^{p-r,q+r-1}) \rightarrow \ker(d_r^{'p,q})/\im(d_r^{'p-r,q+r-1})
\] 
induced by $u_r^{p,q}$ commute with the given isomorphisms $\alpha_r^{p,q}$ (that is, $\alpha_r^{'p,q} \circ \overline{u}_r^{p,q} = \overline{u}_{r+1}^{p,q} \circ \alpha_r^{p,q}$) so that, in the appropriate sense, $u_{r+1}^{p,q}$ is the morphism induced by $u_r^{p,q}$.  If $E$ and $E'$ are bounded spectral sequences with abutments $\{E^m\}$ and $\{E'^m\}$, a \emph{morphism with abutments} between the spectral sequences is a morphism $u: E \rightarrow E'$ as just defined together with a family of morphisms $u^m: E^m \rightarrow E'^m$ compatible with the filtrations on $E^m$ and $E'^m$ such that, if we denote by $u_{\infty}^{p,q}$ the map induced by $u_1^{p,q}$ (or $u_2^{p,q}$) between the stable objects $E_{\infty}^{p,q}$ and $E_{\infty}^{'p,q}$, this map must commute with the isomorphisms $\beta^{p,q}$: if we denote by $u^m_p$ the morphism $E^m_p/E^m_{p+1} \rightarrow E^{'m}_p/E^{'m}_{p+1}$ induced by $u^m$, which is required to be filtration-compatible, we must have $\beta^{'p,q} \circ u_{\infty}^{p,q} = u^{p+q}_p \circ \beta^{p,q}$.
\end{definition}

\begin{convention}
For the remainder of this paper, every spectral sequence will be a bounded spectral sequence with abutment, and every morphism of spectral sequences will be a morphism with abutments.  Consequently, we suppress the phrase ``with abutment''.
\end{convention}

To show that two spectral sequences are \emph{isomorphic}, it suffices to construct a morphism between them which is an isomorphism on the objects of the initial ($r=1$ or $r=2$) terms.  This result is crucial to our work in both this section and in section \ref{cohom}, so we record a version here:

\begin{prop}\cite[Thm. 5.2.12]{weibel}\label{E1isos}
Let $\mathcal{C}$ be an Abelian category, and let $u = (u_r^{p,q}, u^n)$ be a morphism between two spectral sequences $E$, $E'$ in $\mathcal{C}$.  If there exists $r$ such that $u_r^{p,q}$ is an isomorphism for all $p$ and $q$, then $u_s^{p,q}$ is an isomorphism for all $p$ and $q$ and all $s \geq r$, and $u^m$ is an isomorphism for all $m$.  It follows that the abutments of $E$ and $E'$ are isomorphic as filtered objects.
\end{prop}

There is also a notion of a \emph{degree-shifted} morphism of spectral sequences and a degree-shifted analogue of Proposition \ref{E1isos}, which we will make use of in this paper.  Again, for us, all spectral sequences will be bounded and all morphisms will be morphisms with abutments.

\begin{definition}\label{shiftmorphism}
Let $E$ and $E'$ be two spectral sequences in $\mathcal{C}$ with respective differentials $d$ and $d'$.  If $a,b \in \mathbb{Z}$, a \emph{morphism} $u: E \rightarrow E'$ \emph{of bidegree $(a,b)$} is a family of morphisms $u_r^{p,q}: E_r^{p,q} \rightarrow E_r^{'p+a,q+b}$ such that the $u_r^{p,q}$ are compatible with the differentials ($d_r^{'p+a,q+b} \circ u_r^{p,q} = u_r^{p+r,q-r+1} \circ d_r^{p,q}$ for all $p,q,r$) and $u_{r+1}^{p,q}$ is induced on cohomology by $u_r^{p,q}$.  If $E$ and $E'$ are bounded spectral sequences with abutments $\{E^m\}$ and $\{E'^m\}$, a \emph{morphism with abutments} between the spectral sequences \emph{of bidegree $(a,b)$} is a morphism $u: E \rightarrow E'$ of bidegree $(a,b)$ as just defined together with a family of morphisms $u^m: E^m \rightarrow E'^{m+a+b}$ such that $u^m(E^m_p) \subset E'^{m+a+b}_{p+a}$ for all $p$ and satisfying the obvious compatibility conditions analogous to those in the non-degree-shifted definition.
\end{definition}

The degree-shifted analogue of Proposition \ref{E1isos} is proved in exactly the same way, but in the conclusion (that the abutments are isomorphic as filtered objects), it is worth recording precisely which filtrations are being compared and what the corresponding degree shifts are:

\begin{prop}\label{E1shifted}
Let $\mathcal{C}$ be an Abelian category, and let $u = (u_r^{p,q}, u^n)$ be a morphism of bidegree $(a,b)$ between two spectral sequences $E$, $E'$ in $\mathcal{C}$.  If there exists $r$ such that $u_r^{p,q}$ is an isomorphism for all $p$ and $q$, then $u_s^{p,q}$ is an isomorphism for all $p$ and $q$ and all $s \geq r$, and $u^m$ is an isomorphism for all $m$.  This has the following consequence for the abutments: for all $m$, $E^m$, endowed with the filtration $\{E^m_p\}_{p=0}^{m+1}$ where $E^m_p/E^m_{p+1} \simeq E_{\infty}^{p,m-p}$, is isomorphic (as a filtered object) to $E'^{m+a+b}$, endowed with the filtration $\{E'^{m+a+b}_{p+a}\}_{p=0}^{m+1}$, where $E'^{m+a+b}_{p+a}/E'^{m+a+b}_{p+a+1} \simeq E_{\infty}^{'p+a,m+b-p}$.
\end{prop}

Double complexes are a common source of spectral sequences: the cohomology of the totalization of a double complex can be approximated, and in some cases even computed, by the objects in the early terms of either of two spectral sequences associated with the double complex.  To be precise, let $K^{\bullet, \bullet}$ be a double complex in an Abelian category $\mathcal{C}$, which we think of abusively as the ``$E_0$-term'' of a spectral sequence, and let $T^{\bullet}$ be its totalization.  Our conventions for double complexes are those of EGA: the horizontal ($d_h^{\bullet, \bullet}$) and vertical ($d_v^{\bullet, \bullet}$) differentials of $K^{\bullet, \bullet}$ commute, we define $T^i = \oplus_{p+q=i}K^{p,q}$, and the differentials of $T^{\bullet}$ require signs, namely $d(x) = d_h(x) + (-1)^pd_v(x)$ for $x \in K^{p,q}$.  The two spectral sequences associated with $K^{\bullet, \bullet}$ \cite[\S 11.3]{ega31} are the \emph{column-filtered} (``vertical differentials first'') spectral sequence, for which $E_1^{p,q} = h^q_v(K^{p,\bullet})$ (and the differentials are those induced on vertical cohomology by the maps $d_h^{p,q}$), and the \emph{row-filtered} (``horizontal differentials first'') spectral sequence, for which $E_1^{p,q} = h^p_h(K^{\bullet,q})$ (and the differentials are those induced on horizontal cohomology by the maps $d_v^{p,q}$).  Both have $h^{p+q}(T^{\bullet})$, the cohomology of the totalization, for their abutment.  A morphism $K^{\bullet, \bullet} \rightarrow K^{'\bullet, \bullet}$ of double complexes induces morphisms between their column-filtered spectral sequences as well as between their row-filtered spectral sequences \cite[p. 30]{ega31}.

The spectral sequences of a double complex are useful for computing hyperderived functors of left-exact functors between Abelian categories.  Suppose $\mathcal{A}, \mathcal{B}$ are Abelian categories, $\mathcal{A}$ has enough injective objects, and $F: \mathcal{A} \rightarrow \mathcal{B}$ is a left-exact additive functor.  If $K^{\bullet}$ is a complex with differential $d$ in $\mathcal{A}$, the (right) hyperderived functors of $F$ evaluated at $K^{\bullet}$ are defined as follows \cite[\S 11.4]{ega31}: if $K^{\bullet} \rightarrow I^{\bullet}$ is a quasi-isomorphism and $I^{\bullet}$ is a complex of injective objects in $\mathcal{A}$, then $\mathbf{R}^iF(K^{\bullet}) = h^i(F(I^{\bullet}))$, and the objects of $\mathcal{B}$ thus obtained are independent of the choice of $I^{\bullet}$.  Such a complex $I^{\bullet}$ can be produced as the totalization of a \emph{Cartan-Eilenberg resolution} of $K^{\bullet}$, which is a double complex $J^{\bullet, \bullet}$ with differentials $d_h,d_v$ such that every $J^{p,q}$ is an injective object of $\mathcal{A}$ and, for all $p$, $J^{p,\bullet}$ (resp. $\ker(d_v^{p,\bullet})$, $\im(d_v^{p,\bullet})$, $h^{\bullet}_v(J^{p,\bullet})$) is an injective resolution of $K^p$ (resp. $\ker(d^p)$, $\im(d^p)$, $h^p(K^{\bullet})$).  It follows that $\mathbf{R}^iF(K^{\bullet})$ is the cohomology of the totalization of the double complex $F(J^{\bullet, \bullet})$, and so, by the previous paragraph, we have two spectral sequences whose abutment is this cohomology.  For example, the column-filtered spectral sequence begins $E_1^{p,q} = h^q_v(F(J^{p,\bullet}))$ and has abutment $\mathbf{R}^{p+q}F(K^{\bullet})$.  But since $J^{p,\bullet}$ is an injective resolution of $K^p$, we see that $h^q_v(F(J^{p,\bullet})) = R^qF(K^p)$, the ordinary $q$th right derived functor of $F$ applied to $K^p$; this is the form in which the ``first'' hyperderived functor spectral sequence is usually given \cite[11.4.3.1]{ega31}.  

Now suppose $K^{\bullet}$, $K'^{\bullet}$ are complexes in $\mathcal{A}$ with respective Cartan-Eilenberg resolutions $J^{\bullet, \bullet}$, $J'^{\bullet, \bullet}$.  A morphism of complexes $f: K^{\bullet} \rightarrow K'^{\bullet}$ induces a morphism of double complexes $J^{\bullet, \bullet} \rightarrow J'^{\bullet, \bullet}$ which is unique up to homotopy \cite[p. 33]{ega31}.  This implies that $f$ induces a well-defined morphism between the spectral sequences for the hyperderived functors of $F$ evaluated at $K^{\bullet}$ and at $K'^{\bullet}$ \cite[p. 30]{ega31}, since two double complex morphisms that are chain homotopic induce the same morphisms on horizontal and vertical cohomology, hence the same spectral sequence morphisms.  By taking $K'^{\bullet} = K^{\bullet}$ and $f$ to be the identity, we see that the isomorphism class of the spectral sequence for $F$ evaluated at $K^{\bullet}$ is independent of the Cartan-Eilenberg resolution.

Later in this section, we will need to build a spectral sequence for hyperderived functors using a double complex that is not a Cartan-Eilenberg resolution of the original complex, and for this purpose, the following comparison lemma will be useful.

\begin{lem}\label{sscomp}
Let $\mathcal{A}$ be an Abelian category with enough injective objects, $\mathcal{B}$ another Abelian category, and $F: \mathcal{A} \rightarrow \mathcal{B}$ a left-exact additive functor.  Suppose that $K^{\bullet}$ is a complex in $\mathcal{A}$, concentrated in degrees $p \geq 0$, and that $L^{\bullet, \bullet}$ is a double complex in $\mathcal{A}$ whose objects $L^{p,q}$ are all $F$-acyclic and such that, for all $p \geq 0$, $L^{p,\bullet}$ is a resolution of $K^p$.  Then the first spectral sequence for the hyperderived functors of $F$ applied to $K^{\bullet}$, which begins $E_1^{p,q} = R^qF(K^p)$ and has $\mathbf{R}^{p+q}F(K^{\bullet})$ for its abutment, is isomorphic to the column-filtered spectral sequence of the double complex $F(L^{\bullet, \bullet})$.
\end{lem}
	
\begin{proof}
By definition, the first spectral sequence is the column-filtered spectral sequence of the double complex $F(J^{\bullet, \bullet})$, where $J^{\bullet, \bullet}$ is a choice of Cartan-Eilenberg resolution of $K^{\bullet}$ in $\mathcal{A}$.  The assertion of the lemma is that we can replace $J^{\bullet, \bullet}$ with the resolution $L^{\bullet, \bullet}$, which is generally not a Cartan-Eilenberg resolution and whose objects may not even be injective.  

Our strategy will be to compare both of these double complexes to a third one.  Let $\mathcal{C}^+$ denote the category of complexes in $\mathcal{A}$ that are concentrated in degrees $p \geq 0$.  Then $\mathcal{C}^+$ is an Abelian category with enough injective objects, and if $I^{\bullet} \in \mathcal{C}^+$ is injective, then $I^p$ is an injective object of $\mathcal{A}$ for all $p$ \cite[Thms. 10.42, 10.43; Remark, p. 652]{rotman}.  

We return now to the complex $K^{\bullet}$.  Choose an injective resolution $0 \rightarrow K^{\bullet} \rightarrow I^{\bullet, \bullet}$ of $K^{\bullet}$ in $\mathcal{C}^+$.  In particular, $I^{\bullet, \bullet}$ is a double complex of injective objects in $\mathcal{A}$.  Now note that the two double complex resolutions $J^{\bullet, \bullet}$ and $L^{\bullet, \bullet}$ can also be regarded as resolutions of $K^{\bullet}$ in the category $\mathcal{C}^+$.  Any resolution in $\mathcal{C}^+$ can be compared with an injective one by \cite[Lemma XX.5.2]{lang}: there exist morphisms $J^{\bullet, \bullet} \rightarrow I^{\bullet, \bullet}$ and $L^{\bullet, \bullet} \rightarrow I^{\bullet, \bullet}$ extending the identity on $K^{\bullet}$ and unique up to homotopy as maps in $\mathcal{C}^+$.  These morphisms of double complexes induce morphisms between the column-filtered spectral sequences corresponding to the double complexes after applying the functor $F$.  To finish the proof, by Proposition \ref{E1isos}, it is enough to check that these morphisms of spectral sequences are isomorphisms at the $E_1$-level.  We first consider the morphism $F(J^{\bullet, \bullet}) \rightarrow F(I^{\bullet, \bullet})$.  For all $p$, $J^{p, \bullet} \rightarrow I^{p,\bullet}$ is a morphism between two injective resolutions of $K^p$ extending the identity on $K^p$, which induces an isomorphism $h^q(F(J^{p,\bullet})) \xrightarrow{\sim} h^q(F(I^{p,\bullet}))$, both sides being equal to $R^qF(K^p)$ by definition and being the $E_1^{p,q}$-terms of the respective spectral sequences.  In the case of the morphism $F(L^{\bullet, \bullet}) \rightarrow F(I^{\bullet, \bullet})$, we do not have injective resolutions of $K^p$ (only $F$-acyclic ones) on the left-hand side, but by \cite[Thm. XX.6.2]{lang}, this is enough: the $L^{p,\bullet} \rightarrow I^{p,\bullet}$ also give rise to isomorphisms after applying $F$ and taking cohomology.  We conclude that the three column-filtered spectral sequences corresponding to the double complexes $F(J^{\bullet, \bullet})$, $F(I^{\bullet, \bullet})$, and $F(L^{\bullet, \bullet})$ are isomorphic beginning with their $E_1$-terms, completing the proof.
\end{proof}

We will need one more type of spectral sequence, the \emph{Grothendieck composite-functor spectral sequence}:

\begin{prop}\cite[Thm. 5.8.3]{weibel}\label{compositess}
Let $\mathcal{A}$, $\mathcal{B}$, and $\mathcal{C}$ be Abelian categories, and suppose $\mathcal{A}$ and $\mathcal{B}$ have enough injective objects.  Let $F: \mathcal{A} \rightarrow \mathcal{B}$ and $G: \mathcal{B} \rightarrow \mathcal{C}$ be left-exact additive functors.  Suppose that for every injective object $I$ of $\mathcal{A}$, the object $F(I)$ of $\mathcal{B}$ is acyclic for $G$.  Then for every object $A$ of $\mathcal{A}$, there is a spectral sequence which begins $E_2^{p,q} = (R^pG)((R^qF)(A))$ and abuts to $R^{p+q}(G \circ F)(A)$.
\end{prop}

\begin{example}\label{lccomposite}
For our purposes, the most important example of a composite-functor spectral sequence is the spectral sequence for \emph{iterated local cohomology}.  Let $R$ be a Noetherian ring, and let $I$ and $J$ be ideals of $R$.  If $\mathcal{I}$ is an injective $R$-module, then $\Gamma_J(\mathcal{I})$ is again injective \cite[Lemma III.3.2]{hartshorneAG}, hence acyclic for the functor $\Gamma_I$.  It follows that the left-exact functors $F = \Gamma_J$ and $G = \Gamma_I$ satisfy the conditions of Proposition \ref{compositess}.  Since $R$ is Noetherian, $\Gamma_I \circ \Gamma_J = \Gamma_{I+J}$.  For any $R$-module $M$, the corresponding spectral sequence for the derived (local cohomology) functors begins $E_2^{p,q} = H^p_I(H^q_J(M))$ and abuts to $H^{p+q}_{I+J}(M)$.
\end{example}

\subsection{The proof of Theorem \ref{mainthmA}}\label{thmApf} 

We now recall from \cite{derham} Hartshorne's definition of de Rham homology for a complete local ring.  Let $k$ be a field of characteristic zero and $A$ be a complete local ring with coefficient field $k$.  By Cohen's structure theorem, there exists a surjection of $k$-algebras $\pi: R \rightarrow A$ where $R = k[[x_1, \ldots, x_n]]$ for some $n$.  Let $I \subset R$ be the kernel of this surjection.  We have a corresponding closed immersion $Y \hookrightarrow X$ where $Y = \Spec(A)$ and $X = \Spec(R)$.  The \emph{de Rham homology} of the local scheme $Y$ is defined as $H_i^{dR}(Y) = \mathbf{H}_Y^{2n-i}(X, \Omega^{\bullet}_X)$, the hypercohomology (with support at $Y$) of the \emph{continuous} de Rham complex of sheaves of $k$-spaces on $X$.  Here, the sheaf $\Omega^1_X$ is free of rank $n$ with basis $dx_1, \ldots, dx_n$, and the other sheaves in the complex are its corresponding exterior powers.  In fact, the complex $\Omega_X^{\bullet}$ is the sheafified version of the de Rham complex $\Omega_R^{\bullet}$ of the left $\D(R,k)$-module $R$ as defined in subsection \ref{dmodprelim}.

The de Rham homology spaces defined above are known to be independent of the choice of $R$ and $\pi$ \cite[Prop. III.1.1]{derham} and to be finite-dimensional $k$-spaces \cite[Thm. III.2.1]{derham}.  In this section we give arguments for the embedding-independence and the finiteness which are purely local and provide new information.  Recall from section \ref{intro} that the \emph{Hodge-de Rham} spectral sequence for homology has $E_1$-term given by $E_1^{n-p, n-q} = H^{n-q}_Y(X, \Omega_X^{n-p})$ and abuts to $H_{p+q}^{dR}(Y)$.  (When needed, we will write $\{E^{n-p,n-q}_{r,R}\}$ for this spectral sequence, recording the dependence on the base ring $R$.) The assertion of Theorem \ref{mainthmA} is that, beginning with the $E_2$-term, this spectral sequence consists of finite-dimensional $k$-spaces and its isomorphism class is independent (up to a bidegree shift) of $R$ and $\pi$; this immediately recovers the embedding-independence and finiteness for the abutment $H_{*}^{dR}(Y)$.  To make the line of argument clearer, we give the proof first for the $E_2$-term only (Proposition \ref{E2case}), then explain the additional steps needed to make the basic strategy work for the rest of the spectral sequence.

\begin{lem}\label{E2shape}
Let the surjection $\pi: R = k[[x_1, \ldots, x_n]] \rightarrow A$ (and the associated objects $I, X, Y$) be as above, and $\{E_r^{p,q}\}$ the corresponding Hodge-de Rham spectral sequence for homology.
\begin{enumerate}[(a)]
\item For all $q$, $H^q_Y(X, \mathcal{O}_X) \simeq H^q_I(R)$ as $R$-modules; indeed, if $M$ is any $R$-module and $\mathcal{F}$ the associated quasi-coherent sheaf on $X$, we have $H^q_Y(X, \mathcal{F}) \simeq H^q_I(M)$.
\item For all $p$ and $q$, we have
\[
E_2^{p,q} \simeq H^p_{dR}(H^q_Y(X, \mathcal{O}_X)) \simeq H^p_{dR}(H^q_I(R))
\]
as $k$-spaces, where the $\D$-module structure on $H^q_I(R) \simeq H^q_Y(X, \mathcal{O}_X)$ is defined as in Remark \ref{lcdmod}.
\end{enumerate}
\end{lem}

\begin{proof}
As $X$ is affine, part (a) is well-known (e.g. \cite[Exp. II, Prop. 5]{SGA2}).  Now consider the $E_1$-term of the spectral sequence.  Its differentials are horizontal and so its $E_2$-objects are the cohomology objects of its rows.  Fix such a row, say the $q$th row, which takes the form $E_1^{\bullet,q} = H^q_Y(X, \Omega^{\bullet}_X)$.  The $\Omega^p_X$ are finite free sheaves on $X$ and local cohomology $H^q_Y$ commutes with direct sums, so for all $p$, $E_1^{p,q} \simeq H^q_Y(X, \mathcal{O}_X) \otimes \Omega_R^p$.  As $p$ varies, we obtain the complex $H^q_Y(X, \mathcal{O}_X) \otimes \Omega_R^{\bullet}$, whose $k$-linear maps are the de Rham differentials of the $\D$-module $H^q_Y(X, \mathcal{O}_X)$.  Therefore its cohomology objects, the $E_2$-objects of the spectral sequence, are of the stated form.
\end{proof}

\begin{prop}\label{E2case}
Let the surjection $\pi: R = k[[x_1, \ldots, x_n]] \rightarrow A$ (and the associated objects $I, X, Y$) be as above.
\begin{enumerate}[(a)]
\item For all $p$ and $q$, the $k$-space $E_2^{p,q} = H^p_{dR}(H^q_I(R))$ is finite-dimensional.  
\item Suppose we have another surjection of $k$-algebras $\pi': R' = k[[x_1, \ldots, x_{n'}]] \rightarrow A$ with kernel $I'$.  Write $\{E_{r,R}^{p,q}\}$ (resp. $\{\mathbf{E}_{r,R'}^{p,q}\}$) for the Hodge-de Rham spectral sequence for homology defined using $\pi$ (resp. $\pi'$).  Then for all $p$ and $q$, the $k$-spaces $E_{2,R}^{p,q}$ and $\mathbf{E}_{2,R'}^{p+n'-n,q+n-n'}$ are isomorphic.  (That is, the $E_2$-term is independent of $R$ and $\pi$, up to a bidegree shift.)
\end{enumerate}
\end{prop}

\begin{proof}[Proof of part (a)]
For all $q$, the $\D$-module $H^q_I(R)$ is holonomic \cite[2.2(d)]{lyubeznik}, so its de Rham cohomology spaces are finite-dimensional by Theorem \ref{findimdR}.  This proves part (a).
\end{proof}
  
A proof of part (b) is considerably longer. We first reduce it to Lemma \ref{specialcase} below and then prove Lemma \ref{specialcase}.

Write $X'$ for the spectrum of $R'$.  The surjection $\pi'$ induces a closed immersion $Y \hookrightarrow X'$.  Form the complete tensor product $R'' = R \widehat{\otimes}_k R'$ \cite[V.B.2]{serre}, again a complete regular $k$-algebra, and let $\pi'': R'' \rightarrow A$ be the induced surjection $\pi \widehat{\otimes}_k \pi'$ of $k$-algebras, which gives rise to a third closed immersion $Y \hookrightarrow X'' = \Spec(R'')$ and a third Hodge-de Rham spectral sequence $\{E_{r,R''}^{p,q}\}$.  It suffices to show that both $E_{2,R}^{p,q}$ and $\mathbf{E}_{2,R'}^{p+n'-n,q+n-n'}$ are isomorphic to $E_{2,R''}^{p+n',q+n'}$.  Replacing $R'$ with $R''$ and using symmetry, we reduce to the case in which the two surjections $\pi: R \rightarrow A$ and $\pi': R' \rightarrow A$ satisfy $\pi' = \pi \circ g$ for some surjection $g: R' \rightarrow R$ of $k$-algebras.  Let $I = \ker \pi$, $I' = \ker \pi'$, and $I'' = \ker g$, and suppose the dimensions of $R$ and $R'$ are $n$ and $n'$ respectively.  As $R'/I'' \simeq R$ is regular, $I''$ is generated by $n'-n$ elements that form part of a regular system of parameters for $R'$.  By induction on $n'-n$, we reduce further to the case $n'-n=1$, since we can factor the closed immersion $X \hookrightarrow X'$ into a sequence of codimension-one immersions, and the isomorphisms on $E_2$-terms compose while the bidegree shifts add.  Therefore we assume $\ker g$ is a principal ideal, of the form $(f)$ where $x_1, \ldots, x_n, f$ is a regular system of parameters for $R'$.  By Cohen's structure theorem, the complete regular local $k$-algebra $R'$ takes the form $k[[x_1, \ldots, x_n, z]]$; making a change of variables if necessary, we may assume $f = z$.  By Proposition \ref{dRind}, this change of variables does not affect de Rham cohomology.  Thus $R = k[[x_1, \ldots, x_n]]$, $R' = R[[z]]$, and $g$ is the surjection carrying $z$ to $0$, so that $I' = IR' + (z)$.  We state this special case, to which we have reduced the proposition, in the form of the following lemma. Because of the amount of new notation needed to define them, we have not specified the isomorphisms in the statement of Lemma \ref{specialcase}, only asserted their existence. However, the maps themselves will be needed later, and so we will refer not only to Lemma \ref{specialcase} but also to its proof. 

\begin{lem}\label{specialcase}
Let $R = k[[x_1, \ldots, x_n]]$ and let $I$ be an ideal of $R$.  Let $R' = R[[z]]$ and $I' = IR' + (z)$.  Then for all $p$ and $q$, we have an isomorphism
\[
H^p_{dR}(H^q_I(R)) \simeq H^{p+1}_{dR}(H^{q+1}_{I'}(R'))
\]
of $k$-spaces, where the de Rham cohomology is computed by regarding $H^q_I(R)$ as a $\D(R,k)$-module and $H^{q+1}_{I'}(R')$ as a $\D(R',k)$-module. 
\end{lem}

We first give a definition.

\begin{definition}\label{plusop}
Let $M$ be any $k$-space.  Then $M_+ = \oplus_{i > 0} M \cdot z^{-i}$, a $\D(k[[z]],k)$-module, whose elements are finite sums $\sum_i \frac{m_i}{z^i}$ where $m_i \in M$.  If $M$ is an $R$-module (resp. a $\D(R,k)$-module), then $M_+$ defined in this way is an $R'$-module (resp. a $\D(R',k)$-module), with $\partial_z$-action defined by the quotient rule: $\partial_z(\frac{m}{z^{\alpha}}) = \frac{-\alpha m}{z^{\alpha + 1}}$.
\end{definition}

We will frequently refer to this functor as the ``$+$-operation'' on $R$-modules or $k$-spaces.  In the case of an $R$-module $M$, this definition coincides with the ``key functor'' $G(M) = M \otimes_R R'_z/R'$ of N\'{u}\~{n}ez-Betancourt and Witt \cite[\S 3]{luisemily}.  A special case of one of their results will be useful for us:

\begin{prop}\cite[Lemma 3.9]{luisemily}\label{pluslc}
With $R$, $R'$, $I$, and $I'$ as in the statement of Lemma \ref{specialcase}, we have isomorphisms
\[
(H^p_I(R))_+ \simeq H^{p+1}_{I'}(R')
\]
of $R'$-modules, for all $p$ (these isomorphisms are functorial in $R$).  
\end{prop}

We need one final ingredient before giving the proof of Lemma \ref{specialcase}: a short exact sequence relating the ``full'' de Rham complex $\Omega_{R'}^{\bullet}$ of $R'$ with its ``partial'' de Rham complex $R' \otimes \Omega_R^{\bullet}$ defined using the derivations $\partial_1, \ldots, \partial_n$ but omitting $\partial_z = \frac{\partial}{\partial z}$.

\begin{definition}\label{partialdR}
With $R$ and $R'$ as in the statement of Lemma \ref{specialcase}, we define a short exact sequence of complexes
\[
0 \rightarrow R' \otimes \Omega_R^{\bullet}[-1] \xrightarrow{\iota} \Omega_{R'}^{\bullet} \xrightarrow{\pi} R' \otimes \Omega_R^{\bullet} \rightarrow 0
\]
where the map $\iota$ is simply the wedge product with $dz$, and so its image is precisely the direct sum of those summands of $\Omega_{R'}^{\bullet}$ with a $dz$ wedge factor (thus $\pi$ corresponds to setting $dz = 0$).  The sheaf-theoretic analogue, a short exact sequence of complexes of sheaves on $X'$, is constructed similarly.
\end{definition}

This short exact sequence of complexes gives rise to a long exact sequence of cohomology:
\[
\cdots \rightarrow h^{p-1}(R' \otimes \Omega_R^{\bullet}) \xrightarrow{c^{p-1}} h^p(R' \otimes \Omega_R^{\bullet}[-1]) \rightarrow h^p(\Omega_{R'}^{\bullet}) \rightarrow h^p(R' \otimes \Omega_R^{\bullet}) \xrightarrow{c^p} h^{p+1}(R' \otimes \Omega_R^{\bullet}[-1]) \rightarrow \cdots,
\]
where \emph{c} denotes the connecting homomorphism.  After accounting for the shift of $-1$, we see that $c^p$ is a map from the $\D(k[[z]],k)$-module $h^p(R' \otimes \Omega_R^{\bullet})$ to itself.  We will need precisely to identify the maps $c^p$.

\begin{lem}\label{connhom}
With the notation of the previous paragraph, we have $c^p = (-1)^p \partial_z$ as maps from the $\D(k[[z]],k)$-module $h^p(R' \otimes \Omega_R^{\bullet})$ to itself.  (The same holds if $R'$ is replaced with any $\D(R',k)$-module.) 
\end{lem}

\begin{proof}
We use the explicit construction of the connecting homomorphism given, for example, in \cite[Prop. 6.9]{rotman}.  Denote by $d_{R'}^{\bullet}$ the differentials in the de Rham complex $\Omega_{R'}^{\bullet}$.  Given an element of $h^p(R' \otimes \Omega_R^{\bullet})$, which is the cohomology class of some cocycle $\omega \in R' \otimes \Omega_R^p$, the image under $c^p$ of this class is taken to be the class of the cocycle $(\iota^{p+1})^{-1}(d_{R'}^p((\pi^p)^{-1}(\omega)))$, where the superscript $-1$ means ``choose \emph{any} preimage'': this definition is independent of all choices made.  We proceed to calculate this composite, making convenient choices for the preimages.  

We can write $\omega$ as a sum 
\[
\omega = \sum_{i_1, \ldots, i_p} \rho_{i_1 \cdots i_p} \, dx_{i_1} \wedge \cdots \wedge dx_{i_p}
\]
where all $\rho_{i_1 \cdots i_p} \in R'$.  One choice of preimage $(\pi^p)^{-1}(\omega)$ is $\omega$ itself, since none of its terms contain $dz$ wedge factors and hence all are left fixed by $\pi^p$.  Therefore $d_{R'}^p((\pi^p)^{-1}(\omega)) = d_{R'}^p(\omega)$. Since $\omega$ is a cocycle in $R' \otimes \Omega_R^p$, its image under the $p$th de Rham differential with respect to $dx_1, \ldots, dx_n$ is zero, and so the only terms in the definition of $d_{R'}^p(\omega)$ that survive are those involving $dz$.  That is, we have
\[
d_{R'}^p(\omega) = \sum_{i_1, \ldots, i_p} \partial_z(\rho_{i_1 \cdots i_p}) \, dz \wedge dx_{i_1} \wedge \cdots \wedge dx_{i_p},
\]
which, by rearranging the wedge terms, is equal to
\[
\sum_{i_1, \ldots, i_p} (-1)^p \partial_z(\rho_{i_1 \cdots i_p}) \, dx_{i_1} \wedge \cdots \wedge dx_{i_p} \wedge dz.
\]
Finally, a choice of preimage under $\iota^{p+1}$ (which is simply the map $\wedge \, dz$) for the above sum is
\[
(\iota^{p+1})^{-1}(d_{R'}^p((\pi^p)^{-1}(\omega))) = \sum_{i_1, \ldots, i_p} (-1)^p \partial_z(\rho_{i_1 \cdots i_p}) \, dx_{i_1} \wedge \cdots \wedge dx_{i_p} = (-1)^p \partial_z(\omega),
\]
from which the lemma follows.  (The same calculation works for arbitrary $\D(R',k)$-modules $M$, replacing each $\rho_{i_1 \cdots i_p}$ with an element $m_{i_1 \cdots i_p}$ of $M$.)
\end{proof}

We can now prove Lemma \ref{specialcase}.

\begin{proof}[Proof of Lemma \ref{specialcase}]
The differentials in the complexes of Definition \ref{partialdR} are merely $k$-linear, but in every degree $p$, the short exact sequence
\[
0 \rightarrow R' \otimes \Omega_R^{p-1} \xrightarrow{\iota^p} \Omega^p_{R'} \rightarrow R' \otimes \Omega^p_R \rightarrow 0
\]
is a \emph{split} exact sequence of finite free $R'$-modules.  As local cohomology commutes with direct sums, this sequence remains split exact after applying the functor $H^q_{I'}$ for any $q$:
\[
0 \rightarrow H^q_{I'}(R' \otimes \Omega_R^{p-1}) \xrightarrow{\iota^p_q} H^q_{I'}(\Omega^p_{R'}) \rightarrow H^q_{I'}(R' \otimes \Omega^p_R) \rightarrow 0.
\]
Fixing $q$ but varying $p$, we obtain a short exact sequence of complexes of $k$-spaces
\[
0 \rightarrow H^q_{I'}(R' \otimes \Omega_R^{\bullet}[-1]) \xrightarrow{\iota^{\bullet}_q} H^q_{I'}(\Omega^{\bullet}_{R'}) \rightarrow H^q_{I'}(R' \otimes \Omega^{\bullet}_R) \rightarrow 0
\]
which we can rewrite (replacing $q$ with $q+1$) as
\[
0 \rightarrow H^{q+1}_{I'}(R') \otimes \Omega_R^{\bullet}[-1] \xrightarrow{\iota^{\bullet}_{q+1}} H^{q+1}_{I'}(R') \otimes \Omega_{R'}^{\bullet} \rightarrow H^{q+1}_{I'}(R') \otimes \Omega_R^{\bullet} \rightarrow 0
\]
since $\Omega_R^i$ (resp. $\Omega_{R'}^i$) is a finite free $R$- (resp. $R'$-) module.  This short exact sequence of complexes gives rise to a long exact sequence of cohomology which (accounting for the shift of $-1$) takes the form 
\begin{align*}
\cdots \rightarrow h^p(H^{q+1}_{I'}(R') \otimes \Omega_R^{\bullet}) \xrightarrow{\partial_z} h^p(H^{q+1}_{I'}(R') \otimes \Omega_R^{\bullet}) &\xrightarrow{\overline{\iota}^p_{q+1}} 
h^{p+1}(H^{q+1}_{I'}(R') \otimes \Omega_{R'}^{\bullet})\\ \rightarrow h^{p+1}(H^{q+1}_{I'}(R') \otimes \Omega_R^{\bullet}) &\xrightarrow{\partial_z} h^{p+1}(H^{q+1}_{I'}(R') \otimes \Omega_R^{\bullet}) \rightarrow \cdots,
\end{align*}
where we know by Lemma \ref{connhom} that, up to a sign, the connecting homomorphism is $\partial_z$.  Now by Lemma \ref{pluslc}, we know that $H^{q+1}_{I'}(R') \simeq (H^q_I(R))_+$ as $R'$-modules.  The differentials in the complex $H^{q+1}_{I'}(R') \otimes \Omega_R^{\bullet}$ do not involve $z$ or $dz$, and so the $+$-operation passes to its cohomology, since cohomology commutes with direct sums: we have 
\[
h^p(H^{q+1}_{I'}(R') \otimes \Omega_R^{\bullet}) \simeq (h^p(H^q_I(R) \otimes \Omega_R^{\bullet}))_+
\]
as $k$-spaces for all $p$.  For any $k$-space $M$, the action of $\partial_z$ on $M_+$ is given in Definition \ref{plusop}, and it is clear from this definition (since $\ch(k) = 0$) that $\ker(\partial_z: M_+ \rightarrow M_+) = 0$ and $\coker(\partial_z: M_+ \rightarrow M_+) \simeq M$, the latter corresponding to the $\frac{1}{z}$-component of $M_+$.  Returning to the displayed portion of the long exact sequence (with $M = h^p(H^q_I(R) \otimes \Omega_R^{\bullet})$), the second $\partial_z$ is injective, and so by exactness the unlabeled arrow is the zero map; this implies that $\overline{\iota}^p_{q+1}$ is surjective, inducing an isomorphism between $h^{p+1}(H^{q+1}_{I'}(R') \otimes \Omega_{R'}^{\bullet}) = H^{p+1}_{dR}(H^{q+1}_{I'}(R'))$ and the cokernel of the first $\partial_z$.  Since this cokernel is isomorphic to $h^p(H^q_I(R) \otimes \Omega_R^{\bullet}) = H^p_{dR}(H^q_I(R))$, the proof of Lemma \ref{specialcase}, and hence of Proposition \ref{E2case} (b), is complete.
\end{proof}

Proposition \ref{E2case} gives a new set of invariants for complete local rings in equicharacteristic zero, namely, the (finite) dimensions of the $E_2$-objects, with the bidegree shift taken into account:

\begin{definition}\label{rho}
For all $p,q \geq 0$, let $\rho_{p,q} = \dim_k(H^{n-p}_{dR}(H^{n-q}_I(R)))$.
\end{definition}

By Proposition \ref{E2case}, $\rho_{p,q}$ is finite and depends only on $A$ and a choice of coefficient field $k \subset A$.  We note the similarity of the definition of the invariants $\rho_{p,q}$ to the Lyubeznik numbers $\lambda_{p,q}$ \cite[Thm.-Def. 4.1]{lyubeznik}, although our $\rho_{p,q}$ appear to be well-defined only in the characteristic zero case.  One way to define $\lambda_{p,q}$ is as the dimension of the socle of $H^p_{\mathfrak{m}}(H^{n-q}_I(R))$, where $\mathfrak{m} \subset R$ is the maximal ideal \cite[Lemma 2.2]{socle}.  To define the $\rho_{p,q}$, we use de Rham cohomology instead of iterated local cohomology; furthermore, note the difference in the indices.  

\begin{remark}\label{lyunos}
If $H^{n-q}_I(R)$ is supported only at $\mathfrak{m}$, so that $H^{n-q}_I(R) \simeq E^{\oplus \lambda_{0,q}}$ for some $\lambda_{0,q} \geq 0$ (\cite[Thm. 3.4]{lyubeznik}; here $E$ is the Matlis dualizing module), then the de Rham cohomology of $H^{n-q}_I(R)$ is easy to calculate: we have $\rho_{p,q} = \lambda_{0,q}$ if $p = 0$, and $\rho_{p,q} = 0$ otherwise.  Therefore, in this case, $\rho_{p,q} = \lambda_{p,q}$ for all $p$ and $q$.
\end{remark}

We now prove the full statement of Theorem \ref{mainthmA}(a) (we have already proved part (b) above).  Our goal is to construct a bidegree-shifted morphism between the Hodge-de Rham spectral sequences arising from two surjections $R \rightarrow A$ and $R' \rightarrow A$ of $k$-algebras which, at the level of $E_2$-objects, consists of the isomorphisms of Lemma \ref{specialcase}: by Proposition \ref{E1shifted}, this is enough.  The preliminary reductions given in the paragraph before Lemma \ref{specialcase} remain valid when considering the spectral sequences, so we need only address the case in which $R$, $R'$, $I$, and $I'$ are as in the statement of that lemma.  The basic strategy is the same: we use the short exact sequence \ref{partialdR} of complexes as well as the $+$-operation, which will now be applied to double complex resolutions of the de Rham complexes $\Omega_X^{\bullet}$ and $\Omega_{X'}^{\bullet}$.  Our first task is to construct these; we work first at the level of $R$- (resp. $R'$-) modules and $k$-linear maps, and then sheafify the results.

\begin{lem}\label{gormin}
Let $\mathcal{I}^{\bullet}$ be the minimal injective resolution of $R$ as an $R$-module.  For all $p$, $\mathcal{I}^p$ has a structure of $\D(R,k)$-module, and $R \rightarrow \mathcal{I}^{\bullet}$ is a complex in the category of $\D(R,k)$-modules.
\end{lem}
	
\begin{proof}
We use the Cousin complex, for which \cite[IV.2]{residues} is the original reference.  Recall that the Cousin complex $C^{\bullet}(R)$ of $R$ is constructed recursively in the following way: $C^{-2}(R) = 0$, $C^{-1}(R) = R$, and for $i \geq 0$, $C^i(R) = \oplus (\coker d^{i-2})_{\mathfrak{p}}$, the direct sum extending over all $\mathfrak{p} \in \Spec(R)$ such that $\hgt \mathfrak{p} = i$.  The differentials in the complex are simply the natural localization maps.  It is immediate from the definition of the Cousin complex that it is a complex of $\D(R,k)$-modules, since localizations of $\D(R,k)$-modules are again $\D(R,k)$-modules and natural localization maps are $\D(R,k)$-linear \cite[Example 2.1]{lyubeznik}.  However, since $R$ is a Gorenstein local ring, its minimal injective resolution and its Cousin complex coincide \cite[Thm. 5.4]{sharp}.
\end{proof}

Likewise, if we let $\mathcal{J}^{\bullet}$ be the minimal injective resolution of $R'$ as an $R'$-module, Lemma \ref{gormin} implies that $R' \rightarrow \mathcal{J}^{\bullet}$ is a complex in the category of $\D(R',k)$-modules.  By taking finite direct sums of the resolutions $\mathcal{I}^{\bullet}$ and $\mathcal{J}^{\bullet}$, we construct three double complexes:

\begin{definition}\label{Jbicomplexes}
Let $\mathcal{I}^{\bullet, \bullet}$ be the double complex $\mathcal{I}^{p,q} = \mathcal{I}^q \otimes_R \Omega^p_R$ whose vertical differentials are induced by the differentials in the complex $\mathcal{I}^{\bullet}$ and whose horizontal differentials are those in the de Rham complexes $\mathcal{I}^q \otimes_R \Omega_R^{\bullet}$ of the $\D(R,k)$-modules $\mathcal{I}^q$.  Similarly, let $\mathcal{J}_0^{\bullet, \bullet}$ be the double complex $\mathcal{J}_0^{p,q} = \mathcal{J}^q \otimes_R \Omega^p_R$ and let $\mathcal{J}^{\bullet, \bullet}$ be the double complex $\mathcal{J}^{p,q} = \mathcal{J}^q \otimes_{R'} \Omega^p_{R'}$.
\end{definition}

Note that these double complexes have exact sequences of $R$- (or $R'$-) modules for columns, but merely complexes in the category of $k$-spaces for rows.  In the case of $\mathcal{I}^{\bullet, \bullet}$, the rows are the de Rham complexes of the $\D(R,k)$-modules $\mathcal{I}^q$; in the case of $\mathcal{J}_0^{\bullet, \bullet}$, the rows are the de Rham complexes of the $\mathcal{J}^q$ regarded as $\D(R,k)$-modules; and in the case of $\mathcal{J}^{\bullet, \bullet}$, the rows are the de Rham complexes of the $\mathcal{J}^q$ regarded as $\D(R',k)$-modules.  (We recall again Convention \ref{tensorsubscript}: if we write $\otimes \, \Omega_R^{\bullet}$, the tensor products of objects are being taken over $R$, but if we write $\otimes \, \Omega_{R'}^{\bullet}$, the tensor products of objects are being taken over $R'$.)

Each of these three double complexes can be sheafified.  Consider first the double complex $\mathcal{I}^{\bullet,\bullet}$.  For all $p$ and $q$, let $\widetilde{\mathcal{I}^{p,q}}$ denote the associated quasi-coherent sheaf on $X$. The vertical differentials of $\mathcal{I}^{\bullet, \bullet}$ are $R$-linear, and so induce $\mathcal{O}_X$-linear morphisms between the associated sheaves, and the horizontal differentials induce $k$-linear morphisms on the associated sheaves in the same way that the de Rham complex of $\mathcal{O}_X$ is constructed.  For all $p$ and $q$, $\mathcal{I}^{p,q}$ is an injective $R$-module, and so the sheaf $\widetilde{\mathcal{I}^{p,q}}$ is flasque \cite[Prop. III.3.4]{hartshorneAG}, and hence acyclic for the functor $\Gamma_Y$ on the category of sheaves of $k$-spaces on $X$ \cite[Prop. 1.10]{hartshorneLC}.  Therefore we have a double complex $\widetilde{\mathcal{I}^{\bullet, \bullet}}$ whose objects are all $\Gamma_Y$-acyclic sheaves of $k$-spaces on $X$ and whose columns are acyclic resolutions of the $\Omega_X^p$ (because the associated sheaf functor is exact when applied to complexes of $R$-modules).  In the same way, we sheafify the double complexes $\mathcal{J}_0^{\bullet, \bullet}$ and $\mathcal{J}^{\bullet, \bullet}$, obtaining double complexes $\widetilde{\mathcal{J}_0^{\bullet, \bullet}}$ and $\widetilde{\mathcal{J}^{\bullet, \bullet}}$ of sheaves of $k$-spaces on $X'$ which are $\Gamma_Y$-acyclic.

\begin{definition}\label{threesss}
Let $E_{\bullet, R}^{\bullet, \bullet}$ be the column-filtered spectral sequence associated with the double complex $\Gamma_Y(X, \widetilde{\mathcal{I}^{\bullet, \bullet}})$ of $k$-spaces.  Similarly, let $\mathbf{E}_{\bullet, R'}^{\bullet, \bullet}$ be the column-filtered spectral sequence associated with the double complex $\Gamma_Y(X', \widetilde{\mathcal{J}^{\bullet, \bullet}})$, and let $\mathcal{E}_{\bullet}^{\bullet, \bullet}$ be the column-filtered spectral sequence associated with the double complex $\Gamma_Y(X', \widetilde{\mathcal{J}_0^{\bullet, \bullet}})$.
\end{definition}

By Lemma \ref{sscomp}, we know that $E_{\bullet, R}^{\bullet, \bullet}$ coincides with the Hodge-de Rham spectral sequence for the complex $\Omega_X^{\bullet}$, and that $\mathbf{E}_{\bullet, R'}^{\bullet, \bullet}$ coincides with the Hodge-de Rham spectral sequence for the complex $\Omega_{X'}^{\bullet}$, so there is no ambiguity of notation.  The ``intermediate'' spectral sequence $\mathcal{E}_{\bullet}^{\bullet, \bullet}$, which by Lemma \ref{sscomp} coincides with the hypercohomology spectral sequence for the complex $\mathcal{O}_{X'} \otimes \Omega_X^{\bullet}$, will be used to relate the Hodge-de Rham spectral sequences for $X$ and $X'$ via the $+$-operation.  The first step in this process is the following lemma:

\begin{lem}\label{luisemilyinj}
Let $R$, $R'$, $I$, and $I'$ be as in the statement of Lemma \ref{specialcase}, and let $\mathcal{I}^{\bullet}$ (resp. $\mathcal{J}^{\bullet}$) be the minimal injective resolution of $R$ (resp. $R'$) in the category of $R$-modules (resp. $R'$-modules) as above.  Then for all $q$, we have an isomorphism
\[
(\Gamma_I(\mathcal{I}^{q-1}))_+ \simeq \Gamma_{I'}(\mathcal{J}^q)
\]
of $\D(R',k)$-modules.
\end{lem}

\begin{proof}
Fix $q \geq 0$.  As $R'$ is a Gorenstein local ring, the structure of its minimal injective resolution $\mathcal{J}^{\bullet}$ is well-known \cite[Thm. 18.8]{matsumura}: $\mathcal{J}^q = \oplus_{\hgt \mathfrak{p} = q} E(R'/\mathfrak{p})$, where $E(R'/\mathfrak{p})$ is the $R'$-injective hull of $R'/\mathfrak{p}$.  In particular, $\mathcal{J}^q= 0$ for $q > n+1$ and $\mathcal{J}^{n+1}$ is the Matlis dualizing module $E_{R'}$.  Applying the functor $\Gamma_{I'}$ amounts to discarding those summands corresponding to prime ideals outside the closed subscheme $V(I') \subset X'$ \cite[Ex. 10.1.11]{brodmann}: that is, we have the equality
\[
\Gamma_{I'}(\mathcal{J}^q) = \oplus_{\hgt \mathfrak{p} = q, I' \subset \mathfrak{p}} E_{R'}(R'/\mathfrak{p}),
\]
where again only prime ideals of height $q$ appear in the decomposition, since $\mathcal{J}^q = \oplus_{\hgt \mathfrak{p} = q} E(R'/\mathfrak{p})$.  Since $I' = IR' + (z)$, there is a one-to-one correspondence between prime ideals $\mathfrak{p}$ of $R'$ containing $I'$ and prime ideals $\mathfrak{q}$ of $R = R'/(z)$ containing $I$ (indeed, any such $\mathfrak{p}$ takes the form $\mathfrak{q}R' + (z)$).  If $\hgt \mathfrak{p} = q$, then $\hgt \mathfrak{q} = q - 1$, so we have the decomposition
\[
\Gamma_{I'}(\mathcal{J}^q) = \oplus_{I \subset \mathfrak{q} \in \Spec(R), \hgt \mathfrak{q} = q - 1} E_{R'}(R'/(\mathfrak{q}R' + (z)))
\]
as $R'$-modules.  Note that we have $R'$-module isomorphisms
\[
R'/(\mathfrak{q}R' + (z)) \simeq (R'/(z))/(((\mathfrak{q}R' + (z))/(z)) \simeq R/\mathfrak{q},
\]
where $R/\mathfrak{q}$ is viewed as an $R$-module upon which $z \in R'$ acts trivially.  But by \cite[Prop. 3.11]{luisemily}, the $R'$-module $E_{R'}(R'/(\mathfrak{q}R' + (z)))$ is obtained from the $R$-module $E_R(R/\mathfrak{q})$ by the $+$-operation.  (This identification holds at the level of $\D(R',k)$-modules.) We then have isomorphisms
\[
\Gamma_{I'}(\mathcal{J}^q) \simeq \oplus_{I \subset \mathfrak{q} \in \Spec(R), \hgt \mathfrak{q} = q - 1} (E_R(R/\mathfrak{q}))_+ = (\Gamma_I(\mathcal{I}^{q-1}))_+,
\]
of $\D(R',k)$-modules, where we have again used \cite[Ex. 10.1.11]{brodmann}.
\end{proof}
	
\begin{lem}\label{ssplus}
Let $E_{\bullet, R}^{\bullet, \bullet}$ and $\mathcal{E}_{\bullet}^{\bullet, \bullet}$ be the spectral sequences of Definition \ref{threesss}.  There is an isomorphism
\[
(E_{\bullet, R}^{\bullet, \bullet})_+ \xrightarrow{\sim} \mathcal{E}_{\bullet}^{\bullet, \bullet}[(0,1)],
\]
where the object on the left-hand side is obtained by applying the $+$-operation to all the objects and differentials of the spectral sequence $E_{\bullet, R}^{\bullet, \bullet}$ (this notation means that the morphism of spectral sequences has bidegree $(0,1)$, as in Definition \ref{shiftmorphism}).
\end{lem}

\begin{proof}
We consider the objects of the double complexes giving rise to these spectral sequences.  For all $p$ and $q$, we have
\[ 
\Gamma_Y(X', \widetilde{\mathcal{J}_0^{p,q+1}}) \simeq \Gamma_{I'}(\mathcal{J}_0^{p,q+1}) = \Gamma_{I'}(\mathcal{J}^{q+1} \otimes \Omega_R^p) \simeq \Gamma_{I'}(\mathcal{J}^{q+1}) \otimes \Omega_R^p
\]
and similarly
\[
(\Gamma_Y(X, \widetilde{\mathcal{I}^{p,q}}))_+ \simeq (\Gamma_I(\mathcal{I}^{p,q}))_+ = (\Gamma_I(\mathcal{I}^q \otimes \Omega_R^p))_+ \simeq (\Gamma_I(\mathcal{I}^q) \otimes \Omega_R^p)_+
\]
by Lemma \ref{E2shape} and the fact that $\Gamma_I = H^0_I$ and $\Gamma_{I'} = H^0_{I'}$ commute with direct sums.  By Lemma \ref{luisemilyinj}, we have $\Gamma_{I'}(\mathcal{J}^{q+1}) \simeq (\Gamma_I(\mathcal{I}^q))_+$ for all $q$.  Therefore, for all $p$ and $q$, we have
\[
\Gamma_{I'}(\mathcal{J}^{q+1}) \otimes \Omega_R^p \simeq (\Gamma_I(\mathcal{I}^q))_+ \otimes \Omega_R^p \simeq (\Gamma_I(\mathcal{I}^q) \otimes \Omega_R^p)_+,
\]
so the \emph{objects} of the double complexes are isomorphic with the indicated bidegree shift.  Finally, we observe that the differentials in the complex $\Omega_R^{\bullet}$ do not involve $z$ or $dz$, so the isomorphisms $(\Gamma_I(\mathcal{I}^q))_+ \otimes \Omega_R^p \simeq (\Gamma_I(\mathcal{I}^q) \otimes \Omega_R^p)_+$ commute with the differentials of the double complex and thus assemble to an isomorphism of double complexes.  An isomorphism of double complexes induces an isomorphism between the corresponding column-filtered spectral sequences, and the lemma follows.
\end{proof}

We are now ready to complete the proof of Theorem \ref{mainthmA}.

\begin{proof}[Proof of Theorem \ref{mainthmA}(a)]
As already described, we need only prove the result in the special case of Lemma \ref{specialcase}.  We retain the notation of that lemma.  Consider again the short exact sequence of Definition \ref{partialdR} and its sheafified version
\[
0 \rightarrow \mathcal{O}_{X'} \otimes \Omega^{\bullet}_X[-1] \xrightarrow{\iota} \Omega_{X'}^{\bullet} \rightarrow \mathcal{O}_{X'} \otimes \Omega^{\bullet}_X \rightarrow 0.
\]
As described in subsection \ref{specseq}, the morphism of complexes $\iota$ induces a morphism between the corresponding spectral sequences for hypercohomology supported at $Y$.  These spectral sequences were identified as $\mathcal{E}_{\bullet}^{\bullet,\bullet}[(-1,0)]$ (respectively, $\mathbf{E}_{\bullet, R'}^{\bullet, \bullet}$) in the paragraph following Definition \ref{threesss}.  Accounting for the shift of $-1$, we see that this induced morphism has the following form:
\[
\iota_{\bullet}^{\bullet, \bullet}: \mathcal{E}_{\bullet}^{\bullet,\bullet} \rightarrow \mathbf{E}_{\bullet, R'}^{\bullet, \bullet}[(1,0)].
\]
Identifying first $\mathcal{E}_{\bullet}^{\bullet,\bullet}[(0,1)]$ with $(E_{\bullet, R}^{\bullet, \bullet})_+$ (by Lemma \ref{ssplus}) and then $E_{\bullet, R}^{\bullet, \bullet}$ with the $\frac{1}{z}$-component of $(E_{\bullet, R}^{\bullet, \bullet})_+$, we see that this further induces a morphism
\[ 
\phi_{\bullet}^{\bullet,\bullet}: E_{\bullet,R}^{\bullet, \bullet} \rightarrow \mathbf{E}_{\bullet,R'}^{\bullet, \bullet}[(1,1)],
\]
given in every degree by the inclusion of $E_{r,R}^{\bullet,\bullet}$ as the $\frac{1}{z}$-component of $(E_{r,R}^{\bullet,\bullet})_+ \simeq \mathcal{E}_r^{\bullet, \bullet}[(0,1)]$ followed by $\iota_{\bullet}^{\bullet, \bullet}$.  If $r=2$, the maps $\phi_2^{p,q}$ are precisely the isomorphisms $H^p_{dR}(H^q_I(R)) \xrightarrow{\sim} H^{p+1}_{dR}(H^{q+1}_{I'}(R'))$ appearing in the proof of Lemma \ref{specialcase}, which were induced by the morphism of complexes $\iota$ and the inclusion of $E_{2,R}^{p,q} = H^p_{dR}(H^q_I(R))$ as the $\frac{1}{z}$-component of $(E_{2,R}^{p,q})_+$.  Therefore the morphism $\phi^{\bullet,\bullet}_{\bullet}$ of spectral sequences is an isomorphism at the $E_2$-level.  By Proposition \ref{E1shifted}, it follows that $\phi$ is an isomorphism at all later levels, including the abutments.  The proof is complete.
\end{proof}

\section{Matlis duality and $\Sigma$-continuous maps}\label{klinear}

In this section, we describe formulations of Matlis duality for local rings containing a field $k$ in terms of continuous $k$-linear maps to $k$.  Many of these results are not new; some of them are stated without proof in SGA2 \cite[Exp. IV]{SGA2}. For lack of adequate references for the proofs, we provide their full details here.  We also define the class of $k$-linear maps (the ``$\Sigma$-continuous'' maps) between arbitrary modules over such rings that admit Matlis duals.

Let $(R, \mathfrak{m})$ be a (Noetherian) local ring with coefficient field $k$. Let $E = E(R/\mathfrak{m})$ denote a choice of injective hull of $R/\mathfrak{m} \simeq k$ as an $R$-module.  The Matlis duality functor $D$ is defined by $D(M) = \Hom_R(M,E)$ for all $R$-modules $M$.  Since $E$ is injective, $D$ is an exact, contravariant functor.  (See \cite[\S 18]{matsumura} for a standard treatment of this duality theory, or \cite{matlis} for its original statement by Matlis.)  If $f: M \rightarrow N$ is an $R$-linear homomorphism of $R$-modules, its Matlis dual is the $R$-linear homomorphism $f^*: D(N) \rightarrow D(M)$ defined by pre-composition with $f$: $f^*(\phi) = \phi \circ f$ for $R$-linear maps $\phi: N \rightarrow E$.  Using this definition, it does not make sense \emph{a priori} to speak of the Matlis dual of a map $\delta: M \rightarrow N$ that is not $R$-linear.  However, we will show that a more general class of maps can be dualized.  We will make use of functorial identifications of the Matlis dual of a finite-length (resp. finitely generated) $R$-module with the set of $k$-linear (resp. $k$-linear and $\mathfrak{m}$-adically continuous) maps from the module to $k$; from these identifications, we will see that any $k$-linear map between finite-length $R$-modules, and any $\mathfrak{m}$-adically continuous $k$-linear map between finitely generated $R$-modules, has a Matlis dual.  We will also explain how this theory can be extended to the case of arbitrary modules.

\begin{remark}
In \cite{SGA2} the following results are stated in a slightly more general setting: $(R, \mathfrak{m})$ is a (Noetherian) local ring containing a field $k_0$ such that the residue field $k = R/\mathfrak{m}$ is a finite extension of $k_0$.  Since we need only the case where $k = k_0$, we make this assumption throughout to simplify the discussion.  However, with only minor modifications to the arguments, all of what follows in this section is true at the level of generality of \cite{SGA2}.
\end{remark}

Let $R$ be as above, and let $N$ be any $R$-module (\emph{a fortiori}, $N$ is a $k$-space).  We can define an $R$-module structure on the $k$-space $\Hom_k(N,k)$ as follows.  Given a $k$-linear homomorphism $\lambda: N \rightarrow k$, we define $r \cdot \lambda: N \rightarrow k$ by $(r \cdot \lambda)(n) = \lambda(rn)$, which is again $k$-linear since, if $\alpha \in k$, we have
\[
(r \cdot \lambda)(\alpha n) = \lambda(r(\alpha n)) = \alpha \lambda(rn) = \alpha (r \cdot \lambda)(n)
\] 
by the $k$-linearity of $\lambda$.  (We will use the dot $\cdot$ throughout this section to denote an $R$-action on maps which is defined by multiplication on the input of the map when multiplying the output by $r \in R$ may not make sense.)

The \emph{socle} $\Soc(E) = (0 :_E \mathfrak{m})$ of $E$ is a one-dimensional $k$-space. We fix, once and for all, a $k$-linear projection $E \rightarrow \Soc(E)$ which we identify with a $k$-linear map $\sigma: E \rightarrow k$.

\begin{remark}\label{noncanonical}
The various functorial identifications made throughout this section will depend on the choices of $E$ and $\sigma$ made here. At the end of this section, we will specify the choices of $E$ and $\sigma$ in the case where $R$ is complete and regular that will be used in the rest of the paper.
\end{remark}   

Now define, for any $R$-module $N$, a map $\Phi_N: \Hom_R(N,E) \rightarrow \Hom_k(N,k)$ by $\Phi_N(g) = \sigma \circ g$.  Clearly, if $g$ is $R$-linear (and hence $k$-linear), the composition $\sigma \circ g$ is $k$-linear.  

\begin{lem}
The map $\Phi_N$ defined above is an $R$-linear homomorphism.
\end{lem}  

\begin{proof}
Let $g \in \Hom_R(N,E)$ and $r \in R$ be given.  Then for any $n \in N$, we have
\[
\Phi_N(rg)(n) = \sigma((rg)(n)) = \sigma(rg(n)) = \sigma(g(rn)) = (\sigma \circ g)(rn) = (r \cdot \Phi_N(g))(n),
\] 
so that $\Phi_N(rg) = r \cdot \Phi_N(g)$.
\end{proof}

We list some more elementary properties of the maps $\Phi_N$.  Suppose we have an $R$-linear homomorphism $f: M \rightarrow N$ of $R$-modules.  The Matlis dual of $f$, which is the map $f^*: \Hom_R(N,E) \rightarrow \Hom_R(M,E)$, is clearly $R$-linear.  

\begin{lem}
The map $f^{\vee}: \Hom_k(N,k) \rightarrow \Hom_k(M,k)$ defined by pre-composition with $f$ (i.e., $f^{\vee}(\lambda) = \lambda \circ f$ for a $k$-linear $\lambda: N \rightarrow k$) is $R$-linear.
\end{lem}

\begin{proof}
Let $r \in R$ and $\lambda \in \Hom_k(N,k)$ be given.  Then
\[
f^{\vee}(r \cdot \lambda)(m) = (r \cdot \lambda)(f(m)) = \lambda(rf(m)) = \lambda(f(rm)) = f^{\vee}(\lambda)(rm) = (r \cdot f^{\vee}(\lambda))(m)
\] 
for any $m \in M$, as desired.
\end{proof}  

Moreover, we note that given any $g \in \Hom_R(N,E)$, both $\Phi_M(f^*(g))$ and $f^{\vee}(\Phi_N(g))$ are equal to the composite $\sigma \circ g \circ f: M \rightarrow k$.  Therefore, the diagram below is commutative and all its arrows are $R$-linear maps:
\[
\begin{CD}
\Hom_R(N,E) @>f^*>> \Hom_R(M,E)\\
@VV \Phi_N V           @VV \Phi_M V\\
\Hom_k(N,k) @>f^{\vee}>> \Hom_k(M,k)
\end{CD}
\]

We have now established enough preliminaries to prove the following:

\begin{prop}\label{finlength}
The map $\Phi_N: \Hom_R(N,E) \rightarrow \Hom_k(N,k)$ defined by $\Phi_N(\phi) = \sigma \circ \phi$ is an isomorphism of $R$-modules whenever $N$ is of finite length.
\end{prop}

\begin{proof}
We proceed by induction on the length $\l(N)$ in the category of $R$-modules, remarking that any finite-length $N$ is a $k$-space of dimension $\l(N)$.  The base case, $\l(N) = 1$, is the case $N \simeq k$; here $\Phi_N$ is an isomorphism identifying $\Hom_k(k,k) \simeq k$ with the socle $\Soc(E) \simeq \Hom_R(k,E) = \Hom_R(R/\mathfrak{m}, E)$ of $E$.  Now suppose $\l(N) \geq 2$, in which case there is a short exact sequence $0 \rightarrow k \rightarrow N \rightarrow N' \rightarrow 0$ of $R$-modules where $\l(N') = \l(N) - 1$.  As $E$ is an injective $R$-module, the functor $\Hom_R(-,E)$ is exact.  Moreover, $\Hom_k(-,k)$ is also an exact functor on the category of $k$-spaces (all $k$-spaces are injective objects).  We therefore obtain a commutative diagram with exact rows:
\[
\begin{CD}
0 @>>> \Hom_R(N',E) @>>> \Hom_R(N,E) @>>> \Hom_R(k,E) @>>> 0\\
@.             @VV \Phi_{N'} V                       @VV \Phi_N V                   @VV \Phi_k V                    @.\\
0 @>>> \Hom_k(N',k) @>>> \Hom_k(N,k) @>>> \Hom_k(k,k) @>>> 0
\end{CD}
\]
All maps in this diagram are $R$-linear, and the bottom row is exact as a sequence of $R$-modules, since it is exact as a sequence of $k$-spaces.  The map $\Phi_k$ is an isomorphism by our base case, and $\Phi_{N'}$ is an isomorphism by the induction hypothesis, so $\Phi_N$ is an isomorphism by the five-lemma and the proof is complete.
\end{proof}

We next consider the case of $R$-modules that are not of finite length, for which we need to restrict attention to $\mathfrak{m}$-adically continuous homomorphisms.  We recall the general definition here:

\begin{definition}\label{madic}
Let $R$ be a commutative ring, $I \subset R$ an ideal, and $M, N$ two finitely generated $R$-modules.  The \emph{$I$-adic topology} on $M$ (resp. $N$) is defined by stipulating that $\{I^nM\}$ (resp. $\{I^nN\}$) be a fundamental system of neighborhoods of $0$.  An Abelian group homomorphism $f: M \rightarrow N$ is \emph{$I$-adically continuous} if it is continuous with respect to these topologies on $M$ and $N$: that is, if for all $t$ there exists an $s$ such that $f(I^sM) \subset I^tN$.
\end{definition}

\begin{remark}\label{infmadic}
Note that any $R$-linear map is automatically $I$-adically continuous.  Also note that this definition makes sense for arbitrary $R$-modules, not necessarily finitely generated.  We insist on finite generation here because we will make use of a different notion of continuity in the case of arbitrary modules.
\end{remark}

In the case of the local ring $(R, \mathfrak{m})$, the only fundamental neighborhood of $0 \in k = R/\mathfrak{m}$ in the $\mathfrak{m}$-adic topology on $k$ is $\{0\}$ itself.  Therefore, if $M$ is a finitely generated $R$-module, a $\mathfrak{m}$-adically continuous map $M \rightarrow k$ is one that annihilates $\mathfrak{m}^tM$ for some $t \geq 0$.  Let $\Hom_{cont, k}(M, k)$ be the $k$-space of $k$-linear maps $M \rightarrow k$ that are $\mathfrak{m}$-adically continuous.   

Now note that if $M$ is a finitely generated $R$-module and $\phi: M \rightarrow E$ is an $R$-linear map, $\phi$ annihilates $\mathfrak{m}^t M$ for some $t$.  Indeed, let $m_1, \ldots, m_n$ be generators for $M$ over $R$.  Since $E = E(R/\mathfrak{m})$ is $\mathfrak{m}$-power torsion (every element of $E$ is annihilated by some power of $\mathfrak{m}$ \cite[Thm. 18.4(v)]{matsumura}), there exist $t_i$ for $i = 1, \ldots, n$ such that $\phi(m_i)$ is annihilated by $\mathfrak{m}^{t_i}$; but then, setting $t = \max\{t_1, \ldots, t_n\}$, we have $\phi(\mathfrak{m}^t M) = 0$.  Therefore every such $\phi$ factors through $M/\mathfrak{m}^t M$ for some $t$, that is, $\Hom_R(M,E) = \varinjlim \Hom_R(M/\mathfrak{m}^t M, E)$.  The $R$-module $M/\mathfrak{m}^t M$ is of finite length for all $t$.  Since $\Hom_R(M/\mathfrak{m}^t M, E)$ is isomorphic via $\Phi_{M/\mathfrak{m}^t M}$ to $\Hom_k(M/\mathfrak{m}^t M, k)$, and these isomorphisms form a compatible system as $t$ varies, we deduce the existence of an isomorphism
\[
\Phi_M: \Hom_R(M, E) \xrightarrow{\sim} \varinjlim \Hom_k(M/\mathfrak{m}^t M, k) = \Hom_{cont, k}(M,k),
\]
again defined by $\Phi_M(\phi) = \sigma \circ \phi$.  Our definition of the action of an element $r \in R$ on a $k$-linear map $\lambda: M \rightarrow k$ by pre-composition by multiplication with $r$ preserves the property of $\mathfrak{m}$-adic continuity, so $\Hom_{cont, k}(M, k)$ is indeed an $R$-module.  As in the finite-length case, we see that $\Phi_M$ is functorial in the $R$-module $M$.  Note that if $M$ is of finite length, every $k$-linear map $M \rightarrow k$ is $\mathfrak{m}$-adically continuous, so that the isomorphism $\Phi_M$ just defined coincides with the $\Phi_M$ defined earlier.  We summarize the above discussion in the following

\begin{thm} \label{sga2}
Let $(R, \mathfrak{m})$ be a local ring with coefficient field $k$.  Let $E$ be an $R$-injective hull of $k$, and let $\sigma: E \rightarrow k$ be a fixed $k$-linear projection of $E$ onto its socle.  For every finitely generated $R$-module $M$, post-composition with $\sigma$ defines an isomorphism
\[ 
\Phi_M: \Hom_R(M, E) \xrightarrow{\sim} \Hom_{cont, k}(M,k)
\] 
of $R$-modules, functorial in the $R$-module $M$.  (As a special case, if $M$ is of finite length, $\Hom_R(M,E) \simeq \Hom_k(M,k)$.)
\end{thm}

For the rest of this section, the assumptions on $R$ and $k$ are as in Theorem \ref{sga2}.  A consequence of this theorem is that we can define the Matlis dual of a map between finitely generated (resp. finite-length) $R$-modules as long as the map is $k$-linear and $\mathfrak{m}$-adically continuous (resp. $k$-linear).  The definition is simply pre-composition:

\begin{definition}\label{matlisdual}
Let $\delta: M \rightarrow N$ be a $k$-linear map between finitely generated $R$-modules that is continuous with respect to the $\mathfrak{m}$-adic topologies on $M$ and $N$.  Then pre-composition with $\delta$ is $k$-linear and carries $\mathfrak{m}$-adically continuous $k$-linear maps $\lambda: N \rightarrow k$ to $\mathfrak{m}$-adically continuous $k$-linear maps $\lambda \circ \delta: M \rightarrow k$, so we can define the \emph{Matlis dual} $\delta^*: D(N) \rightarrow D(M)$ to be the composite
\[
\begin{CD}
\Hom_R(N,E) @>> \Phi_N > \Hom_{cont, k}(N,k) @>> \delta^{\vee} > \Hom_{cont, k}(M,k) @>> \Phi_M^{-1} > \Hom_R(M,E).
\end{CD}
\]
\end{definition}

\begin{remark}
If $M$ and $N$ are of finite length, $\delta^*$ can be defined in the same way as above for any $k$-linear $\delta$.  (More generally, we will see below that any $k$-linear map between \emph{Artinian} modules can be dualized.)  
\end{remark}

Our dual construction behaves well with respect to composition:

\begin{prop}\label{functor}
If $M,N,P$ are finitely generated $R$-modules and $\delta: M \rightarrow N$, $\delta': N \rightarrow P$ are $\mathfrak{m}$-adically continuous $k$-linear maps, then $(\delta' \circ \delta)^* = \delta^* \circ \delta'^*$ as maps $D(P) \rightarrow D(M)$.
\end{prop}

\begin{proof}
We calculate using the definition: 
\[
(\delta' \circ \delta)^* = \Phi_M^{-1} \circ (\delta' \circ \delta)^{\vee} \circ \Phi_P = \Phi_M^{-1} \circ \delta^{\vee} \circ \delta'^{\vee} \circ \Phi_P = (\Phi_M^{-1} \circ \delta^{\vee} \circ \Phi_N) \circ (\Phi_N^{-1} \circ \delta'^{\vee} \circ \Phi_P) = \delta^* \circ \delta'^*,
\]
as desired.
\end{proof}

In the case of an $R$-linear map (which is automatically $k$-linear and $\mathfrak{m}$-adically continuous) between finitely generated $R$-modules, our definition of the Matlis dual of this map agrees with the usual one, so our notation is unambiguous and our definition is in fact a generalization of the usual one.  We make this precise in the following nearly tautological lemma:

\begin{lem}\label{agreement}
Let $M$ and $N$ be finitely generated $R$-modules.  If $f: M \rightarrow N$ is an $R$-linear homomorphism, then $f^* = \Phi_M^{-1} \circ f^{\vee} \circ \Phi_N$, so the usual definition of the Matlis dual of $f$ coincides with ours.
\end{lem}

\begin{proof}
Suppose $\phi: N \rightarrow E$ is $R$-linear.  Then $(f^{\vee} \circ \Phi_N)(\phi) = \sigma \circ \phi \circ f$.  As $\Phi_M$ is an isomorphism, $(\Phi_M^{-1} \circ f^{\vee} \circ \Phi_N)(\phi)$ is the unique $R$-linear map $M \rightarrow E$ which gives $\sigma \circ \phi \circ f$ upon post-composition with $\sigma$; by the unicity, it cannot be anything but $\phi \circ f = f^*(\phi)$.  Therefore the left- and right-hand sides of the asserted equality agree upon evaluation at every $\phi \in D(N)$.
\end{proof}

What can be said in the case of arbitrary (not necessarily finitely generated) modules? An arbitrary $R$-module $M$ can be regarded as the filtered direct limit of its finitely generated $R$-submodules $M_{\lambda}$.  The Matlis dual functor (indeed, any contravariant Hom functor) converts direct limits into inverse limits \cite[Prop. 5.26]{rotman}, so $D(M) = \varprojlim D(M_{\lambda})$.  If $\phi: M \rightarrow E$ is $R$-linear (an element of $D(M)$), the restriction of $\phi$ to a fixed $M_{\lambda}$ corresponds by Theorem \ref{sga2} to an $\mathfrak{m}$-adically continuous map $M_{\lambda} \rightarrow k$.  Therefore, we will be able to repeat our earlier constructions in the case of arbitrary modules if we impose some conditions on how the maps we are studying behave under restriction to finitely generated submodules.  We make this precise with the following definition.

\begin{definition}\label{arbdual}
Let $M$ and $N$ be $R$-modules, and let $\delta: M \rightarrow N$ be a $k$-linear map.  We say that $\delta$ is \emph{$\Sigma$-continuous} if for every finitely generated $R$-submodule $M_{\lambda} \subset M$, the $R$-submodule $\langle \delta(M_{\lambda}) \rangle \subset N$ generated by the image of $M_{\lambda}$ under $\delta$ is finitely generated and the restriction $\delta|_{M_{\lambda}}: M_{\lambda} \rightarrow \langle \delta(M_{\lambda}) \rangle$ is $\mathfrak{m}$-adically continuous in the sense of Definition \ref{madic}.  We write $\Hom_k^{\Sigma}(M,N)$ for the set (indeed, $R$-module) of $\Sigma$-continuous $k$-linear maps $M \rightarrow N$, and refer to such maps simply as ``$\Sigma$-continuous'', the $k$-linearity being understood.
\end{definition}

A $\Sigma$-continuous map is not, in general, $\mathfrak{m}$-adically continuous, but is built from $\mathfrak{m}$-adically continuous maps in ``small'' (finitely generated) stages.  (The terminology ``$\Sigma$-continuous'' is meant to reflect this, by analogy with the $\sigma$-finite measure spaces of real analysis and the $\sigma$-compact spaces of topology.)  

\begin{lem}\label{sigmacontfacts} 
Let $M$, $N$, and $P$ be $R$-modules.
\begin{enumerate}[(a)]
\item Every $R$-linear map $\phi: M \rightarrow N$ is $\Sigma$-continuous.
\item If $\delta: M \rightarrow N$ and $\delta': N \rightarrow P$ are $\Sigma$-continuous, then $\delta' \circ \delta: M \rightarrow P$ is $\Sigma$-continuous. In particular, for every $\Sigma$-continuous map $\delta': N \rightarrow k$, the composite $\delta' \circ \delta: M \rightarrow k$ is $\Sigma$-continuous.
\item A $k$-linear map $\delta: M \rightarrow k$ is $\Sigma$-continuous if and only if, for every finitely generated $R$-submodule $M_{\lambda} \subset M$, there exists $t_{\lambda} \geq 0$ such that $\delta(\mathfrak{m}^{t_{\lambda}}M_{\lambda}) = 0$.
\end{enumerate}
\end{lem}

\begin{proof}
Since $R$-linear maps carry finitely generated submodules to finitely generated submodules, (a) is immediate. (b) holds because both parts of the definition of $\Sigma$-continuity are preserved under composition. Finally, since $k$ is a finitely generated $R$-module, the first part of the definition of $\Sigma$-continuity is automatic when the target of the map is $k$, and therefore a $\Sigma$-continuous map $\delta: M \rightarrow k$ is a map whose restrictions to all finitely generated $R$-submodules of $M$ are $\mathfrak{m}$-adically continuous. By the paragraph following Remark \ref{infmadic}, such maps are exactly those described in (c).
\end{proof}

By restricting attention to $\Sigma$-continuous maps to $k$, we obtain a generalization of Theorem \ref{sga2} to the case of an arbitrary module (\textit{cf.} \cite[Exp. IV, Remarque 5.5]{SGA2}):

\begin{thm}\label{arbsga2}
Let $M$ be an $R$-module, and let $\sigma: E \rightarrow k$ be a fixed $k$-linear projection of $E$ onto its socle. There is an isomorphism of $R$-modules 
\[
\Phi_M: D(M) = \Hom_R(M,E) \xrightarrow{\sim} \Hom_k^{\Sigma}(M, k)
\]
defined by post-composition with $\sigma$ and functorial in the $R$-module $M$.
\end{thm}

\begin{proof}
Consider the family $\{M_{\lambda}\}$ of finitely generated $R$-submodules of $M$.  We view $M$ as the filtered direct limit of the $M_{\lambda}$.  As the Matlis dual functor converts direct limits into inverse limits, we see that $D(M) = \varprojlim D(M_{\lambda})$, the transition maps being pre-composition with inclusions $M_{\lambda} \hookrightarrow M_{\lambda'}$ of finitely generated submodules.  For all $M_{\lambda}$, post-composition with $\sigma$ defines an isomorphism $\Phi_{M_{\lambda}}: D(M_{\lambda}) \rightarrow \Hom_{cont, k}(M_{\lambda}, k)$ by Theorem \ref{sga2}.  Since this isomorphism is functorial in $M_{\lambda}$, the $\{\Phi_{M_{\lambda}}\}$ form a compatible system of $R$-linear isomorphisms, inducing an $R$-linear isomorphism
\[
\Phi'_M: D(M) = \varprojlim D(M_{\lambda}) \xrightarrow{\sim} \varprojlim \Hom_{cont, k}(M_{\lambda}, k);
\]
but the right-hand side of this isomorphism can be identified with $\Hom_k^{\Sigma}(M, k)$, essentially by definition.  The natural isomorphism
\[
\theta: \varprojlim \Hom_{cont, k}(M_{\lambda}, k) \rightarrow \Hom_k^{\Sigma}(M, k)
\] 
is defined on a compatible system $\{f_{\lambda} \in \Hom_{cont, k}(M_{\lambda}, k)\}$ by taking $\theta(\{f_{\lambda}\})$ to be the unique $k$-linear map $M \rightarrow k$ whose restriction to every $M_{\lambda}$ is $f_{\lambda}$ (this map is $\Sigma$-continuous by definition), and the composition $\Phi_M = \theta \circ \Phi'_M$ is nothing but post-composition with $\sigma$, completing the proof.
\end{proof}

\begin{definition}\label{dsigma}
Let $M$ be an $R$-module.  The \emph{$\Sigma$-continuous dual} $D^{\Sigma}(M)$ of $M$ is the $R$-module $\Hom_k^{\Sigma}(M,k)$.  By Theorem \ref{arbsga2}, we have $D(M) \simeq D^{\Sigma}(M)$ as $R$-modules.
\end{definition}

\begin{remark}
If $M$ is a finitely generated $R$-module, then clearly
\[
D^{\Sigma}(M) = \Hom_{cont, k}(M, k),
\] 
so in this case the $\Phi_M$ of Theorem \ref{arbsga2} is the same as the $\Phi_M$ of Theorem \ref{sga2}.
\end{remark}

Theorem \ref{arbsga2} allows us to identify the Matlis dual with the \emph{full} $k$-linear dual in the case of an Artinian module:

\begin{cor}\label{artinian}
If $M$ is an $R$-module such that $\Supp(M) = \{\mathfrak{m}\}$ (for instance if $M$ is Artinian), then $D(M) \simeq \Hom_k(M, k)$ as $R$-modules.
\end{cor}

\begin{proof}
Let $M_{\lambda}$ be a finitely generated $R$-submodule of $M$.  The hypothesis on $M$ implies that $M_{\lambda}$ is annihilated by a power of $\mathfrak{m}$ and consequently is of finite length.  Given any $k$-linear map $M \rightarrow k$, its restriction to $M_{\lambda}$ is therefore $\mathfrak{m}$-adically continuous.  We conclude that $D^{\Sigma}(M) = \Hom_k(M, k)$: the corollary now follows from Theorem \ref{arbsga2}.
\end{proof}

We can now extend the definition of the Matlis dual of an $\mathfrak{m}$-adically continuous map between finitely generated $R$-modules to a definition of the Matlis dual of a $\Sigma$-continuous map between arbitrary modules as follows.

\begin{prop}[Proposition-Definition]\label{arbprecomp}
Let $M$ and $N$ be $R$-modules, and let $\delta: M \rightarrow N$ be a $\Sigma$-continuous map.  Define the \textbf{Matlis dual} $\delta^*: D(N) \rightarrow D(M)$ to be the composite
\[
\begin{CD}
D(N) @>> \Phi_N > D^{\Sigma}(N) @>> \delta^{\vee} > D^{\Sigma}(M) @>> \Phi_M^{-1} > D(M),
\end{CD}
\]
where again $\delta^{\vee}$ is pre-composition with $\delta$. This construction satisfies the following properties: given another $R$-module $P$ and a $\Sigma$-continuous map $\delta': N \rightarrow P$, we have $(\delta' \circ \delta)^* = \delta^* \circ \delta'^*$ as $k$-linear maps $D(P) \rightarrow D(M)$, and $\delta^*$ is the usual Matlis dual in the case in which $\delta$ is $R$-linear.
\end{prop}

\begin{proof}
By Lemma \ref{sigmacontfacts}(b), $\delta^{\vee}$ is a well-defined $k$-linear map $D^{\Sigma}(N) \rightarrow D^{\Sigma}(M)$.  The final two assertions follow immediately since both Proposition \ref{functor} and Lemma \ref{agreement} hold for the restriction of $\delta$ to any finitely generated $R$-submodule of $M$.
\end{proof}  

\begin{remark}\label{cofinal}
Suppose $M$, $N$, and $\delta$ are as above, and suppose we are given \emph{cofinal} families $\{M_{\lambda}\}$ (resp. $\{N_{\mu}\}$) of $R$-submodules of $M$ (resp. $N$), meaning that every $R$-submodule of $M$ (resp. $N$) is contained in some $M_{\lambda}$ (resp. some $N_{\mu}$), with the further condition that for all $\lambda$, there exists $\mu$ such that $\delta(M_{\lambda}) \subset N_{\mu}$.  Since $M = \varinjlim M_{\lambda}$ and $N = \varinjlim N_{\mu}$, it follows that $D(M) \simeq \varprojlim D(M_{\lambda})$ and $D(N) \simeq \varprojlim D(N_{\mu})$ as $R$-modules.  If we write $\delta_{\lambda \mu}$ for $\delta|_{M_{\lambda}}: M_{\lambda} \rightarrow N_{\mu}$, then it is clear from the construction of Proposition \ref{arbprecomp} that $\delta^*: D(N) \rightarrow D(M)$ can be identified with $\varprojlim_{\mu} (\delta_{\lambda \mu}^*: D(N_{\mu}) \rightarrow D(M_{\lambda}))$.
\end{remark}

We next show that over a \emph{complete} local ring, the Matlis dual of a $\Sigma$-continuous map is again $\Sigma$-continuous in some important special cases.  At present, we do not know an example of a $\Sigma$-continuous map for which this is false.  (It would be interesting to find necessary and sufficient conditions on $M$ and $N$ for any $\Sigma$-continuous map $M \rightarrow N$ to have $\Sigma$-continuous dual.)

\begin{prop}\label{klindual}
Let $R$ be a \emph{complete} local ring with coefficient field $k$.  Let $M$ and $N$ be $R$-modules, and let $\delta: M \rightarrow N$ be a $k$-linear map.
\begin{enumerate}[(a)]
\item If $M$ is Artinian, then $\delta$ is $\Sigma$-continuous.
\item If $M$ and $N$ are finitely generated, then $\delta$ is $\mathfrak{m}$-adically continuous if and only if it is $\Sigma$-continuous.
\item If $M$ and $N$ are Artinian, then $\delta^*$, which is defined by part (a), is $\Sigma$-continuous.
\item If $N$ is finitely generated and $\delta$ is $\Sigma$-continuous, then $\delta^*$ is $\Sigma$-continuous.
\end{enumerate}
\end{prop}

\begin{proof}
Suppose that $M$ is Artinian.  Any finitely generated $R$-submodule $M_{\lambda} \subset M$ is of finite length and hence is annihilated by $\mathfrak{m}^l$ for some $l$, so in particular $\delta(\mathfrak{m}^l M_{\lambda}) = 0$.  Since $M_{\lambda}$ is of finite length (hence is a finite-dimensional $k$-space) and $\delta$ is $k$-linear, $\delta(M_{\lambda})$ is also a finite-dimensional $k$-space, so the $R$-submodule of $N$ that it generates is finitely generated over $R$.  We conclude that $\delta$ is $\Sigma$-continuous, proving (a).

If $M$ and $N$ are finitely generated, the image under $\delta$ of any submodule of $M$ is contained in a finitely generated $R$-module, namely $N$ itself.  This proves the forward direction of (b), and the converse is obvious.  

Suppose that $M$ and $N$ are Artinian.  Let $t$ be a natural number, and let $M_t = (0 :_M \mathfrak{m}^t)$, a finite-length $R$-submodule of $M$.  As $N$ is Artinian, every element of $N$ is annihilated by some power of $\mathfrak{m}$, so there exists some $s$ such that $\delta(M_t) \subset N_s = (0 :_N \mathfrak{m}^s)$.  Now apply Matlis duality.  The containment $\delta(M_t) \subset N_s$ implies that the kernel of the map $D(N) \rightarrow D(N_s)$ is carried by $\delta^*$ into the kernel of $D(M) \rightarrow D(M_t)$.  This kernel (which we may identify with $D(M/M_t)$ by the exactness of $D$) is $\mathfrak{m}^t D(M)$; this means that $\delta^*(\mathfrak{m}^s D(N)) \subset \mathfrak{m}^t D(M)$, that is, that $\delta^*$ is continuous with respect to the $\mathfrak{m}$-adic topologies on $D(N)$ and $D(M)$.  Since $M$ and $N$ are Artinian and $R$ is complete, $D(N)$ and $D(M)$ are finitely generated \cite[Thm. 18.6(v)]{matsumura}, so by part (b), $\delta^*$ is $\Sigma$-continuous.  This proves (c).

If $N$ is finitely generated, then since $R$ is complete, $D(N)$ is Artinian \cite[Thm. 18.6(v)]{matsumura}.  The map $\delta^*: D(N) \rightarrow D(M)$ is $k$-linear, hence $\Sigma$-continuous by part (a).  This completes the proof of (d) and the proposition.
\end{proof}

If $M$ is any $R$-module, there is a natural map $\iota: M \rightarrow D(D(M))$ defined by evaluation: we set $\iota(m)(\phi) = \phi(m)$ for all $m \in M$ and $\phi \in D(M)$.  If $R$ is complete and $M$ is a finitely generated or Artinian $R$-module, $\iota$ is an isomorphism by the main theorem of classical Matlis duality \cite[Thm. 18.6(v)]{matsumura}; in fact, $\iota$ is injective even for arbitrary $M$ \cite[Thm. 18.6(i)]{matsumura}.  Moreover, if $f: M \rightarrow N$ is an $R$-linear homomorphism between finitely generated or Artinian $R$-modules, $f^{**}: D(D(M)) \rightarrow D(D(N))$ can be naturally identified with $f$ via the evaluation isomorphisms.  Having verified that $\delta^{**}$ is well-defined whenever $\delta$ is a $\Sigma$-continuous map between Artinian or finitely generated $R$-modules (Proposition \ref{klindual}(c,d)), we now show that the same is true of the double Matlis dual of a $\Sigma$-continuous map between such modules.  

\begin{prop}\label{sigmabij}
Let $R$ be a complete local ring with coefficient field $k$.
\begin{enumerate}[(a)]
\item Let $M$ be a finitely generated or Artinian $R$-module.  The evaluation map
\[
\iota': M \rightarrow D^{\Sigma}(D^{\Sigma}(M)),
\]
analogous to $\iota: M \rightarrow D(D(M))$ and defined by $\iota'(m)(\delta) = \delta(m)$ for all $m \in M$ and $\delta \in D^{\Sigma}(M)$, is an isomorphism of $R$-modules.
\item If $M$ and $N$ are $R$-modules that are either both finitely generated or both Artinian, and $\delta: M \rightarrow N$ is a $\Sigma$-continuous map, then 
\[
\delta^{\vee \vee}: D^{\Sigma}(D^{\Sigma}(M)) \rightarrow D^{\Sigma}(D^{\Sigma}(N))
\]
can be naturally identified with $\delta$ via the evaluation isomorphisms of part (a); consequently, $\delta^{**}: D(D(M)) \rightarrow D(D(N))$ can be identified with $\delta$ via the isomorphisms of classical Matlis duality and Theorem \ref{arbsga2}.  (See Proposition \ref{arbprecomp} for the definitions of $\delta^*$ and $\delta^{\vee}$.)
\item With the hypotheses of part (b), the Matlis dual operation defines an isomorphism
\[
\Hom^{\Sigma}_k(M,N) \simeq \Hom^{\Sigma}_k(D(N),D(M))
\]
of $k$-spaces.  (In particular, every $\Sigma$-continuous map $D(N) \rightarrow D(M)$ is the Matlis dual of a $\Sigma$-continuous map $M \rightarrow N$.)
\end{enumerate}
\end{prop}

\begin{proof}
If $M$ is any $R$-module, we obtain an isomorphism
\[
\Psi: D(D(M)) \xrightarrow{\sim} D^{\Sigma}(D^{\Sigma}(M))
\]
of $R$-modules by applying Theorem \ref{arbsga2} twice.  This map is defined by the formula 
\[
\Psi(\psi)(\delta) = \sigma(\psi(\Phi_M^{-1}(\delta)))
\]
for $\psi \in D(D(M))$ and $\delta \in D^{\Sigma}(M)$.  We claim that the evaluation maps $\iota$ and $\iota'$ satisfy the identity $\Psi \circ \iota = \iota'$.  If we set $\psi = \iota(m)$ for some $m \in M$, we find
\[
\Psi(\iota(m))(\delta) = \sigma(\Phi_M^{-1}(\delta)(m)) = \delta(m),
\]
so that $\Psi \circ \iota$ agrees with $\iota'$ upon evaluation at every $\delta \in D^{\Sigma}(M)$, as claimed.  If $M$ is finitely generated or Artinian, $\iota$ is an isomorphism by classical Matlis duality, so $\iota'$ is the composite of two $R$-module isomorphisms, proving (a).

Suppose now that $M$ and $N$ are finitely generated $R$-modules and $\delta: M \rightarrow N$ is a $\Sigma$-continuous map (the proof of (b) in the Artinian case is the same).  We identify $M$ with $D^{\Sigma}(D^{\Sigma}(M))$ via the evaluation map $\iota'$, and likewise with $N$.  Consider the double dual
\[
\delta^{\vee \vee}: D^{\Sigma}(D^{\Sigma}(M)) \rightarrow D^{\Sigma}(D^{\Sigma}(N)).
\]
Every element of the source takes the form $\iota'(m)$ for some $m \in M$, and by definition we have $\delta^{\vee \vee}(\iota'(m)) = \iota'(\delta(m))$; therefore $\delta^{\vee \vee}$ coincides with $\delta$ upon identifying $M$ and $N$ with their double $\Sigma$-continuous duals.  The second statement of part (b) follows from the first by Proposition \ref{arbprecomp}.

In particular, with the hypotheses of part (b), it is clear that the map $\delta \mapsto \delta^*$ is a bijection, which already implies the parenthetical remark in part (c).  It remains to show that the map $\delta \mapsto \delta^*$ is $k$-linear.  Let $\lambda \in k$ be given.  Then we have
\[
(\lambda \cdot \delta)^* = (\delta \circ \lambda)^* = \lambda^* \circ \delta^* = \lambda \circ \delta^* = \delta^* \circ \lambda = \lambda \cdot \delta^*,
\]
where we have used Proposition \ref{functor} and Lemma \ref{agreement}, as well as the $k$-linearity of $\delta^*$ (since $\delta^*$ need not be $R$-linear, it is \emph{not} true that the Matlis dual operation defines an isomorphism of $R$-modules).  This completes the proof.
\end{proof}

Finally, consider the special case where $R = k[[x_1, \ldots, x_n]]$ (for $k$ a field) and $\mathfrak{m}$ is the maximal ideal of $R$.  Since $R$ is Gorenstein, the local cohomology module $H^n_{\mathfrak{m}}(R)$ is isomorphic to $E$ \cite[Lemma 11.2.3]{brodmann}.  We pick one such isomorphism and think of $E$ as $H^n_{\mathfrak{m}}(R)$ via this isomorphism.  If we compute $H^n_{\mathfrak{m}}(R)$ using the \v{C}ech complex of $R$ with respect to $x_1, \ldots, x_n$ \cite[Thm. 5.1.20]{brodmann}, the resulting $R$-module consists of all $k$-linear combinations of ``inverse monomials'' $x_1^{s_1} \cdots x_n^{s_n}$ where $s_1, \ldots, s_n < 0$.  The $R$-action is defined by the usual exponential rules with the caveat that non-negative powers of the variables are set equal to zero, so the product of a formal power series in $R$ with such an ``inverse polynomial'' has only finitely many nonzero terms.  We may define a $k$-linear projection $\sigma: E \rightarrow k$ of $E$ onto its socle by the formula
\[
\sigma(\sum \alpha_{s_1, \ldots, s_n} x_1^{s_1} \cdots x_n^{s_n}) = \alpha_{-1, \ldots, -1} \in k,
\]  
which we think of as ``taking the $(-1, \ldots, -1)$-coefficient'' of such an ``inverse polynomial'' \cite[Ch. 5]{kunz}.

The map $\sigma$ can be thought of as a realization of the \emph{residue map} in coordinates $\{x_i\}$, and we will sometimes abusively refer to $\sigma$ as ``the residue map'', the chosen coordinates being understood.  The residue map is a canonical map $H^n_{\mathfrak{m}}(\Omega^n_R) \rightarrow k$.  The original reference for the theory of the residue map is \cite{residues}: see also \cite[Exp. IV, Remarque 5.5]{SGA2} and \cite[Ch. 5]{kunz} for concrete descriptions in our special case.  A choice of a regular system of parameters $x_1, \ldots, x_n$ for $R$ induces a non-canonical $R$-module isomorphism $H^n_{\mathfrak{m}}(R) \simeq H^n_{\mathfrak{m}}(\Omega_R^n)$ (defined by $\eta \mapsto \eta \, dx_1 \wedge \cdots \wedge dx_n$), and $\sigma$ is the composite of this isomorphism with the residue map.  One could instead take $E = H^n_{\mathfrak{m}}(\Omega_R^n)$ and the canonical residue map for $\sigma$, but we will have to use coordinates in examples later in the paper. 

\section{Matlis duality for $\D$-modules}\label{dmod}

We now specialize to the case where $(R, \mathfrak{m})$ is a \emph{complete} local ring with coefficient field $k$.  This assumption is in force throughout the section unless otherwise indicated.  

We show that if $\D = \D(R,k)$ is the non-commutative ring of $k$-linear differential operators on $R$ and $M$ is a left $\D$-module, then elements of $\D$ act on $M$ via $\Sigma$-continuous maps.  This fact allows us immediately to apply the formalism of the previous section to define the Matlis dual of the action of an element of $\D$; in this way, $D(M)$ becomes a \emph{right} $\D$-module.  In the case of a complete \emph{regular} local ring with coefficient field $k$ \emph{of characteristic zero}, we describe a ``transposition'' operation allowing us to regard these Matlis duals as \emph{left} $\D$-modules.

Let $\D = \D(R,k)$ be the ring of $k$-linear differential operators on $R$.  Suppose $M$ is a left $\D$-module, and let $\delta \in \D(R,k)$ be a differential operator.  The map $\delta: M \rightarrow M$ defined by the action of $\delta$ (that is, $\delta(m) = \delta \cdot m$) is a $k$-linear map (we abusively use the same letter $\delta$ to simplify notation).

We need a lemma on the continuity of differential operators' action:

\begin{lem}\label{diffcont}
Let $M$ be a left $\D(R,k)$-module (here $R$ is any commutative ring, and $k \subset R$ any commutative subring), and let $I \subset R$ be any ideal.  For any $\delta \in \D^j(R)$ and $s \geq 0$, we have $\delta(I^{s+j}M) \subset I^sM$.
\end{lem}

\begin{proof}
We proceed by induction on $s + j$.  If $s+j=0$, there is nothing to prove.  Now suppose $s + j \geq 0$ and the containment established for smaller values of $s + j$.  Since $\delta \in \D^j(R)$, it follows that $[\delta, r] \in \D^{j-1}(R)$ for any $r \in R$, and so $[\delta, r](I^{s+j-1}M) \subset I^s M$ by the induction hypothesis for $s$ and $j-1$.  That is, $[\delta,r](tm) = \delta(rtm) - r\delta(tm) \in I^s M$ for any $r \in R$, $t \in I^{s+j-1}$ and $m \in M$.  If we further suppose that $r \in I$ and put $x = rt \in I^{s+j-1}I = I^{s+j}$, we see that $\delta(xm) = r\delta(tm) + [\delta,r](tm)$.  By the induction hypothesis for $s-1$ and $j$, $\delta(I^{s+j-1}M) \subset I^{s-1}M$, so $\delta(tm) \in I^{s-1}M$ and $r \delta(tm) \in I^s M$.  Thus both terms on the right-hand side (and their sum $\delta(xm)$) belong to $I^s M$.  As any element of $I^{s+j}M$ can be expressed as a finite sum $\sum_{\alpha} r_{\alpha} t_{\alpha} m_{\alpha}$ with $r_{\alpha} \in I$ and $t_{\alpha} \in I^{s+j-1}$, the result follows.
\end{proof}

If a left $\D$-module $M$ is finitely generated as an $R$-module, then given any $\delta \in \D$, the lemma shows that the corresponding map $\delta: M \rightarrow M$ is $\mathfrak{m}$-adically continuous.  Therefore the formalism of the previous section applies, and we can define the $k$-linear Matlis dual $\delta^*: D(M) \rightarrow D(M)$ in such a way that if $\delta' \in \D$ is another differential operator, we have $(\delta' \circ \delta)^* = \delta^* \circ (\delta')^*$ (note that, by Proposition \ref{functor}, the order of composition is reversed).  This allows us to define a structure of right $\D$-module on $D(M)$ by $\phi \cdot \delta = \delta^*(\phi)$.  Here is an example of this dual construction:

\begin{example}\label{singlevar}
Let $R = k[[x]]$, where $k$ is a field, equipped with the $k$-linear derivation $\delta = \frac{d}{dx}$.  Denote by $\mathfrak{m}$ its maximal ideal $(x)$.  Take $M = R$, a finitely generated $R$-module, and take for $E$ the local cohomology module $H^1_{(x)}(R)$, which is the module of finite sums $\sum_{s > 0} \frac{\alpha_s}{x^s}$ where $\alpha_s \in k$. We define Matlis duals of $k$-linear maps between $R$-modules using the \emph{residue} map $\sigma: E \rightarrow k$ given by $\sigma(\sum_{s > 0} \frac{\alpha_s}{x^s}) = \alpha_1$. 

We determine explicitly the Matlis dual $\delta^*: D(R) \rightarrow D(R)$.  Note that $E$ is naturally isomorphic as $R$-module to $D(R) = \Hom_R(R,E)$: an element $\mu \in E$ corresponds to the map $r \mapsto r\mu$ (and, in the other direction, a map $R \rightarrow E$ corresponds to the image of $1$ in $E$).  Let $\mu = \sum_{s > 0} \frac{\alpha_s}{x^s}$ be such an element of $E$ and suppose that $t$ is the greatest integer such that $\alpha_t \neq 0$.  The $R$-linear map $R \rightarrow E$ corresponding to $\mu$ is defined as follows: given an element $\sum_{i=0}^{\infty} \beta_i x^i$ of $R$, its image in $E$ is the product $(\sum_{i=0}^{\infty} \beta_i x^i)(\sum_{s > 0} \frac{\alpha_s}{x^s})$, which is carried by $\sigma$ to $\sum_{s>0} \alpha_s \beta_{s-1}$.  Therefore the corresponding $k$-linear map $\sigma \circ \mu: R \rightarrow k$ is given by $\sum_{i=0}^{\infty} \beta_i x^i \mapsto \sum_{s>0} \alpha_s \beta_{s-1}$.  The map $\mu$ annihilates $\mathfrak{m}^t$; write $\overline{\mu}$ for the $R$-linear map $R/\mathfrak{m}^t \rightarrow E$ induced on the quotient. 

The derivation $\delta$ induces a $k$-linear map $R/\mathfrak{m}^{t+1} \rightarrow R/\mathfrak{m}^t$ by the Leibniz rule, and the $k$-linear composite $R/\mathfrak{m}^{t+1} \rightarrow R/\mathfrak{m}^t \xrightarrow{\sigma \circ \overline{\mu}} k$ is defined by 
\[
\sum_{i=0}^t \lambda_i x^i \mapsto \sum_{i=0}^{t-1} (i+1)\lambda_{i+1}x^i \mapsto \sum_{s>0} s \alpha_s \lambda_s
\] 
where the scalars $\alpha_s$ are the numerators in the expansion of $\mu$.  Since $R/\mathfrak{m}^{t+1}$ is of finite length as an $R$-module, this composite corresponds to a unique $R$-linear map $R/\mathfrak{m}^{t+1} \rightarrow E$, that is, a map $R \rightarrow E$ defined by multiplication by an element of $E$ that is annihilated by $\mathfrak{m}^{t+1}$.  This map $R \rightarrow E$ is the image of $\mu$ under $\delta^*$, and examining the formula above, we see that the corresponding element of $E$ (the image of $1$) cannot be anything but $\sum_{s>0} \frac{s \alpha_s}{x^{s+1}}$.  Therefore, viewed at the level of a map $E \rightarrow E$, the map $\delta^*$ is defined by
\[
\delta^*(\frac{1}{x^s}) = \frac{s}{x^{s+1}},
\] 
which differs from the standard (``quotient rule'') differentiation map on $E$ only by a minus sign.  (Our discussion of ``transposition'' at the end of this section will indicate the reason for this sign.)  This concludes Example \ref{singlevar}.
\end{example}

In order to extend this definition of Matlis dual to arbitrary $\D$-modules, removing the finite generation hypothesis, we need to show that differential operators act on $\D$-modules via $\Sigma$-continuous maps.  We prove this now using the alternate characterization of differential operators given in EGA \cite{EGAIV}.

\begin{definition}
Suppose that $R$ is any commutative ring and $k \subset R$ a commutative subring.  We denote by $B$ the ring $R \otimes_k R$, by $\mu: B \rightarrow R$ the multiplication map $\mu(r \otimes s) = rs$, and by $J \subset B$ the kernel of $\mu$.  We identify the subring $R \otimes_k 1 = \{r \otimes 1|r \in R\}$ of $B$ with $R$, and view $B$ as an $R$-algebra using this identification.  In this way, $B$ and all ideals of $B$ can be regarded as $R$-modules.  Finally, for any $j \geq 0$, we denote by $P^j_{R/k}$ (or $P^j$) the quotient $B/J^{j+1}$.
\end{definition}

In \cite{EGAIV}, differential operators are described in terms of $R$-linear maps via the following correspondence:

\begin{prop}\cite[Prop. 16.8.4]{EGAIV}\label{egadiff}
For any commutative ring $R$ and commutative subring $k \subset R$, there is an isomorphism
\[
\Hom_R(P^j, R) \simeq \D^j(R)
\] 
of $R$-modules, where the differential operators on the right-hand side are understood to be $k$-linear.
\end{prop}

In our case, where $R$ is a complete local ring and $k$ a coefficient field, both sides of the isomorphism of Proposition \ref{egadiff} are finitely generated $R$-modules:

\begin{prop}\label{fgdiff}
Let $(R, \mathfrak{m})$ be a (Noetherian) complete local ring with coefficient field $k$.  For all $j$, the $R$-module $\D^j(R) \subset \D(R,k)$ is finitely generated.
\end{prop}

The following definition will be used in the proof of Proposition \ref{fgdiff}:

\begin{definition}\label{sepquotient}
Let $(R, \mathfrak{m})$ be a local ring.  For any $R$-module $L$, the \emph{maximal $\mathfrak{m}$-adically separated quotient} of $L$ is $L^{sep} = L/(\cap_s \mathfrak{m}^s L)$.  Note that $L^{sep}$ is $\mathfrak{m}$-adically separated, and $L/\mathfrak{m}L \simeq L^{sep}/\mathfrak{m}L^{sep}$.
\end{definition}

\begin{proof}[Proof of Proposition \ref{fgdiff}]
Let $j$ be given.  By Proposition \ref{egadiff}, it suffices to show that $\Hom_R(P^j, R)$ is a finitely generated $R$-module.  As $R$ is $\mathfrak{m}$-adically separated, every $f \in \Hom_R(P^j, R)$ factors uniquely through $(P^j)^{sep}$.  Therefore $\Hom_R((P^j)^{sep}, R) \simeq \Hom_R(P^j, R)$ as $R$-modules.  We claim that $(P^j)^{sep}$ is itself a finitely generated $R$-module, from which it will follow at once that 
\[
\Hom_R((P^j)^{sep}, R) \simeq \Hom_R(P^j, R) \simeq \D^j(R)
\]
is finitely generated as well.  We have
\[
(P^j)^{sep}/\mathfrak{m}(P^j)^{sep} \simeq P^j/\mathfrak{m}P^j = B/(\mathfrak{m}B + J^{j+1})
\]
as $k$-spaces.  Since $(P^j)^{sep}$ is $\mathfrak{m}$-adically separated and $R$ is $\mathfrak{m}$-adically complete, $(P^j)^{sep}$ is a finitely generated $R$-module if $(P^j)^{sep}/\mathfrak{m}(P^j)^{sep}$ is a finite-dimensional $k$-space, by a form of Nakayama's lemma \cite[Ex. 7.2]{eisenbud}.  Therefore we have reduced the proof of the proposition to the proof that for all $j$, the $k$-space $B/(\mathfrak{m}B + J^{j+1})$ is finite-dimensional, which we give now.

The elements $r \otimes 1 - 1 \otimes r$ with $r \in R$ generate $J$ as an $R$-module: given $b = \sum_i r'_i \otimes r_i \in B$, we have $b \in J$ if and only if $\sum_i r'_ir_i = 0$, and in this case we see that  
\[
b = \sum_i(r'_i \otimes r_i - r'_ir_i \otimes 1) = -\sum_i (r'_i \otimes 1)(r_i \otimes 1 - 1 \otimes r_i) = \sum_i (-r'_i) \cdot (r_i \otimes 1 - 1 \otimes r_i) 
\]
is an $R$-linear combination of elements of the form described.  The equation 
\[
r \otimes r' = rr' \otimes 1 - r \cdot (r' \otimes 1 - 1 \otimes r')
\]
for $r, r' \in R$ also shows that we have an $R$-module direct sum decomposition $B = (R \otimes_k 1) \oplus J$.  Since $R$ is Noetherian, we can fix a finite set of generators $x_1, \ldots, x_s$ for $\mathfrak{m}$.  Moreover, since $R$ contains its residue field $k$, we have a direct sum decomposition (as $k$-spaces) $R = k \oplus \mathfrak{m}$, so given any $r \in R$, we can write $r = c + x_1y_1 + \cdots + x_sy_s$ where $c \in k$ and $y_i \in R$.  We then have $r \otimes 1 - 1 \otimes r = \sum_i (x_iy_i \otimes 1 - 1 \otimes x_iy_i)$, since $c \otimes 1 = 1 \otimes c$ for $c \in k$.  For all $i$, we can express $x_iy_i \otimes 1 - 1 \otimes x_iy_i$ as
\[
(y_i \otimes 1)(x_i \otimes 1 - 1 \otimes x_i) - (x_i \otimes 1 - 1 \otimes x_i)(y_i \otimes 1 - 1 \otimes y_i) + (x_i \otimes 1)(y_i \otimes 1 - 1 \otimes y_i)
\]
where the second summand belongs to $J^2$ and the third summand belongs to $\mathfrak{m}J$.  We conclude that $r \otimes 1 - 1 \otimes r - (\sum y_i \cdot (x_i \otimes 1 - 1 \otimes x_i)) \in J^2 + \mathfrak{m}J$.  Therefore, if we write $b_i = x_i \otimes 1 - 1 \otimes x_i$, the classes of the $b_i$ generate $J/(J^2 + \mathfrak{m}J)$ as an $R$-module.  Since $\mathfrak{m}$ annihilates $J/(J^2 + \mathfrak{m}J)$, and $R = k \oplus \mathfrak{m}$ as $k$-spaces, we see that moreover the classes of the $b_i$ span $J/(J^2 + \mathfrak{m}J)$ as a $k$-space.  Let $L_j$ be the $k$-span of the monomials of degree at most $j$ in the $b_i$, and let $L'_j$ be the $k$-span of such monomials of degree precisely $j$, so that $L_j = \cup_{\l=0}^j L'_{\l}$: clearly all $L_j$ and $L'_j$ are finite-dimensional $k$-spaces.  With this notation, what we have just shown is that $J = L'_1 + J^2 + \mathfrak{m}J$.

Now let $b \in B$ be arbitrary.  Using the $R$-module direct sum decomposition $B = (R \otimes_k 1) \oplus J$, we write $b = (r \otimes 1) + x$ where $x \in J$.  Using the $k$-space direct sum decomposition $R = k \oplus \mathfrak{m}$, we write $r = c + y$ where $c \in k$ and $y \in \mathfrak{m}$, so that $y \otimes 1 \in \mathfrak{m}B$.  Our work above shows that there exist $\beta \in J^2$ and $\gamma \in \mathfrak{m}J$ such that $x - \beta - \gamma$ lies in $L'_1$.  We have $(y \otimes 1) + \gamma \in \mathfrak{m}B$.  We conclude from the decomposition
\[
b = (c \otimes 1) + (x - \beta - \gamma) + \beta + (\gamma + (y \otimes 1))
\]
that $B \subset L_1 + J^2 + \mathfrak{m}B$ (since $c \otimes 1 = c \cdot (1 \otimes 1) \in L_0 \subset L_1$), and hence that $B/(\mathfrak{m}B + J^2)$ is spanned as a $k$-space by $L_1$.  Moreover, it follows by induction that for all $j$, $B/(\mathfrak{m}B + J^{j+1})$ is spanned as a $k$-space by $L_j$.  Assume the conclusion for $j-1$, that is, that $B \subset L_{j-1} + \mathfrak{m}B + J^j$.  We have already shown $J = L'_1 + \mathfrak{m}J + J^2$.  Taking the $j$th power of both sides, we find $J^j \subset \mathfrak{m}J + \sum_{\l = 0}^j L'_{j-\l} J^{2\l} \subset \mathfrak{m}B + L'_j + J^{j+1}$, since $b_i \in J$ and for any $\l > 0$, $2\l + (j-\l) \geq j+1$.  Therefore
\[
B \subset (L_{j-1} + L'_j) + \mathfrak{m}B + J^{j+1}.
\]  
Since $L_{j-1} + L'_j = L_j$, we have shown $B/(\mathfrak{m}B + J^{j+1})$ is spanned over $k$ by $L_j$, completing the induction.  It follows that every $B/(\mathfrak{m}B + J^{j+1})$ is a finite-dimensional $k$-space, completing the proof.
\end{proof}

\begin{remark}\label{psep} 
Our proof of Proposition \ref{fgdiff} above, together with \cite[Ex. 7.2]{eisenbud}, has the following consequence which we record separately for reference: if we fix generators $x_1, \ldots, x_s$ for the maximal ideal $\mathfrak{m}$, then $(P^j)^{sep}$ is generated over $R$ by the classes of monomials in $b_0, b_1, \ldots, b_n$ of degree at most $j$, where $b_0 = 1 \otimes 1$ and $b_i = x_i \otimes 1 - 1 \otimes x_i$ for $i = 1, \ldots, n$.
\end{remark}

We have now assembled enough preliminaries to prove the $\Sigma$-continuity of differential operators over a complete local ring:

\begin{prop}\label{diffsigma}
Let $(R, \mathfrak{m})$ be a (Noetherian) complete local ring with coefficient field $k$, and let $\D = \D(R,k)$.  If $M$ is a left $\D$-module and $\delta \in \D(R,k)$, then $\delta: M \rightarrow M$ is $\Sigma$-continuous. (The analogous statement is also true if $M$ is a right $\D$-module.)
\end{prop}

\begin{proof}
We verify the conditions of Definition \ref{arbdual}.  Let $M_{\lambda}$ be a finitely generated $R$-submodule of $M$ and let $\delta \in \D(R,k)$.  We assert that the $R$-submodule $\langle \delta(M_{\lambda}) \rangle$ of $M$ generated by the image of $M_{\lambda}$ under $\delta$ is finitely generated over $R$.  (By Lemma \ref{diffcont}, we already know the restriction of $\delta$ to $M_{\lambda}$ will be $\mathfrak{m}$-adically continuous, so this is all that must be proved.)  With $\lambda$ fixed, we proceed by induction on the order $j$ of $\delta$.  Fix a finite set of generators $m_1, \ldots, m_n$ for $M_{\lambda}$.  If $j = 0$, then $\delta$ is $R$-linear, so $\delta(m_1), \ldots, \delta(m_n)$ generate $\langle \delta(M_{\lambda}) \rangle$.  Now suppose $j > 0$ and the statement proved for smaller values of $j$.  By Proposition \ref{fgdiff}, we can find a finite set of $R$-module generators $d_1, \ldots, d_s$ for $\D^{j-1}(R)$.  Then we claim
\[
\{\delta(m_i)\}_i \cup \{d_{\l}(m_i)\}_{\l, i}
\] 
is a finite set of generators for $\langle \delta(M_{\lambda}) \rangle$.  Indeed, given any element $m_{\lambda} \in M_{\lambda}$, we can write it as a linear combination $r_1m_1 + \cdots + r_nm_n$, and $\delta(r_im_i) = r_i\delta(m_i) + [\delta, r_i](m_i)$ for all $i$.  Since $[\delta,r_i] \in \D^{j-1}(R)$, we can write $[\delta,r_i] = \rho_{1,i}d_1 + \cdots + \rho_{s,i}d_s$ for some $\rho_{1,i}, \ldots, \rho_{s,i} \in R$.  Then
\[
\delta(r_im_i) = r_i\delta(m_i) + \rho_{1,i}d_1(m_i) + \cdots + \rho_{s,i}d_s(m_i)
\] 
for all $i$.  The sum $\sum \delta(r_im_i) = \delta(m_{\lambda})$ thus belongs to the $R$-submodule of $M$ generated by the specified finite set, completing the proof.
\end{proof}

\begin{remark}\label{twomodules}
More generally, if $M$ and $N$ are any two $R$-modules, there is a notion of $k$-linear differential operators $M \rightarrow N$ \cite[D\'{e}f. 16.8.1]{EGAIV}. We do not know under what conditions such differential operators are $\Sigma$-continuous.
\end{remark} 

\begin{cor}[Corollary-Definition]\label{dualright}
Let $(R, \mathfrak{m})$ be a complete local ring with coefficient field $k$, and let $\D = \D(R,k)$.  Let $M$ be a left $\D$-module.  Then the Matlis dual $D(M) = \Hom_R(M,E)$ has a natural structure of right $\D$-module. If $M$ is a right $\D$-module, $D(M)$ has a natural structure of left $\D$-module.
\end{cor}

\begin{proof}
By symmetry, it suffices to prove the statement when $M$ is a left $\D$-module. Given any $\delta \in \D(R,k)$, Proposition \ref{diffsigma} and Proposition \ref{arbprecomp} imply that the Matlis dual $\delta^*: D(M) \rightarrow D(M)$ is defined.  We define a right $\D(R,k)$-action on $D(M)$ by $\phi \cdot \delta = \delta^*(\phi)$.  This definition satisfies the axioms for a right action since, given another differential operator $\delta' \in \D(R,k)$, we have $(\delta' \circ \delta)^* = \delta^* \circ (\delta')^*$, again by Proposition \ref{arbprecomp}. 
\end{proof}

\begin{prop}\label{dualofdlinear}
Let $\phi: M \rightarrow N$ be a homomorphism of left $\D$-modules. The Matlis dual $\phi^*: D(N) \rightarrow D(M)$, as defined in Proposition \ref{arbprecomp}, is a homomorphism of right $\D$-modules. The analogous statement for a homomorphism of right $\D$-modules is true as well. Therefore, Matlis duality defines contravariant functors from the category of left $\D$-modules to the category of right $\D$-modules and \emph{vice versa}.
\end{prop}

\begin{proof}
Again it suffices to prove the first statement. Note first that since $\phi$ is left $\D$-linear, it is $R$-linear and hence $\Sigma$-continuous by Lemma \ref{sigmacontfacts}(a). Let $\delta \in \D$ be given, and write $\delta_M: M \rightarrow M$ and $\delta_N: N \rightarrow N$ for the actions of $\delta$ on $M$ and $N$. As is clear from the proof of Corollary \ref{dualright}, the right $\D$-module structure on $D(M)$ is defined by transport of structure from that on $D^{\Sigma}(M)$ via the functorial isomorphism of Theorem \ref{arbsga2}. Therefore we need only check that $\phi^{\vee}: D^{\Sigma}(N) \rightarrow D^{\Sigma}(M)$ is right $\D$-linear. Since $\phi$ is left $\D$-linear, we have $\phi \circ \delta_M = \delta_N \circ \phi$. Therefore, if $\lambda \in D^{\Sigma}(N)$, we have
\[
\phi^{\vee}(\lambda \cdot \delta) = \phi^{\vee}(\lambda \circ \delta_N) = \lambda \circ \delta_N \circ \phi = \lambda \circ \phi \circ \delta_M = \phi^{\vee}(\lambda) \circ \delta_M = \phi^{\vee}(\lambda) \cdot \delta,
\]
so that $\phi^{\vee}$ is right $\D$-linear, completing the proof.
\end{proof}

Recall from section \ref{klinear} that if $M$ is an $R$-module, we have the natural \emph{evaluation} map $\iota: M \rightarrow D(D(M))$ which is injective for arbitrary $M$ and an isomorphism if $M$ is finitely generated or Artinian.  If $M$ is a left $\D$-module, then $D(D(M))$ is also a left $\D$-module (apply Corollary \ref{dualright} twice).  The following proposition justifies our use of the expression ``Matlis duality for $\D$-modules'':

\begin{prop}\label{doubleddual}
Let $M$ be a left $\D$-module.  The evaluation map $\iota: M \hookrightarrow D(D(M))$ is a morphism of left $\D$-modules.  In particular, if $M$ is finitely generated or Artinian as an $R$-module, then $\iota$ is an isomorphism of left $\D$-modules.
\end{prop}

\begin{proof}
By the proof of Proposition \ref{sigmabij}(a), we have an isomorphism
\[
\Psi: D(D(M)) \xrightarrow{\sim} D^{\Sigma}(D^{\Sigma}(M))
\]
of $R$-modules, and the left $\D$-module structure on $D(D(M))$ is defined, by transport of structure, using the left $\D$-module structure on the right-hand side.  We also have an evaluation map
\[
\iota': M \rightarrow D^{\Sigma}(D^{\Sigma}(M))
\]
defined by $\iota'(m)(\delta) = \delta(m)$ for all $m \in M$ and $\delta \in D^{\Sigma}(M)$.  The proof of Proposition \ref{sigmabij}(a) shows that $\Psi \circ \iota = \iota'$.  It therefore suffices to check that $\iota'$ is left $\D$-linear.  We must show that for any $m \in M$ and $d \in \D$, we have $\iota'(d \cdot m) = d \cdot \iota'(m)$.  The left-hand side is the evaluation map at $d \cdot m \in M$.  The right-hand side is the same, because $d$ acts on the evaluation map $\iota'$ by \emph{precomposition} with $d: M \rightarrow M$.  This completes the proof.
\end{proof}

We give some examples of Matlis duals with right $\D(R,k)$-structures, mostly involving local cohomology:

\begin{example}\label{farrago}
Let $(R, \mathfrak{m})$ be a complete local ring with coefficient field $k$.  
\begin{enumerate}[(a)]
\item Since $R$ is a left $\D(R,k)$-module and $E = D(R)$, $E$ has a natural structure of right $\D(R,k)$-module. 

\item If $M$ is a left $\D(R,k)$-module, so is $M_S$ for any multiplicatively closed subset $S \subset R$ \cite[Example 5.1(a)]{lyubeznik}.  If $I \subset R$ is an ideal, the local cohomology modules $H^i_I(R)$ supported at $I$ have the structure of left $\D(R,k)$-modules, because $H^i_I(R)$ is the $i$th cohomology object of a complex whose objects are localizations of the left $\D(R,k)$-module $R$ and whose maps, sums of natural localization maps, are $\D(R,k)$-linear \cite[Example 5.1(c)]{lyubeznik}.  By the previous theorem, the Matlis duals $D(H^i_I(R))$ are right $\D(R,k)$-modules.

\item Now suppose $I = \mathfrak{m}$.  If $R$ is Cohen-Macaulay, then $H^i_{\mathfrak{m}}(R)$ is zero unless $i = \dim(R)$.  The Matlis dual $D(H^{\dim(R)}_{\mathfrak{m}}(R))$, which is the \emph{canonical module} of $R$, therefore has a structure of right $\D(R,k)$-module.  If $R$ is not Cohen-Macaulay, there exists some $i < \dim(R)$ such that $H^i_{\mathfrak{m}}(R)$ is nonzero.  For such an $i$, the Matlis dual $D(H^i_{\mathfrak{m}}(R)$) is a right $\D(R,k)$-module that is finitely generated (since $H^i_{\mathfrak{m}}(R)$ is Artinian) as an $R$-module and whose dimension is strictly less than the dimension of $R$ (in fact, its dimension is bounded above by $i$: \cite[Exp. V, Thm. 3.1(ii)]{SGA2}).
\end{enumerate}
\end{example}

We give some more details related to Example \ref{farrago}(a).  Since $R$ is a finitely generated $R$-module, we know by Theorem \ref{sga2} that $E = D(R) \simeq \Hom_{cont, k}(R,k)$, and the right $\D$-action on $\Hom_{cont, k}(R,k)$ is defined by $\delta \cdot d = \delta \circ d$ for $d \in D$ and $\delta \in \Hom_{cont, k}(R,k)$.  If $M$ is any $R$-module, we have an isomorphism
\[
D^{\Sigma}(M) \xrightarrow{\sim} \Hom_R(M, \Hom_{cont, k}(R,k))
\]
by combining Theorem \ref{arbsga2} and Theorem \ref{sga2}.  Concretely, this isomorphism carries a $\Sigma$-continuous map $\delta: M \rightarrow k$ to the $R$-linear map $M \rightarrow \Hom_{cont, k}(R,k)$ defined by $m \mapsto (r \mapsto \delta(rm))$, and its inverse carries an $R$-linear map $\psi: M \rightarrow \Hom_{cont, k}(R,k)$ to the $\Sigma$-continuous map $M \rightarrow k$ defined by $m \mapsto \psi(m)(1)$.

By identifying $E$ with $\Hom_{cont, k}(R,k)$ endowed with the right $\D$-structure above, we can give an alternate description of the right action of \emph{derivations} (differential operators of order precisely $1$) on the Matlis dual of any left $\D$-module, using the formulas of \cite[Prop. 1.2.9]{hotta}: if $\delta \in \D$ is a derivation and $M$ is a left $\D$-module, then for any $R$-linear map $\phi: M \rightarrow E$, we define an $R$-linear map $\phi \cdot \delta: M \rightarrow E$ by $(\phi \cdot \delta)(m) = \phi(\delta \cdot m) + \phi(m) \cdot \delta$.  This formula extends to define a right action of the $R$-subalgebra of $\D$ generated by $R$ together with the derivations (e.g., if $R$ is regular, this subalgebra is all of $\D$).  This right action coincides with the right action we have defined in Corollary \ref{dualright}:

\begin{prop}\label{hottaaction}
With the hypotheses of Corollary \ref{dualright}, let $M$ be a left $\D$-module, and let $D(M)$ be its Matlis dual.  If $\delta \in \D$ is a derivation and $\phi: M \rightarrow E$ is $R$-linear, we have $\delta^*(\phi) = \phi \cdot \delta$, where the right-hand side is defined as in \cite[Prop. 1.2.9]{hotta}.
\end{prop}

\begin{proof}
We use the various identifications between equivalent forms of $D(M)$.  The map $\phi$ corresponds, under the isomorphism $D(M) \xrightarrow{\Phi_M} D^{\Sigma}(M)$ of Theorem \ref{arbsga2}, to $\sigma \circ \phi$, and $\delta^*(\phi) = \sigma \circ \phi \circ \delta$.  Under the identification $D^{\Sigma}(M) \simeq \Hom_R(M, \Hom_{cont, k}(R,k))$ given earlier, $\sigma \circ \phi \circ \delta$ corresponds to the map 
\[
m \mapsto (r \mapsto \sigma(\phi(\delta(rm)))).
\]  
On the other hand, for any $m \in M$, we have $(\phi \cdot \delta)(m) = \phi(\delta \cdot m) + \phi(m) \cdot \delta \in E$, where we identify $E$ with the right $\D$-module $\Hom_{cont, k}(R,k)$ in order to define $\phi(m) \cdot \delta$.  Under this identification, $\phi(m) \cdot \delta$ is the map $r \mapsto \sigma(r \phi(\delta \cdot m))$, and $\phi(\delta \cdot m)$ is the map $r \mapsto \sigma(\delta(r) \phi(m))$.  It is therefore enough to show that
\[
\sigma(\phi(\delta(rm))) = \sigma(r \phi(\delta \cdot m)) + \sigma(\delta(r) \phi(m))
\]
for all $r \in R$ and $m \in M$, which follows immediately from the relations $\delta(rm) = \delta(r)m + r \delta(m)$ holding in any left $\D$-module $M$.
\end{proof}

We now specialize further to the case in which $R = k[[x_1, \ldots, x_n]]$ is a formal power series ring over a field of characteristic zero.  In this case, there is a transposition operation that converts left modules over $\D(R,k)$ to right modules and conversely.  We recall its definition (see \cite[p. 19 and Lemma 1.2.6]{hotta} for the case of a polynomial ring, or more generally, a smooth variety over $k$; the formal power series definition is essentially the same).  Since $k$ is of characteristic zero, the ring $\D(R,k)$ is a free left $R$-module generated by monomials in $\partial_1 = \frac{\partial}{\partial x_1}, \ldots, \partial_n = \frac{\partial}{\partial x_n}$.   

\begin{definition}\label{transpose}
Let $R = k[[x_1, \ldots, x_n]]$ where $k$ is a field of characteristic zero, and let $M$ be a left $\D(R,k)$-module.  Let $\rho \partial_1^{a_1} \cdots \partial_n^{a_n}$ be an element of $\D(R,k)$, where $\rho \in R$.  If $\cdot$ denotes the given left action of $\D(R,k)$, then for any $m \in M$, the formula 
\[
m \ast (\rho \partial_1^{a_1} \cdots \partial_n^{a_n}) = ((-1)^{a_1 + \cdots + a_n} \partial_n^{a_n} \cdots \partial_1^{a_1} \rho) \cdot m
\]
defines the \emph{transpose} action, a \emph{right} $\D(R,k)$-action $\ast$ on $M$.
\end{definition}

There is, of course, a symmetric notion of the transpose of a right $\D$-module, which is a left $\D$-module.  To see that the right action given above is well-defined, we view it in the following way.  Let $\D(R,k)^{\circ}$ be the opposite algebra of $\D(R,k)$.  There exists a unique isomorphism $\phi: \D(R,k)^{\circ} \rightarrow \D(R,k)$ such that $\phi(\rho) = \rho$ for all $\rho \in R$ and $\phi(\partial_i) = -\partial_i$ for all $i$, which when viewed as a map $\phi: \D(R,k) \rightarrow \D(R,k)$ is called the \emph{principal anti-automorphism} of $\D(R,k)$.  To see that this is an isomorphism, note that since all elements of $R$ commute with each other and all $\partial_i$ commute with each other, the only non-trivial relations among elements of $\D(R,k)$ are the relations $\partial_i \rho = \rho \partial_i + \partial_i(\rho)$.  The map $\phi$, which is clearly bijective, carries $\partial_i \rho$ to $-\rho \partial_i$ and $\rho \partial_i + \partial_i(\rho)$ to $-\partial_i \rho + \partial_i(\rho)$; since $-\rho \partial_i = -\partial_i \rho + \partial_i(\rho)$, the relations are respected.  The transposed action $\ast$ is then simply defined by $m \ast \delta = \phi(\delta) \cdot m$ for $\delta \in \D(R,k)$.

\begin{remark}\label{transposedependence}
As defined in the previous paragraph, the anti-automorphism $\phi$ (and therefore the transposition operation) depends on the choice of variables $x_1, \ldots, x_n$.
\end{remark} 

\begin{prop}\label{regleft}
Let $R = k[[x_1, \ldots, x_n]]$ with $k$ a field of characteristic zero, $\D = \D(R,k)$, and $M$ any left $\D$-module.  There is a natural structure of left $\D$-module on the Matlis dual $D(M) = \Hom_R(M,E)$.
\end{prop}

\begin{proof}
Apply the right-to-left version of the transposition operation described above to the right $\D$-module $D(M)$ with structure defined in Corollary \ref{dualright}.
\end{proof}

Therefore, in this case, Matlis duality provides a (contravariant) functor from left $\D$-modules to left $\D$-modules.  We recall here Remark \ref{hellus}: even if $M$ is holonomic, the Matlis dual $D(M)$ need not be holonomic.

\section{The de Rham complex of a Matlis dual}\label{mittleff}

Let $R = k[[x_1, \ldots, x_n]]$ where $k$ is a field of characteristic zero, and let $\D = \D(R,k)$. For any left $\D$-module $M$, we can define its de Rham complex $M \otimes \Omega_R^{\bullet}$ (see section \ref{homology} for definitions and notation concerning de Rham complexes).  The Matlis dual $D(M)$ is also a left $\D$-module by Proposition \ref{regleft}, so we can consider the de Rham complex $D(M) \otimes \Omega_R^{\bullet}$.  (Every $\D$-module we consider in this section will be a left $\D$-module, so we no longer say so explicitly.)

Our goal in this section is to compare the cohomology of these two complexes.  Specifically, we will show the following:

\begin{thm}\label{dualcoh}
Let $R$ and $\D$ be as above.  If $M$ is a \emph{holonomic} $\D$-module, then for all $i$, we have isomorphisms
\[
(H^i_{dR}(M))^{\vee} \simeq H^{n-i}_{dR}(D(M))
\]
where $\vee$ denotes $k$-linear dual, \emph{i.e.} $\Hom_k(-,k)$.
\end{thm}

By Proposition \ref{diffsigma}, the differentials in the complex $M \otimes \Omega_R^{\bullet}$ are $\Sigma$-continuous, so the entire complex can be Matlis dualized.  Since the functor $D$ is contravariant, the $i$th object in this dualized complex $D(M \otimes \Omega_R^{\bullet})$ is $D(M \otimes \Omega_R^{n-i})$. Theorem \ref{dualcoh} is a trivial consequence of the following pair of propositions.

\begin{prop}\label{dualcoh1}
Let $R$ and $\D$ be as above.  If $M$ is a \emph{holonomic} $\D$-module, then for all $i$, we have isomorphisms
\[
(h^i(M \otimes \Omega_R^{\bullet}))^{\vee} \simeq h^{n-i}(D(M \otimes \Omega_R^{\bullet})).
\]
\end{prop}

\begin{prop}\label{dualcoh2}
Let $R$ and $\D$ be as above.  If $M$ is \emph{any} $\D$-module, then for all $i$, we have isomorphisms
\[
h^i(D(M \otimes \Omega_R^{\bullet})) \simeq h^i(D(M) \otimes \Omega_R^{\bullet}).
\]
\end{prop}

Proposition \ref{dualcoh2} is relatively straightforward, and we prove it first.  The proof of Proposition \ref{dualcoh1}, which is the longest and most involved proof in this paper, will take up the remainder of this section.

Before giving the proof of Proposition \ref{dualcoh2}, we recall the definition of the Koszul complex, for which a reference is \cite[\S 4.5]{weibel}:

\begin{definition}\label{koszul}
Let $R$ be a commutative ring, let $\mathbf{x} = (x_1, \ldots, x_n)$ be a sequence of elements of $R$, and let $M$ be an $R$-module.  The \emph{Koszul complex} $K_{\bullet}(\mathbf{x})$ of $R$ with respect to $\mathbf{x}$ is the homologically indexed complex of length $n$ whose $i$th object $K_i(\mathbf{x})$ is a direct sum of $n \choose i$ copies of $R$ (indexed by $i$-tuples $e_{j_1} \wedge \cdots \wedge e_{j_i}$ where $1 \leq j_1 < \cdots < j_i \leq n$) and where the differential $d_i: K_i(\mathbf{x}) \rightarrow K_{i-1}(\mathbf{x})$ carries the basis element $e_{j_1} \wedge \cdots \wedge e_{j_i}$ to 
\[
\sum_{s=1}^i (-1)^{s-1} x_{j_s} \, e_{j_1} \wedge \cdots \wedge \widehat{e_{j_s}} \wedge \cdots \wedge e_{j_i},
\]
where the $\widehat{e_{j_s}}$ means that symbol is omitted.  The \emph{homological Koszul complex} $K_{\bullet}(M;\mathbf{x})$ of $M$ with respect to $\mathbf{x}$ is the complex $K_{\bullet}(\mathbf{x}) \otimes_R M$, and the \emph{cohomological Koszul complex} $K^{\bullet}(M; \mathbf{x})$ of $M$ with respect to $\mathbf{x}$ is the complex $\Hom_R(K_{\bullet}(\mathbf{x}), M)$.
\end{definition}

\begin{prop}\cite[Ex. 4.5.2]{weibel}\label{koszulhc}
Let $R$ be a commutative ring, let $\mathbf{x} = (x_1, \ldots, x_n)$ be a sequence of elements of $R$, and let $M$ be an $R$-module.  For all $i$, we have $h_i(K_{\bullet}(M;\mathbf{x})) \simeq h^{n-i}(K^{\bullet}(M; \mathbf{x}))$ as $R$-modules.
\end{prop}

We can now prove Proposition \ref{dualcoh2}:

\begin{proof}[Proof of Proposition \ref{dualcoh2}]
We first compute the differentials in the complex $D(M \otimes \Omega_R^{\bullet})$.  Let $i$ be given, and consider the differential $d^i: M \otimes \Omega_R^i \rightarrow M \otimes \Omega_R^{i+1}$.  An element of $M \otimes \Omega_R^i$ is a sum of terms of the form $m_{j_1 \cdots j_i} \, dx_{j_1} \wedge \cdots \wedge dx_{j_i}$ where $1 \leq j_1 < \cdots < j_i \leq n$, and the formula for $d^i$ is 
\[
d^i(m\, dx_{j_1} \wedge \cdots \wedge dx_{j_i}) = \sum_{s=1}^n \partial_s(m)\, dx_s \wedge dx_{j_1} \wedge \cdots \wedge dx_{j_i}.
\]
Now consider the Matlis dual of this differential.  Since the Matlis dual functor commutes with finite direct sums, we can identify $D(M \otimes \Omega_R^i)$ with a direct sum of $n \choose i$ copies of $D(M)$, again indexed by the $dx_{j_1} \wedge \cdots \wedge dx_{j_i}$.  If $\phi \in D(M)$, we have the formula
\[
(d^i)^*(\phi \, dx_{j_1} \wedge \cdots \wedge dx_{j_{i+1}}) = \sum_{s=1}^{i+1} (-1)^{s-1} \partial_{j_s}^*(\phi)\, dx_{j_1} \wedge \cdots \wedge \widehat{dx_{j_s}} \wedge \cdots \wedge dx_{j_{i+1}}.
\]
Consider the commutative subring $\Delta = k[\partial_1, \ldots, \partial_n] \subset \D$.  The $\D$-module $D(M)$ is \emph{a fortiori} a $\Delta$-module, and the de Rham complex $D(M) \otimes \Omega_R^{\bullet}$ is the \emph{cohomological} Koszul complex $K^{\bullet}(D(M); \mathbf{\partial})$ of the $\Delta$-module $D(M)$ with respect to $\mathbf{\partial} = (\partial_1, \ldots, \partial_n)$, where the $\partial_i$ act on $D(M)$ via the maps $-\partial_i^*$ (according to Definition \ref{transpose}).  On the other hand, by the formula above, it is clear that $D(M \otimes \Omega_R^{\bullet})$ is the \emph{homological} Koszul complex $K_{\bullet}(D(M); \mathbf{\partial})$ (up to a sign, which does not affect cohomology).  We have
\[
h^i(K^{\bullet}(D(M); \mathbf{\partial})) \simeq h_{n-i}(K_{\bullet}(D(M); \mathbf{\partial}))
\] 
by Proposition \ref{koszulhc}; regarding the complex on the right as being \emph{cohomologically} indexed (as we do when considering it as the Matlis dual of $M \otimes \Omega_R^{\bullet}$), we see that the right-hand side is its $i$th cohomology object, completing the proof.  
\end{proof}

The remainder of this section is long and contains a great deal of preliminary material necessary for the proof of Proposition \ref{dualcoh1}.  Before giving this preliminary material, we first outline it for the reader's benefit, then introduce some notation that we will use repeatedly.

\begin{itemize}
\item First, we prove some lemmas concerning direct and inverse systems of modules.  The key definition here, Definition \ref{strongstable} (\emph{strong-sense stability}), is a dual version of the Mittag-Leffler condition.
\item We then introduce some definitions and results due to van den Essen, who in a series of papers studied the kernels and cokernels of differential operators.  Not only his results, but also some of the ideas in his proofs, will be of necessary use to us.  We discuss changes of variable and prove a technical lemma, Lemma \ref{techlemma}, that relies on van den Essen's work.
\item We next describe how to ``stratify'' the de Rham complex $M \otimes \Omega_R^{\bullet}$, writing it as a direct limit of ``de Rham-like'' complexes whose objects are finitely generated $R$-modules.  The crucial result concerning this direct system is Proposition \ref{mitlef}, which asserts that the cohomology objects of these complexes satisfy strong-sense stability with finite-dimensional stable images.  We also give a more general version of this result, Corollary \ref{fgstabdR}, which will be of no use to us in this section but to which we will need to appeal in section \ref{cohom}.
\item Finally, we give the proof of Proposition \ref{dualcoh1}, using our work in section \ref{klinear} on Matlis duality.
\end{itemize}

We need to work not only with the rings $R$ and $\D$, but also with subrings defined using proper subsets of $\{x_1, \ldots, x_n\}$:

\begin{definition}\label{partialR}
Let $j \geq 0$ be given.  We denote by $R_j$ the subring $k[[x_1, \ldots, x_j]] \subset R$, and by $R^j$ the subring $k[[x_{n-j+1}, \ldots, x_n]] \subset R$.  (Thus, $R = R_n = R^n$ and $k = R_0 = R^0$.)  We denote by $\D_j = \D(R_j, k)$ and $\D^j = \D(R^j, k)$ the corresponding rings of differential operators, which are subrings of $\D$.
\end{definition}

This notation will be in force throughout the section, and will not be repeated in the statements of results.

\begin{remark}\label{partialdRj}
If $M$ is any $\D$-module, it is also a module over $\D_j$ and $\D^j$ for all $j$, and we have short exact sequences relating its ``partial'' and ``full'' de Rham complexes analogous to the short exact sequence of Definition \ref{partialdR}.  These sequences take the form
\[
0 \rightarrow M \otimes \Omega_{R^j}^{\bullet}[-1] \rightarrow M \otimes \Omega_{R^{j+1}}^{\bullet} \rightarrow M \otimes \Omega^{\bullet}_{R^j} \rightarrow 0
\]
where the first map is given by $\wedge \, dx_{n-j}$.  The maps in the complex $M \otimes \Omega_{R^j}^{\bullet}$ are $\D_{n-j}$-linear, and hence its cohomology objects are $\D_{n-j}$-modules; by the same argument as in Lemma \ref{connhom}, if we consider the associated long exact cohomology sequence, its connecting homomorphisms (up to a sign) are simply $\partial_{n-j}$.  Of course, there is also a version of this sequence involving de Rham complexes over $R_j$ and $R_{j+1}$ instead of $R^j$ and $R^{j+1}$.
\end{remark}

In the proof of Proposition \ref{dualcoh1}, we will consider various direct and inverse systems of complexes of $R$-modules and $k$-spaces.  (All our direct and inverse systems will be indexed by the natural numbers, but the following discussion applies to any filtered index set.) The interaction of cohomology with \emph{direct} limits in these categories is not complicated: the direct limit is an exact functor \cite[Thm. 2.6.15]{weibel}, and so it commutes with cohomology.  For \emph{inverse} limits, more caution is required, as the inverse limit is, in general, only left exact.  In order to ensure that cohomology commutes with inverse limits, we will need to verify the Mittag-Leffler condition for the inverse systems we consider (see Proposition \ref{mlcohcomm} below).

\begin{definition}\cite[13.1.2]{ega31}\label{mittagleffler}
Let $\{M_i\}$ be an inverse system (indexed by $\mathbb{N}$) of modules over a commutative ring $R$, with inverse limit $M = \varprojlim M_i$ and transition maps $f_{ji}: M_j \rightarrow M_i$ for $i \leq j$.  We say that the system $\{M_i\}$ satisfies the \emph{Mittag-Leffler condition} if for all $\l$, the descending chain $\{f_{\l+s, \l}(M_{\l+s})\}_{s \geq 0}$ of submodules of $M_{\l}$ becomes stationary: there exists some $s$ such that for all $t \geq s$, $f_{\l+s, \l}(M_{\l+s}) = f_{\l+t, \l}(M_{\l+t})$.
\end{definition}

The utility of the Mittag-Leffler condition for us is contained in the following proposition, which is an immediate consequence of \cite[Prop. 13.2.3]{ega31} (the result is stated in \cite{ega31} only for complexes of Abelian groups, but the proof given there is valid more generally):

\begin{prop}\cite[Prop. 13.2.3]{ega31}\label{mlcohcomm}
Let $\{M_i^{\bullet}\}$ be an inverse system (indexed by $\mathbb{N}$) of complexes of modules over a commutative ring $R$.  For every $j$, there is a canonical homomorphism
\[
\eta_j: h^j(\varprojlim_i M_i^{\bullet}) \rightarrow \varprojlim_i h^j(M_i^{\bullet})
\]
of $R$-modules.  Suppose that for every $j$, the inverse systems $\{M_i^j\}_i$ and $\{h^j(M_i^{\bullet})\}_i$ of $R$-modules both satisfy the Mittag-Leffler condition.  Then for every $j$, $\eta_j$ is an isomorphism of $R$-modules.
\end{prop}

The following condition on \emph{direct} systems can be thought of as a sort of dual to the Mittag-Leffler condition:

\begin{definition}\label{strongstable}
Let $\{M_i\}$ be a direct system (indexed by $\mathbb{N}$) of modules over a commutative ring $R$, with direct limit $M = \varinjlim M_i$, $R$-linear transition maps $f_{ij}: M_i \rightarrow M_j$ for $i \leq j$, and $R$-linear insertion maps $f_i: M_i \rightarrow M$.  Fix $j$ and let $N_j \subset M_j$ be an $R$-submodule.  We say that the images of $N_j$ under the transition maps \emph{stabilize in the strong sense} if there exists $\l$ such that $f_{j+\l}: M_{j+\l} \rightarrow M$ induces an isomorphism 
\[
f_{j,j+\l}(N_j) \xrightarrow{\sim} f_j(N_j);
\] 
equivalently, $f_{j, j+\l}(N_j) \cap \ker f_{j+\l} = 0$.  (If we say that the images of some object stabilize in the strong sense, it will be clear from context with respect to which transition maps and direct system we mean.)
\end{definition}

Note that since $\ker f_{j+\l,j+\l'} \subset \ker f_{j+\l}$ for any $\l' \geq \l$, it follows at once from Definition \ref{strongstable} that there are also isomorphisms $f_{j,j+\l}(N_j) \xrightarrow{\sim} f_{j,j+\l'}(N_j)$ induced by the transition maps. 

This condition is automatic if $R$ is Noetherian and $N_j$ is finitely generated:

\begin{lem}\label{fgstable}
Let $R$ be a commutative Noetherian ring. If $\{M_i\}$ is a direct system of $R$-modules as in Definition \ref{strongstable}, and $N_j \subset M_j$ is a finitely generated $R$-submodule, then the images of $N_j$ stabilize in the strong sense.
\end{lem}

\begin{proof}
Since $R$ is Noetherian, $\ker f_j \cap N_j$ is also a finitely generated $R$-module.  Fix $R$-generators $x_1, \ldots, x_s$ for $\ker f_j \cap N_j$.  For all $i \in \{1, \ldots, s\}$, we have $f_j(x_i) = 0$, and so there exists $\l_i \geq 0$ such that $f_{j, j+\l_i}(x_i) = 0 \in M_{j+\l_i}$.  If we put $\l = \max{\{\l_i\}}$, then $f_{j,j+\l}$ annihilates $\ker f_j \cap N_j$.  We claim that 
\[
f_{j, j+\l}(N_j) \cap \ker f_{j+\l} = 0.
\]  
Suppose that $x \in M_{j + \l}$ belongs to this intersection.  Then we have $f_{j+\l}(x) = 0$ and $x = f_{j,j+\l}(y)$ for some $y \in N_j$.  From $f_j(y) = f_{j+\l}(f_{j,j+\l}(y)) = f_{j+\l}(x) = 0$, we conclude $y \in \ker f_j$; but $y \in N_j$ as well, and since $f_{j,j+\l}$ annihilates $\ker f_j \cap N_j$, it follows that $x = f_{j,j+\l}(y) = 0$, completing the proof.
\end{proof}

In the case of vector spaces over a field, the connection between the Mittag-Leffler condition and strong-sense stability can be made more precise.  First, we need a basic lemma concerning the interaction of images and dual spaces.  (Recall that $\vee$ denotes $k$-linear dual.)

\begin{lem}\label{ddualml}
Let $k$ be a field.
\begin{enumerate}[(a)]
\item If $V$ and $W$ are vector spaces over $k$, and $f: V \rightarrow W$ is a $k$-linear map, we can identify $(\im f)^{\vee}$ with a subspace of $V^{\vee}$ in such a way that $(\im f)^{\vee} = \im(f^{\vee})$.
\item If $\{U_i\}_{i \in \mathbb{N}}$ is an inverse system of $k$-spaces (with transition maps $f_{ji}: U_j \rightarrow U_i$ for $i \leq j$) which satisfies the Mittag-Leffler condition, then the inverse system $\{U_i^{\vee \vee}\}$ also satisfies the Mittag-Leffler condition. 
\end{enumerate} 
\end{lem}

\begin{proof}
Factor $f$ as $V \twoheadrightarrow \im f \subset W$, and dualize this factorization to obtain $W^{\vee} \twoheadrightarrow (\im f)^{\vee} \hookrightarrow V^{\vee}$.  The image of the last map is $\im(f^{\vee})$, and the injectivity allows us to identify this image with $(\im f)^{\vee}$.  This proves (a).

We now prove (b).  Fix $i$.  By the Mittag-Leffler condition, the descending chain $\{f_{ji}(U_j)\}_{j \geq i}$ of subspaces of $U_i$ stabilizes, say at $j = \l$. Therefore $f_{\l i}(U_{\l}) = f_{si}(U_s)$ for all $s \geq \l$, and consequently $(f_{\l i}(U_{\l}))^{\vee \vee} = (f_{si}(U_s))^{\vee \vee}$ as subspaces of $U_i^{\vee \vee}$ for all $s \geq \l$.  After applying part (a) twice, it follows that $f_{\l i}^{\vee \vee}(U_{\l}^{\vee \vee}) = f_{si}^{\vee \vee}(U_s^{\vee \vee})$ as subspaces of $U_i^{\vee \vee}$ for all $s \geq \l$, that is, that the descending chain $\{f_{ji}^{\vee \vee}(U_j^{\vee \vee})\}$ of subspaces of $U_i^{\vee \vee}$ stabilizes at $j = \l$.  We conclude that $\{U_i^{\vee \vee}\}$ also satisfies the Mittag-Leffler condition, completing the proof.
\end{proof} 

\begin{lem}\label{dualml}
Let $\{U_i\}$ be an inverse system (indexed by $\mathbb{N}$) of vector spaces over a field $k$, with transition maps $f_{ji}$.  For all $i$, let $V_i = U_i^{\vee}$ be the $k$-space dual of $U_i$, and regard $\{V_i\}$ as a direct system with transition maps $\lambda_{ij} = f_{ji}^{\vee}$ for $i \leq j$.  Suppose that for all $\l$, the images of $V_{\l}$ under the transition maps $\lambda_{\l,\l+s}$ stabilize in the strong sense.  Then $\{U_i\}$ satisfies the Mittag-Leffler condition.
\end{lem}

\begin{proof}
We prove the contrapositive.  Suppose that the system $\{U_i\}$ does not satisfy Mittag-Leffler.  Then there is an $\l$ such that for all $s$, there exists $t \geq s$ such that $f_{\l+t,\l}(U_{\l+t}) \subsetneq f_{\l+s,\l}(U_{\l+s})$.  Since $k$-space dual is an exact functor, proper injections dualize to surjections with nontrivial kernels, so the surjection $(f_{\l+s,\l}(U_{\l+s}))^{\vee} \twoheadrightarrow (f_{\l+t,\l}(U_{\l+t}))^{\vee}$ is not an isomorphism.  By Lemma \ref{ddualml}(a), we can identify $(f_{\l+s,\l}(U_{\l+s}))^{\vee}$ with $\lambda_{\l,\l+s}(V_{\l})$ (resp. $(f_{\l+t,\l}(U_{\l+t}))^{\vee}$ with $\lambda_{\l,\l+t}(V_{\l})$) as a subspace of $V_{\l+s}$ (resp. $V_{\l+t}$), and under these identifications, the surjection $(f_{\l+s,\l}(U_{\l+s}))^{\vee} \twoheadrightarrow (f_{\l+t,\l}(U_{\l+t}))^{\vee}$ is nothing but $\lambda_{\l+s,\l+t}$ restricted to $\lambda_{\l,\l+s}(V_{\l})$.  We conclude that the images of $V_{\l}$ cannot stabilize in the strong sense.
\end{proof}

We will also need a technical lemma concerning double duals:

\begin{lem}\label{finddual}
Let $k$ be a field, and let $\{U_i\}_{i \in \mathbb{N}}$ be an inverse system of $k$-spaces, indexed by $\mathbb{N}$, with transition maps $f_{ji}: U_j \rightarrow U_i$ for $i \leq j$, and inverse limit $U = \varprojlim U_i$.  Suppose that the inverse system $\{U_i\}$ satisfies the Mittag-Leffler condition and that the inverse limit $\varprojlim(U_i^{\vee \vee})$ is a finite-dimensional $k$-space.  Then $\varprojlim U_i \simeq \varprojlim(U_i^{\vee \vee})$.
\end{lem}

\begin{proof}
The canonical map from a $k$-space $U_i$ to its double dual $U_i^{\vee \vee}$ is always injective, and the inverse limit is left exact \cite[Prop. 10.2]{atiyah}, so there is a natural injection $U = \varprojlim U_i \hookrightarrow \varprojlim U_i^{\vee \vee}$.  Therefore $U$ is also finite-dimensional.  For all $i$, let $U'_i = \cap_{i \leq j} f_{ji}(U_j)$ be the stable image of the transition maps inside $U_i$ (this descending chain stabilizes by the Mittag-Leffler hypothesis).  Then \cite[p. 191]{hartshorneAG} $\{U_i'\}$ is also an inverse system, now with surjective transition maps, such that $\varprojlim U'_i = U$ and $U$ maps surjectively to all $U'_i$.  In particular, all $U'_i$ are finite-dimensional $k$-spaces and hence isomorphic to their double duals $(U'_i)^{\vee \vee}$.  By Lemma \ref{ddualml}(b), the inverse system $\{U_i^{\vee \vee}\}$ also satisfies the Mittag-Leffler condition, and moreover, we can identify the double dual of a stable image, $(U'_i)^{\vee \vee}$, with the stable image of the double dual, $(U_i^{\vee \vee})' \subset U_i^{\vee \vee}$.  Since for all $i$, $U'_i$ is canonically isomorphic to its double dual, the corresponding inverse limits are also isomorphic.  Thus we have isomorphisms
\[
\varprojlim U_i \simeq \varprojlim U'_i \simeq \varprojlim (U'_i)^{\vee \vee} \simeq \varprojlim (U_i^{\vee \vee})' \simeq \varprojlim U_i^{\vee \vee},
\]
completing the proof.
\end{proof}

\begin{remark}\label{garrett}
The Mittag-Leffler hypothesis in Lemma \ref{finddual} is actually superfluous: Paul Garrett (private communication) has given a proof assuming only that $\varprojlim(U_i^{\vee \vee})$ is finite-dimensional.  However, we will only need the statement in the Mittag-Leffler case, and the proof in general is more difficult.
\end{remark}

In a series of papers \cite{essthesis, kernel, example, cokernel, essen}, van den Essen examined the effect of the operator $\partial_n$ acting on a holonomic $\D$-module, culminating in a proof of Theorem \ref{findimdR}.  (See also \cite{nsexpos} for an expository account of this proof.)  Van den Essen's first result was that the kernel of $\partial_n$ behaves as well as we might hope:

\begin{thm} \label{kernel} \cite[Thm.]{kernel}
If $M$ is a holonomic $\D$-module and $M^{*}$ is the kernel of $\partial_n: M \rightarrow M$, then $M^{*}$ is a holonomic $\D_{n-1}$-module.
\end{thm}

The key step in the proof of Theorem \ref{kernel} is the following lemma, which we will use below in the proof of Lemma \ref{techlemma}:

\begin{lem} \label{kernellem} \cite[Cor. 2]{kernel}
Let $M$ be a $\D$-module and let $M^{*}$ be the kernel of $\partial_n: M \rightarrow M$.  If $M = R \cdot M^{*}$, then $M = M^{*} \oplus x_n M$ as $\D_{n-1}$-modules.
\end{lem}

Without some conditions, there is no analogous result for cokernels; van den Essen gives a concrete example in \cite[Thm.]{example} in which $M$ is a holonomic $\D$-module but $M/\partial_n(M)$ is not holonomic.

\begin{prop}\cite[Prop. 1]{example}\label{fingenloc}
Let $M$ be a holonomic $\D$-module.  There exists an element $g \in R$ such that the localization $M_g$ is a finitely generated $R_g$-module.
\end{prop}

\begin{definition}\label{esubtau}
Let $M$ be a $\D$-module, let $m \in M$, and let $\tau \in \D$ be a derivation.  We write $E_{\tau}(m)$ for the $R$-submodule of $M$ generated by $\{\tau^i(m)\}_{i \geq 0}$.
\end{definition}

\begin{lem}\cite[Ch. II, Prop. 1.16]{essthesis}\label{vdethesis}
Let $M$ be a $\D$-module.  Suppose that there exists $g \in R$ which is \emph{$x_n$-regular} (that is, such that $g(0, 0, \ldots, 0, x_n) \neq 0$) and such that $M_g$ is a finitely generated $R_g$-module.  Then for all $m \in M$, there exists an $x_n$-regular $f \in R$ such that the $R$-submodule $E_{f \partial_n}(m)$ of $M$ is finitely generated.  
\end{lem}

The source \cite{essthesis} remains unpublished; see also \cite[Lemma 4.2]{nsexpos} for a proof of Lemma \ref{vdethesis}.  The conclusion of this lemma leads us to make the following definition:

\begin{definition}\label{xnreg}
If $M$ is a $\D$-module, an element $m \in M$ is \emph{$x_n$-regular} if there exists an $x_n$-regular element $f \in R$ such that $E_{f \partial_n}(m)$ is a finitely generated $R$-module.  A holonomic $\D$-module $M$ is \emph{$x_n$-regular} if there exists an $x_n$-regular $m \in M$ such that $M = \D \cdot m$.
\end{definition}

From the following proposition, it follows immediately that if a holonomic $\D$-module $M$ is $x_n$-regular, then \emph{every} $m \in M$ such that $M = \D \cdot m$ is $x_n$-regular:

\begin{prop}\cite[Prop. II.1.3(2)]{essthesis}\label{deltaregular}
Let $M$ be a $\D$-module, let $m \in M$, and let $\tau \in \D$ be a derivation.  If the $R$-module $E_{\tau}(m)$ is finitely generated, so is the $R$-module $E_{\tau}(\delta(m))$ for any $\delta \in \D$.
\end{prop}

Since \cite{essthesis} remains unpublished, we reproduce for the reader the proof of Proposition \ref{deltaregular} given in \cite{essthesis}:

\begin{proof}
Since $\D$ is generated over $R$ by finite products of the derivations $\partial_i$, it clearly suffices to check that if $E_{\tau}(m)$ is finitely generated, so is $E_{\tau}(\partial_i(m))$ for all $i$.  Fix such an $i$.  The hypothesis that $E_{\tau}(m)$ is finitely generated implies that there exists $p$ such that 
\[
\tau^p(m) \in R \cdot m + R \cdot \tau(m) + \cdots + R \cdot \tau^{p-1}(m).
\]  
For every $j$, the commutator $[\tau, \partial_j] = \tau \partial_j - \partial_j \tau \in \D$ is again a derivation, hence is an $R$-linear combination of $\partial_1, \ldots, \partial_n$.  That is, we have
\[
\tau \partial_j - \partial_j \tau \in R \cdot \partial_1 + \cdots + R \cdot \partial_n.
\]
An induction argument using the two displayed statements shows that for every $\l$, we have
\[
\tau^{\l}(\partial_i(m)) \in \sum_{0 \leq q \leq p-1, 1 \leq j \leq n} R \cdot \partial_j(\tau^q(m)).
\]
(We induce on $\l$: the second displayed statement is used to move $\tau$s to the right past $\partial$s, and the first displayed statement is used to limit the number of distinct powers of $\tau$ that appear.)  It follows that $E_{\tau}(\partial_i(m))$ is contained in the $R$-submodule of $M$ generated by the finite set $\{\partial_j(\tau^q(m))\}_{0 \leq q \leq p-1, 1 \leq j \leq n}$, completing the proof.
\end{proof}

More generally, a $\D$-module $M$ is said to be \emph{$x_n$-regular} if every $m \in M$ is $x_n$-regular.

\begin{lem}\cite[Cor. 1.8]{essen}\label{coordchange}
Let $M$ be a holonomic $\D$-module.  There exists a change of variables (that is, a replacement of $x_1, \ldots, x_n$ with another regular system of parameters for $R$) after which $M$ is $x_n$-regular.
\end{lem}

\begin{proof}
We claim that, in fact, there is such a change of variables which is \emph{linear} in the $x_i$.  By Proposition \ref{fingenloc}, there exists $g \in R$ such that $M_g$ is finitely generated over $R_g$.  By Lemma \ref{vdethesis}, if $g$ is $x_n$-regular, then $M$ is $x_n$-regular.  Since $k$ is of characteristic zero and hence infinite, there is a linear change of variables after which $g$ is $x_n$-regular, completing the proof.
\end{proof}

By Proposition \ref{dRind}, the de Rham cohomology $H^{*}_{dR}(M)$ is independent of this coordinate change.  Our next observation is that the $x_n$-regularity condition is precisely what is required for the holonomy of the cokernel of $\partial_n$:

\begin{thm} \label{cokernel} \cite[Thm.]{cokernel}  
If $M$ is a holonomic $\D$-module that is $x_n$-regular, then $\overline{M} = M/\partial_n(M)$ is a holonomic $\D_{n-1}$-module.
\end{thm}

The key step in the proof of Theorem \ref{cokernel} is the following lemma, which we will also use below in the proof of Lemma \ref{techlemma}:

\begin{lem} \label{vdElemma} \cite[Cor. 2]{cokernel}
Let $M$ be a $\D$-module.  Suppose that $m \in M$ is $x_n$-regular.  Then there exists a finitely generated $R_{n-1}$-submodule $L$ of $R \cdot m$ and a natural number $p$ such that $R \cdot m \subset L + \sum_{i=1}^p \partial_n^i(R \cdot m) \subset L + \partial_n(\sum_{i=0}^{p-1} R \cdot \partial_n^i(m))$.
\end{lem} 

\begin{proof}
The first containment is \cite[Cor. 2]{cokernel}, and the second follows from the fact (easy to check using the Leibniz rule) that for all $i \geq 0$ we have $\partial_n^i(R \cdot m) \subset \sum_{j=0}^i R \cdot \partial_n^j(m)$.
\end{proof}

Our next two preliminary results (\ref{coords} and \ref{techlemma}) deal with the kernels and cokernels of the operators $\partial_i$.  We are going to study the de Rham complex of a $\D$-module $M$ by isolating one $\partial_i$ at a time, and these results guarantee that this process is well-behaved.

\begin{prop}\label{coords}
Let $M$ be a holonomic $\D$-module.  There exists a change of variables (which, by Proposition \ref{dRind}, does not alter the de Rham cohomology of $M$) after which, for all $i$ and $j$, $h^i(M \otimes \Omega^{\bullet}_{R^j})$ is a holonomic $\D_{n-j}$-module which is $x_{n-j}$-regular.
\end{prop}

\begin{proof}
Let $j$ be given, and consider the short exact sequence
\[
0 \rightarrow M \otimes \Omega^{\bullet}_{R^j}[-1] \rightarrow M \otimes \Omega^{\bullet}_{R^{j+1}} \rightarrow M \otimes \Omega^{\bullet}_{R^j} \rightarrow 0
\]
of Remark \ref{partialdRj}.  The corresponding long exact sequence in cohomology takes the form
\begin{align*}
\cdots \rightarrow h^i(M \otimes \Omega^{\bullet}_{R^{j+1}}) \rightarrow h^i(M \otimes \Omega^{\bullet}_{R^j}) &\xrightarrow{\partial} h^{i+1}(M \otimes \Omega^{\bullet}_{R^j}[-1]) \, (= h^i(M \otimes \Omega^{\bullet}_{R^j})) \rightarrow \\ h^{i+1}(M \otimes \Omega^{\bullet}_{R^{j+1}}) &\rightarrow h^{i+1}(M \otimes \Omega^{\bullet}_{R^j})  \xrightarrow{\partial} \cdots,
\end{align*}
where the connecting homomorphisms $\partial$ are given, up to a sign, by $\partial_{n-j}$.  From this long sequence, we obtain short exact sequences
\[
0 \rightarrow C^{i-1}_{n-j} \rightarrow h^i(M \otimes \Omega^{\bullet}_{R^{j+1}}) \rightarrow K^i_{n-j} \rightarrow 0
\]
of $\D_{n-j-1}$-modules, where $C^{i-1}_{n-j}$ is the cokernel of $\partial_{n-j}$ acting on $h^{i-1}(M \otimes \Omega^{\bullet}_{R^j})$ and $K^i_{n-j}$ is the kernel of $\partial_{n-j}$ acting on $h^i(M \otimes \Omega^{\bullet}_{R^j})$.  If the $\D_{n-j}$-module $h^i(M \otimes \Omega^{\bullet}_{R^j})$ is holonomic, then the $\D_{n-j-1}$-module $K^i_{n-j}$ is holonomic by Theorem \ref{kernel}.  If, moreover, the $\D_{n-j}$-module $h^{i-1}(M \otimes \Omega^{\bullet}_{R^j})$ is holonomic and $x_{n-j}$-regular, then the $\D_{n-j-1}$-module $C^{i-1}_{n-j}$ is holonomic by Theorem \ref{cokernel}, from which it follows from the short exact sequence displayed above that the $\D_{n-j-1}$-module $h^i(M \otimes \Omega^{\bullet}_{R^{j+1}})$ is holonomic as well.  Therefore, if $h^i(M \otimes \Omega^{\bullet}_{R^j})$ is holonomic and $x_{n-j}$-regular for all $i$, then $h^i(M \otimes \Omega^{\bullet}_{R^{j+1}})$ is holonomic for all $i$.  This fact will enable us to prove the proposition by induction on $j$.

By hypothesis, $M = h^0(M \otimes \Omega^{\bullet}_{R^0})$ is a holonomic $\D$-module, and by Lemma \ref{coordchange}, there exists a change of variables after which $M$ is $x_n$-regular.  It follows from the previous paragraph that $h^i(M \otimes \Omega^{\bullet}_{R^1})$ is a holonomic $\D_{n-1}$-module for $i = 0,1$.  Now let $j \in \{0, \ldots, n-1\}$ be given, and assume that we have found a change of variables after which $h^i(M \otimes \Omega^{\bullet}_{R^{\l}})$ is a holonomic and $x_{n-\l}$-regular $\D_{n-\l}$-module for all $i$ and all $\l \leq j$ (consequently, $h^i(M \otimes \Omega^{\bullet}_{R^{j+1}})$ is a holonomic $\D_{n-j-1}$-module for all $i$).  For every $i$, there exists (by Proposition \ref{fingenloc}) an element $g_i$ of $R_{n-j-1}$ such that the localization $h^i(M \otimes \Omega^{\bullet}_{R^{j+1}})[(g_i)^{-1}]$ is a finitely generated $R_{n-j-1}[(g_i)^{-1}]$-module.  Since there are only finitely many $g_i$, and the field $k$ is infinite, there is a single change of variables after which every $g_i$ is $x_{n-j-1}$-regular.  Moreover, since $g_i$ only involves the variables $x_1, \ldots, x_{n-j-1}$, this change of variables leaves $x_{n-j}, \ldots, x_n$ fixed.  

The complexes $M \otimes \Omega^{\bullet}_{R^{\l}}$ for $\l \leq j$ are defined using only derivations from the set $\{\partial_{n-j}, \ldots, \partial_n\}$, and these derivations are unaffected by the change of variables; hence the isomorphism classes of the cohomology objects $h^i(M \otimes \Omega^{\bullet}_{R^{\l}})$ do not change.  Furthermore, for every such $h^i(M \otimes \Omega^{\bullet}_{R^{\l}})$, there exists an $x_{n-\l}$-regular element $g$ of $R_{n-\l}$ such that $h^i(M \otimes \Omega^{\bullet}_{R^{\l}})[g^{-1}]$ is a finitely generated $R_{n-\l}[g^{-1}]$-module.  After the given change of variables, this localization is still finitely generated, and $g$ is still $x_{n-\l}$-regular: by the Weierstrass preparation theorem \cite[Thm. IV.9.2]{lang}, the $x_{n-\l}$-regularity of $g$ is equivalent to the existence of an expression of $g$ as a monic polynomial in $x_{n-\l}$ with coefficients in $R_{n-\l-1}$, and after a change of variables that fixes $x_{n-j}, \ldots, x_n$ and in which the new $x_1, \ldots, x_{n-j-1}$ are linear combinations of the old $x_1, \ldots, x_{n-j-1}$ only, such a polynomial is still monic in $x_{n-\l}$.  We conclude that the chosen change of variables does not invalidate the inductive hypothesis.  After this change of variables, $h^i(M \otimes \Omega^{\bullet}_{R^{j+1}})$ is $x_{n-j-1}$-regular for all $i$ by Lemma \ref{vdethesis}, and therefore $h^i(M \otimes \Omega^{\bullet}_{R^{j+2}})$ is a holonomic $\D_{n-j-2}$-module for all $i$.  This induction completes the proof: beginning with $M$ and repeating the inductive step $n$ times, we end up with a single change of variables after which, for all $i$ and $j$, $h^i(M \otimes \Omega^{\bullet}_{R^j})$ is a holonomic $\D_{n-j}$-module which is $x_{n-j}$-regular.
\end{proof}

Our next technical lemma makes crucial use of van den Essen's results.  The situation we examine is that of a $\D$-module expressible as a direct limit of $R$-modules, and a family of $R_{n-1}$-linear maps between these modules whose direct limit is the $R_{n-1}$-linear map $\partial_n$.  We now describe the situation more precisely.  Let $M$ be a $\D$-module, and suppose that $M = \varinjlim M_i$ as $R$-modules, where $\{M_i\}$ is a direct system (indexed by $\mathbb{N}$) of $R$-submodules of $M$.  Let $f_{ij}: M_i \rightarrow M_j$ (resp. $f_i: M_i \rightarrow M$) be the transition (resp. insertion) maps for this direct system.  Suppose furthermore that there exist $R_{n-1}$-linear maps $\delta_i: M_i \rightarrow M_{i+1}$ for all $i$, compatible with the transition maps, such that $\varinjlim \delta_i = \partial_n$.  (We will write $\partial$ for $\partial_n$.)  Given an element $m \in M$, there exists some $j$ and $m_j \in M_j$ such that $m = f_j(m_j)$, and we have $\partial(m) =  f_{j+1}(\delta_j(m_j))$, independently of the choice of $j$ and $m_j$.

\begin{example}\label{holex}
One example of the preceding situation, which we will consider in our proof of Proposition \ref{dualcoh1}, is the following: let $M = \D \cdot m$ be a holonomic $\D$-module, and for each $i \geq 0$, let $M_i = \D^i(R) \cdot m \subset M$, where $\D^i(R)$ denotes the $R$-submodule of $\D$ consisting of differential operators of order at most $i$ (so, for instance, $M_0 = R \cdot m$).  Then $M = \varinjlim M_i$ (the transition maps here being inclusions), and for every $i$, the restriction of $\partial$ to $M_i$ is an $R_{n-1}$-linear map $\delta_i: M_i \rightarrow M_{i+1}$.  Clearly $\partial = \varinjlim \delta_i$.
\end{example}

We return now to the general case.  The transition maps $f_{ij}$ induce, by the compatibility, maps on the kernels and cokernels of the $\delta_i$: for all $\l$ and $s$, we see that $f_{\l, \l+s}(\ker \delta_{\l}) \subset \ker \delta_{\l + s}$ and that $f_{\l}(\ker \delta_{\l}) \subset \ker \partial$.  From this, it follows that $f_{\l, \l+s}$ (resp. $f_{\l}$) induces $R_{n-1}$-linear maps $f^{*}_{\l, \l+s}: \ker \delta_{\l} \rightarrow \ker \delta_{\l+s}$ and $\overline{f}_{\l, \l+s}: \coker \delta_{\l-1} \rightarrow \coker \delta_{\l+s-1}$ (resp. $f^{*}_{\l}: \ker \delta_{\l} \rightarrow \ker \partial$ and $\overline{f}_{\l}: \coker \delta_{\l-1} \rightarrow \coker \partial$).  We have, for all $\l$ and $s$, $f^{*}_{\l+s} \circ f^{*}_{\l, \l+s} = f^{*}_{\l}$, and a similar compatibility for the $\overline{f}$.  We will use an overline to denote the class of an element modulo $\delta_{\l-1}$ (if the element belongs to $M_{\l}$) or modulo $\partial$ (if the element belongs to $M$).  Since filtered direct limits are exact functors \cite[Thm. 2.6.15]{weibel}, the direct limit of the $\ker \delta_{\l}$ with respect to the restricted transition maps $f^*_{\l,\l+s}$ is $M^* = \ker \partial$, and the direct limit of the $\coker \delta_{\l-1}$ with respect to the induced transition maps $\overline{f}_{\l,\l+s}$ is $\overline{M} = M/\partial(M)$.

If $M$ is holonomic and $x_n$-regular, we have a stability property for the images of these induced maps on kernels and cokernels (see Definition \ref{strongstable} for the notion of stability in the strong sense):

\begin{lem}\label{techlemma}
Let $M$ be a holonomic, $x_n$-regular $\D$-module.  Let $\partial = \partial_n \in \D$.  Suppose that $\{M_i\}$ is a direct system of $R$-submodules of $M$ with $M = \varinjlim M_i$, and that $\{\delta_i: M_i \rightarrow M_{i+1}\}$ is a family of $R_{n-1}$-linear maps, compatible with the transitions, such that $\partial = \varinjlim \delta_i$.  Let $f_{ij}: M_i \rightarrow M_j$ (resp. $f_i: M_i \rightarrow M$) be the transition (resp. insertion) maps for this direct system, and define the induced maps $f^{*}$ and $\overline{f}$ as in the previous paragraph.  Fix $\l$ and let $N_{\l}$ be a finitely generated $R$-submodule of $M_{\l}$.
\begin{enumerate}[(a)]
\item Let $N^{*}_{\l} = N_{\l} \cap \ker \delta_{\l}$.  The images of $N^{*}_{\l}$ under the $f^{*}_{\l,\l+s}$ stabilize in the strong sense, and the stable image is a finitely generated $R_{n-1}$-submodule of $M^{*} = \ker \partial$.
\item Let $\overline{N}_{\l}$ be the image $N_{\l}/(N_{\l} \cap \delta_{\l-1}(M_{\l-1}))$ of $N_{\l}$ in $\coker \delta_{\l-1}$.  The images of $\overline{N}_{\l}$ under the $\overline{f}_{\l,\l+s}$ stabilize in the strong sense, and the stable image is a finitely generated $R_{n-1}$-submodule of $\overline{M} = \coker \partial$.
\end{enumerate}
\end{lem}

\begin{proof}
Let $M^* = \ker \partial$.  The Leibniz rule implies that $R \cdot M^*$ is a $\D$-submodule of $M$, and it is clear that $M^* = (R \cdot M^*)^*$ as $\D_{n-1}$-submodules of $M$, where $(R \cdot M^*)^* = \ker(\partial: R \cdot M^* \rightarrow R \cdot M^*)$.  Moreover, for all $i$, $\ker \delta_i \subset f_i^{-1}(M^*)$ (since $\partial = \varinjlim \delta_i$), so the kernels of $\partial$ and of the $\delta_i$, as well as the maps between them, do not change if we replace $M$ with $M' = R \cdot M^*$ and all $M_i$ by $M'_i = f_i^{-1}(M')$.  Since $M' = \varinjlim M'_i$ (the transition maps being the restrictions of $f_{ij}$ to $M'_i$) and $\partial|_{M'} = \varinjlim \delta_i|_{M'_i}$, we reduce the proof of part (a) to the case $M = M'$.

By Lemma \ref{kernellem}, the hypothesis that $M = M'$ implies that there is an $R_{n-1}$-module direct sum decomposition $M = M^* \oplus x_n M$.  Let $\pi$ be the $R_{n-1}$-linear projection $M \rightarrow M^*$.  Since $f_{\l}(N^*_{\l}) \subset M^*$, $\pi$ restricts to an isomorphism $f_{\l}(N^*_{\l}) \xrightarrow{\sim} \pi(f_{\l}(N^*_{\l}))$, and the composite
\[
N^*_{\l} \hookrightarrow N_{\l} \xrightarrow{f_{\l}} M \xrightarrow{\pi} M^*
\] 
factors through the natural surjection $N_{\l} \rightarrow N_{\l}/x_n N_{\l}$ (since $\ker(\pi \circ f_{\l})$ contains $x_nN_{\l}$).  Since $N_{\l}$ is a finitely generated $R$-module, $N_{\l}/x_nN_{\l}$ is a finitely generated $R_{n-1}$-module.  This implies that $\pi(f_{\l}(N^*_{\l}))$ is contained in the $R_{n-1}$-linear image of the finitely generated $R_{n-1}$-module $N_{\l}/x_nN_{\l}$, so it is a finitely generated $R_{n-1}$-module itself, and so is its isomorphic copy $f_{\l}(N^*_{\l}) = f^*_{\l}(N^*_{\l})$.  It follows that \emph{if} the images of $N^*_{\l}$ stabilize in the strong sense, the stable image is finitely generated.

As $N_{\l}$ is a finitely generated $R$-module, its images under the $R$-linear maps $f_{\l,\l+s}$ stabilize in the strong sense by Lemma \ref{fgstable}; since $N^*_{\l} \subset N_{\l}$ and $f^*_{\l,\l+s}$ is simply a restriction of $f_{\l,\l+s}$ for all $s$, the fact that the images of $N^*_{\l}$ under the $f^*_{\l,\l+s}$ stabilize in the strong sense as well is automatic.  This proves (a).

In our proof of (b) we cannot assume (but do not need) that $M = M'$, so we drop this assumption now.  We now fix a set $\{n_1, \ldots, n_{\alpha_{\l}}\}$ of $R$-generators for $N_{\l}$.  By assumption, $M$ is $x_n$-regular, so by definition all $f_{\l}(n_i) \in M$ are $x_n$-regular.  We can therefore apply Lemma \ref{vdElemma} to every $f_{\l}(n_i)$ in turn. Let $i \in \{1, \ldots, \alpha_{\l}\}$ be fixed, and consider the $R$-submodule $R \cdot f_{\l}(n_i) \subset M$ generated by $f_{\l}(n_i)$.  By Lemma \ref{vdElemma}, there exist a positive natural number $p_i$ and a finitely generated $R_{n-1}$-submodule $L_i$ of $R \cdot f_{\l}(n_i)$ such that
\[
R \cdot f_{\l}(n_i) \subset L_i + \partial(\sum_{j=0}^{p_i - 1} R \cdot \partial^j(f_{\l}(n_i))).
\]  
If we let $L'_{\l}$ be the finitely generated $R_{n-1}$-module $L_1 + \cdots + L_{\alpha_{\l}}$ and $\Gamma_{\l}$ the $R$-submodule of $M$ generated by 
\[
\{\partial^j(f_{\l}(n_i))\}_{1 \leq i \leq \alpha_{\l}, 0 \leq j \leq p_i - 1},
\] 
we see that 
\[
f_{\l}(N_{\l}) = \sum_{i=1}^{\alpha_{\l}} R \cdot f_{\l}(n_i) \subset L'_{\l} + \partial(\Gamma_{\l}).
\]  
Let $\{f_{\l}(y_1), \ldots, f_{\l}(y_{\eta_{\l}})\}$ be a set of $R_{n-1}$-generators for $L'_{\l} \subset f_{\l}(N_{\l})$, and write $L''_{\l}$ for the $R_{n-1}$-submodule of $N_{\l}$ generated by $y_1, \ldots, y_{\eta_{\l}}$, so that $f_{\l}(L''_{\l}) = L'_{\l}$.  Then the containment $f_{\l}(N_{\l}) \subset f_{\l}(L''_{\l}) + \partial(\Gamma_{\l})$ implies that
\[
\overline{f}_{\l}(\overline{N}_{\l}) \subset \overline{f}_{\l}(L''_{\l}/(L''_{\l} \cap \delta_{\l-1}(M_{\l-1})))
\]
inside $M/\partial(M)$: given any class $\overline{n}_{\l} \in \overline{N}_{\l}$ of an element $n_{\l} \in N_{\l}$, we can write $f_{\l}(n_{\l}) = f_{\l}(\lambda_{\l}) + \partial(\gamma_{\l})$ for some $\lambda_{\l} \in L''_{\l}$ and $\gamma_{\l} \in \Gamma_{\l}$, so we have
\[
\overline{f}_{\l}(\overline{n}_{\l}) = \overline{f_{\l}(n_{\l})} = \overline{f_{\l}(\lambda_{\l}) + \partial(\gamma_{\l})} = \overline{f_{\l}(\lambda_{\l})} = \overline{f}_{\l}(\overline{\lambda}_{\l})
\]
since the class of $\partial(\gamma_{\l})$ in $\overline{M} = M/\partial(M)$ is zero (the first and last equalities are simply the definition of $\overline{f}_{\l}$).  We know that $\overline{f}_{\l}$ is $R_{n-1}$-linear and that $L''_{\l}$, and hence its quotient $L''_{\l}/(L''_{\l} \cap \delta_{\l-1}(M_{\l-1}))$, is a finitely generated $R_{n-1}$-module, so we can conclude that $\overline{f}_{\l}(\overline{N}_{\l})$ is a finitely generated $R_{n-1}$-module.  It follows that \emph{if} the images of $\overline{N}_{\l}$ stabilize in the strong sense, the stable image is finitely generated.

In order to conclude that the images of $\overline{N}_{\l}$ stabilize in the strong sense, we must prove that there exists $t$ with the property that whenever $f_{\l}(n_{\l}) \in \partial(M)$ for some $n_{\l} \in N_{\l}$, we have $f_{\l, \l+t}(n_{\l}) \in \delta_{\l+t-1}(M_{\l+t-1})$; in other words, if the class of $n_{\l}$ is carried to zero in $\coker \partial$, it is already zero in $\coker \delta_{\l+t-1}$.  (This implies that the images of $\overline{N}_{\l}$ stabilize at the $(\l+t)$th stage.) Choose $t$ so large that every $\partial^j(f_{\l}(n_i))$ has a representative in $M_{\l+t-1}$, as follows: for all $i$ and $j$, there exists a natural number $t_{ij}$ and an element $m_{\l + t_{ij}} \in M_{\l + t_{ij}}$ such that $f_{\l + t_{ij}}(m_{\l + t_{ij}}) = \partial^j(f_{\l}(n_i))$.  Put $t = \max{\{t_{ij}\}} + 1$.  Then since the $\partial^j(f_{\l}(n_i))$ are $R$-generators for $\Gamma_{\l}$, any element of the $R_{n-1}$-module $\partial(\Gamma_{\l})$ can be expressed as $f_{\l+t}(\delta_{\l+t-1}(m))$ for some $m \in M_{\l+t-1}$.  Consider also the $R_{n-1}$-submodule $L'_{\l} \cap \partial(M) \subset M$, which is finitely generated since $L'_{\l}$ is.  Enlarging $t$ if necessary, we may assume (by the finite generation) that every element of $L'_{\l} \cap \partial(M)$ can also be expressed as $f_{\l+t}(\delta_{\l+t-1}(m'))$ for some $m' \in M_{\l+t-1}$.

Now suppose that $n_{\l} \in N_{\l}$ is such that $f_{\l}(n_{\l}) \in \partial(M)$.  Since $f_{\l}(N_{\l}) \subset L'_{\l} + \partial(\Gamma_{\l})$, we have $f_{\l}(n_{\l}) \in (L'_{\l} \cap \partial(M)) + \partial(\Gamma_{\l})$, so $f_{\l}(n_{\l}) = f_{\l+t}(\delta_{\l+t-1}(m))$ for some $m \in M_{\l+t-1}$.  Since $f_{\l} = f_{\l+t} \circ f_{\l, \l+t}$, we see that 
\[
f_{\l, \l+t}(n_{\l}) - \delta_{\l+t-1}(m) \in \ker f_{\l+t}.
\]  
The kernel of the restriction of $f_{\l+t}$ to $f_{\l, \l+t}(N_{\l})$ is a finitely generated $R$-module, so by Lemma \ref{fgstable}, its images stabilize in the strong sense, and the stable image is zero because the image of this kernel in $M$ is clearly zero.  Therefore, enlarging $t$ again if necessary (so that the transition map from the old $t$ to the new $t$ annihilates this kernel; this new $t$ depends only on $\l$, not on $n_{\l}$), we may assume that the difference $f_{\l, \l+t}(n_{\l}) - \delta_{\l+t-1}(m)$ is zero; that is, that $f_{\l, \l+t}(n_{\l})$ belongs to $\delta_{\l+t-1}(M_{\l+t-1})$.  We conclude that the images of $\overline{N}_{\l}$ stabilize in the strong sense, completing the proof of (b) and the lemma.
\end{proof}

The strategy of the proof of Proposition \ref{dualcoh1} is to write the de Rham complex $M \otimes \Omega_R^{\bullet}$ of a holonomic $\D$-module $M$ as a direct limit of complexes whose objects are finitely generated $R$-modules (using the degree, or order, filtration on the ring $\D$).  Fix, once and for all, both a holonomic $\D$-module $M$ and an element $m \in M$ such that $M = \D \cdot m$.  As in Example \ref{holex}, let $M_{\l} = \D^{\l}(R) \cdot m$ for $\l \geq 0$.  Note that $M = \cup_{\l} M_{\l}$; for any $i$, $\partial_i: M \rightarrow M$ induces $k$-linear maps $\partial_i: M_{\l} \rightarrow M_{\l + 1}$ for all $\l$; and every $M_{\l}$ is a finitely generated $R$-module (a set of $R$-generators is given by $\{\delta m\}$ where $\delta$ runs through the monomials in $\partial_1, \ldots, \partial_n$ of total degree at most $\l$).

\begin{definition}\label{partialcplxs}
Let $M$ and $m$ be as above.  For all $j \in \{0, \ldots, n\}$ and $\l \in \mathbb{N}$, let $\mathcal{M}^{j, \bullet}_{\l}$ be the subcomplex
\[
0 \rightarrow \mathcal{M}^{j,0}_{\l} \rightarrow \mathcal{M}^{j,1}_{\l} \rightarrow \cdots \rightarrow \mathcal{M}^{j,j}_{\l} \rightarrow 0
\]
of $M \otimes \Omega^{\bullet}_{R^j}$ whose $i$th object $\mathcal{M}_{\l}^{j,i}$ is a direct sum of $j \choose i$ copies of the $R$-submodule $M_{\l+i}$ of $M$, indexed by $dx_{k_1} \wedge \cdots \wedge dx_{k_i}$ for $i$-tuples $n-j+1 \leq k_1 < \cdots < k_i \leq n$, and whose differentials are the restrictions of those in the complex $M \otimes \Omega^{\bullet}_{R^j}$.  (We simply write $\mathcal{M}^{\bullet}_{\l}$ for $\mathcal{M}^{n, \bullet}_{\l}$.)
\end{definition}

If we suppress the indexing wedge products, the complex $\mathcal{M}^{j, \bullet}_{\l}$ takes the form
\[
0 \rightarrow M_{\l} \rightarrow \oplus_{1 \leq i \leq n} M_{\l+1} \rightarrow \cdots \rightarrow M_{\l+j} \rightarrow 0;
\]
its objects are finitely generated $R$-modules and its differentials are $k$-linear.  We have short exact sequences of complexes
\[
0 \rightarrow \mathcal{M}^{j, \bullet}_{\l+1}[-1] \rightarrow \mathcal{M}^{j+1, \bullet}_{\l} \rightarrow \mathcal{M}^{j, \bullet}_{\l} \rightarrow 0
\]
for all $\l$, analogous to the sequence of Remark \ref{partialdRj}; the first nonzero morphism is given by $\wedge \, dx_{n-j}$ and, in the induced long exact cohomology sequence, the connecting homomorphisms are given by $\partial_{n-j}$, up to a sign.

The complexes $\mathcal{M}^{j, \bullet}_{\l}$ naturally form a filtered direct system as $\l$ varies.  We have $M \otimes \Omega^{\bullet}_{R^j} = \varinjlim \mathcal{M}^{j, \bullet}_{\l}$, and as filtered direct limits commute with cohomology (they are exact functors, by \cite[Thm. 2.6.15]{weibel}), this implies 
\[
h^i(M \otimes \Omega^{\bullet}_{R^j}) \simeq \varinjlim h^i(\mathcal{M}^{j, \bullet}_{\l})
\]
for all $i$.  In particular, taking $j = n$, we have $H^i_{dR}(M) = \varinjlim h^i(\mathcal{M}^{\bullet}_{\l})$ as $k$-spaces.  By Theorem \ref{findimdR}, the left-hand side is a finite-dimensional $k$-space for all $i$.  Thus, for all $\l$, the image of $h^i(\mathcal{M}^{\bullet}_{\l})$ in $H^i_{dR}(M)$ is a finite-dimensional $k$-space.  The key technical result in the proof of Proposition \ref{dualcoh1} is the following:

\begin{prop}\label{mitlef}
Let $M = \D \cdot m$ be a holonomic $\D$-module, and define the approximations $\mathcal{M}^{\bullet}_{\l}$ to its de Rham complex as above.  For all $i$ and $\l$, the images of $h^i(\mathcal{M}^{\bullet}_{\l})$ stabilize in the strong sense, with finite-dimensional stable image.
\end{prop}

\begin{proof}[Proof of Proposition \ref{mitlef}]
We may assume, after possibly making a change of variables as in Proposition \ref{coords}, that for all $i$ and $j$, $h^i(M \otimes \Omega^{\bullet}_{R^j})$ is a holonomic $\D_{n-j}$-module which is $x_{n-j}$-regular.  Let $(\ast_j)$ be the following statement: for all $i \in \{0, \ldots, j\}$ and all $\l \geq 0$, the images of $h^i(\mathcal{M}^{j, \bullet}_{\l})$ in the direct system
\[
\cdots \rightarrow h^i(\mathcal{M}^{j, \bullet}_{\l-1}) \rightarrow h^i(\mathcal{M}^{j, \bullet}_{\l}) \rightarrow h^i(\mathcal{M}^{j, \bullet}_{\l+1}) \rightarrow \cdots \label{cohods} \tag{*}
\] 
stabilize in the strong sense, and the stable image is a finitely generated $R_{n-j}$-submodule of $h^i(M \otimes \Omega^{\bullet}_{R^j})$.  The statement of our proposition is $(\ast_n)$, and we will prove $(\ast_j)$ for $j=1, \ldots, n$ by induction on $j$.

In the base case, $j=1$, the complex $M \otimes \Omega^{\bullet}_{R^1}$ takes the form $0 \rightarrow M \xrightarrow{\partial_n} M \rightarrow 0$, and its $\l$-th approximation $\mathcal{M}^{1, \bullet}_{\l}$ takes the form $0 \rightarrow M_{\l} \xrightarrow{\partial_n} M_{\l + 1} \rightarrow 0$.  By assumption, $M$ is $x_n$-regular, so we may apply Lemma \ref{techlemma} to the direct system $\{M_{\l}\}$: every $M_{\l}$ is already a finitely generated $R$-module, so we simply take $N_{\l} = M_{\l}$ in the statement of that lemma.  The statement $(\ast_1)$ follows at once.

Now suppose $j \geq 1$ and $(\ast_j)$ established.  By assumption, the $\D_{n-j}$-modules $h^i(M \otimes \Omega^{\bullet}_{R^j})$ are holonomic and $x_{n-j}$-regular.  By the inductive hypothesis $(\ast_j)$, the images of $h^i(\mathcal{M}^{j,\bullet}_{\l})$ (for any $i$ and $\l$) in the direct system (\ref{cohods}) stabilize in the strong sense, and the stable image is a finitely generated $R_{n-j}$-module; applying this reasoning to all $i$ at once, we see that given any $\l$, there exists $s$ such that, for all $i$, $h^i(\mathcal{M}^{j,\bullet}_{\l+s})$ contains the stable image of $h^i(\mathcal{M}^{j,\bullet}_{\l})$.  (To keep the notation as simple as possible, we will not always record the dependence on $i$ of various objects and maps.  At each stage, we assume that constructions are being carried out for all $i$ at once, and that indices large enough to work for all $i$ have been chosen.)  Fix $\l$ and let $N \subset h^i(\mathcal{M}^{j,\bullet}_{\l+s})$ be this stable image, which we can identify with an $R_{n-j}$-submodule of $h^i(M \otimes \Omega^{\bullet}_{R^j})$.

For all $t$, we have $R_{n-j-1}$-linear maps $\partial_{n-j}^t: h^i(\mathcal{M}^{j,\bullet}_t) \rightarrow h^i(\mathcal{M}^{j,\bullet}_{t+1})$, which are (up to a sign) the connecting homomorphisms in the long exact cohomology sequence associated with the short exact sequence of complexes
\[
0 \rightarrow \mathcal{M}^{j, \bullet}_{t+1}[-1] \rightarrow \mathcal{M}^{j+1, \bullet}_t \rightarrow \mathcal{M}^{j, \bullet}_t \rightarrow 0
\]
described earlier.  The direct limit $\varinjlim \partial_{n-j}^t$ is $\partial_{n-j}: h^i(M \otimes \Omega^{\bullet}_{R^j}) \rightarrow h^i(M \otimes \Omega^{\bullet}_{R^j})$.  The direct system (\ref{cohods}) induces a direct system
\[
\cdots \rightarrow \ker \partial_{n-j}^{\l+s-1} \rightarrow \ker \partial_{n-j}^{\l+s} \rightarrow \ker \partial_{n-j}^{\l+s+1} \rightarrow \cdots; \label{kerds} \tag{**}
\]
by Lemma \ref{techlemma}, the images of $N \cap \ker \partial_{n-j}^{\l+s}$ in the direct system (\ref{kerds}) stabilize in the strong sense, and the stable image is a finitely generated $R_{n-j-1}$-submodule of $\ker \partial_{n-j}$.  The image of $\ker \partial^{\l}_{n-j}$ in $\ker \partial^{\l+s}_{n-j}$ is $N \cap \ker \partial_{n-j}^{\l+s}$; consequently, the images of $\ker \partial^{\l}_{n-j}$ in the direct system (\ref{kerds}) also stabilize in the strong sense, and the stable image is a finitely generated $R_{n-j-1}$-module.  Enlarging $s$ if necessary, we may assume the images of $\ker \partial^{\l}_{n-j}$ stabilize at the $(\l+s)$th stage.  We fix this $s$ for the rest of the proof.

We now extract more information from the long exact cohomology sequences, which take the form
\[
\cdots \rightarrow h^{i-1}(\mathcal{M}^{j,\bullet}_t) \xrightarrow{\partial^t_{n-j}} h^{i-1}(\mathcal{M}^{j,\bullet}_{t+1}) \rightarrow h^i(\mathcal{M}^{j+1,\bullet}_t) \xrightarrow{\psi_t} h^i(\mathcal{M}^{j,\bullet}_t) \xrightarrow{\partial^t_{n-j}} h^i(\mathcal{M}^{j,\bullet}_{t+1}) \rightarrow \cdots;
\]
by the exactness, we have surjections
\[
h^i(\mathcal{M}^{j+1,\bullet}_t) \twoheadrightarrow \ker \partial^t_{n-j}
\]
for all $t$.  We also have commutative diagrams
\[
\begin{CD}
h^i(\mathcal{M}_{\l}^{j+1, \bullet}) @>> \psi_{\l} > h^i(\mathcal{M}_{\l}^{j,\bullet})\\
@VV \iota_t^{j+1} V                           @VV \iota_t^j V\\
h^i(\mathcal{M}_{\l+t}^{j+1, \bullet})  @>> \psi_{\l+t} > h^i(\mathcal{M}_{\l+t}^{j,\bullet})
\end{CD}
\]
for all $t$ (where the vertical arrow $\iota_t^j$ is the transition map from the $\l$th to the $(\l+t)$th stage in the direct system (\ref{cohods})), so the previous surjections induce surjections
\[
\iota_t^{j+1}(h^i(\mathcal{M}^{j+1,\bullet}_{\l})) \twoheadrightarrow \iota_t^j(\ker \partial^{\l}_{n-j})
\]
whose kernels we denote $K_{\l+t}$.  Consider, for any $t \geq s$, the commutative diagram
\[
\begin{CD}
0 @>>> K_{\l+s} @>>> \iota_s^{j+1}(h^i(\mathcal{M}^{j+1,\bullet}_{\l})) @>>> \iota_s^j(\ker \partial^{\l}_{n-j}) @>>> 0\\
@.          @VVV       @VVV          @VVV  @.\\
0 @>>> K_{\l+t} @>>> \iota_t^{j+1}(h^i(\mathcal{M}^{j+1,\bullet}_{\l})) @>>> \iota_t^j(\ker \partial^{\l}_{n-j}) @>>> 0
\end{CD}
\]
with exact rows.  The right vertical arrow is an isomorphism, since the images of $\ker \partial^{\l}_{n-j}$ in the direct system (\ref{kerds}) have already stabilized at the $(\l+s)$th stage, and the middle vertical arrow is a surjection, since $\iota_t^{j+1}$ factors through $\iota_s^{j+1}$ by the compatibility of transition and insertion maps in (\ref{cohods}).  It follows by the snake lemma that the left vertical arrow is also a surjection, so for $t \geq s$, the image of $K_{\l+s}$ in $K_{\l+t}$ is all of $K_{\l+t}$.

The direct system (\ref{cohods}) also induces a direct system
\[
\cdots \rightarrow \coker \partial_{n-j}^{\l+s-1} \rightarrow \coker \partial_{n-j}^{\l+s} \rightarrow \coker \partial_{n-j}^{\l+s+1} \rightarrow \cdots, \label{cokerds} \tag{***}
\]
and we claim next that the images of
\[
\coker(\partial_{n-j}^{\l+s}: h^{i-1}(\mathcal{M}^{j,\bullet}_{\l+s}) \rightarrow h^{i-1}(\mathcal{M}^{j,\bullet}_{\l+s+1}))
\] 
in the direct system (\ref{cokerds}) stabilize in the strong sense.  To this end, we repeat the reasoning we gave for the kernels of $\partial_{n-j}$ at the beginning of the proof, this time for cokernels.  We write $\overline{h^{i-1}(\mathcal{M}^{j,\bullet}_{\l+s+1})}$ for the cokernel displayed above, and more generally, if $S$ is an $R_{n-j}$-submodule of $h^{i-1}(\mathcal{M}^{j,\bullet}_{\l+t})$ for some $t$ (resp. $h^{i-1}(M \otimes \Omega_{R^j}^{\bullet})$), we write $\overline{S}$ for the image of $S$ in the cokernel of $\partial_{n-j}^{\l+t}$ (resp. $\partial_{n-j}$).  First, by the inductive hypothesis $(\ast_j)$, the images of $h^{i-1}(\mathcal{M}^{j,\bullet}_{\l+s+1})$ in the direct system (\ref{cohods}) stabilize in the strong sense, and the stable image is a finitely generated $R_{n-j}$-submodule $L$ of $h^{i-1}(M \otimes \Omega_{R^j}^{\bullet})$, realized as a submodule of $h^{i-1}(\mathcal{M}^{j,\bullet}_{\l+s+u+1})$ for sufficiently large $u$.  Consider the image $\overline{L}$ of $L$ in the quotient $\overline{h^{i-1}(\mathcal{M}^{j,\bullet}_{\l+s+u+1})}$.  Lemma \ref{techlemma} implies that the images of $\overline{L}$ in the direct system (\ref{cokerds}) stabilize in the strong sense, and the stable image is a finitely generated $R_{n-j-1}$-module.  The image of $\overline{h^{i-1}(\mathcal{M}^{j,\bullet}_{\l+s+1})}$ in $\overline{h^{i-1}(\mathcal{M}^{j,\bullet}_{\l+s+u+1})}$ is $\overline{L}$, so the images of $\overline{h^{i-1}(\mathcal{M}^{j,\bullet}_{\l+s+1})}$ in the direct system (\ref{cokerds}) also stabilize in the strong sense, and the stable image is a finitely generated $R_{n-j-1}$-submodule of $\overline{h^{i-1}(M \otimes \Omega_{R^j}^{\bullet})}$.  This stable image occurs at the $(\l+s+u+1+v)$th stage for some $v$: for simplicity, let $w = u+1+v$.  

We have $K_{\l+s} \subset \ker \psi_{\l+s}$, and by the exactness of the long cohomology sequence,
\[
\ker \psi_{\l+s} \simeq \coker \partial_{n-j}^{\l+s} \, (=\overline{h^{i-1}(\mathcal{M}^{j,\bullet}_{\l+s+1})})
\]
as $R_{n-j-1}$-modules.  These isomorphisms are functorial in $s$, so since the images of $\overline{h^{i-1}(\mathcal{M}^{j,\bullet}_{\l+s+1})}$ in the direct system (\ref{cokerds}) stabilize in the strong sense (with stable image a finitely generated $R_{n-j-1}$-submodule of $\overline{h^{i-1}(\mathcal{M}^{j,\bullet}_{\l+s+w})}$), we can conclude that the image of $\ker \psi_{\l+s}$ in $\ker \psi_{\l+s+w}$ is also a finitely generated $R_{n-j-1}$-module.  This module contains the image of $K_{\l+s}$ in $K_{\l+s+w}$; we have already seen that this image is all of $K_{\l+s+w}$, which implies that $K_{\l+s+w}$ is itself a finitely generated $R_{n-j-1}$-module.  We have a short exact sequence
\[
0 \rightarrow K_{\l+s+w} \rightarrow \iota_{s+w}^{j+1}(h^i(\mathcal{M}^{j+1,\bullet}_{\l})) \rightarrow \iota_{s+w}^j(\ker \partial^{\l}_{n-j}) \rightarrow 0
\]
where the left and right terms are finitely generated $R_{n-j-1}$-modules, so the middle term is also a finitely generated $R_{n-j-1}$-module.  Hence, the images of $h^i(\mathcal{M}^{j+1,\bullet}_{\l})$ in the direct system (\ref{cohods}) are eventually finitely generated $R_{n-j-1}$-modules, and it therefore follows by Lemma \ref{fgstable} that these images also stabilize in the strong sense.  This completes the proof of $(\ast_{j+1})$, and by induction, Proposition \ref{mitlef} follows.
\end{proof}

By reducing to the case of Proposition \ref{mitlef}, it is possible to draw the same conclusion about more general direct systems of complexes with the same limit.  We record this conclusion now; the reader is warned that the following statement will not be used until section \ref{cohom}.

\begin{cor}\label{fgstabdR}
Let $M$ be a holonomic $\D$-module.  Suppose that $\{N^{\bullet}_{\l}\}$ is a direct system of complexes with the following properties: the objects of the complexes are finitely generated $R$-modules, the differentials are $k$-linear, the transition maps $\lambda^{\bullet}_{\l,\l+s}$ are $R$-linear in each degree, $\varinjlim N_{\l}^{\bullet} \simeq M \otimes \Omega_R^{\bullet}$ in the category of complexes of $k$-spaces, and for all $i$, the isomorphism $\varinjlim N_{\l}^i \simeq M \otimes \Omega_R^i$ is an isomorphism of $R$-modules.  Then for all $\l$ and $i$, the images of $h^i(N_{\l}^{\bullet})$ stabilize in the strong sense, with finite-dimensional stable image.
\end{cor}

\begin{proof}
We note first that it is harmless to assume the isomorphism $\varinjlim N_{\l}^{\bullet} \simeq M \otimes \Omega_R^{\bullet}$ is an equality, and we do so.  Fix $\l$.  For all $i$, the images of $N_{\l}^i$ stabilize in the strong sense by Lemma \ref{fgstable}, since by assumption the $N_{\l}^i$ are finitely generated and the transition maps are $R$-linear.  Choose $s$ large enough that the transition maps $\lambda_{\l,\l+s}^i$ realize the stable image for all $i$ at once.  The stable images form a subcomplex $\lambda_{\l,\l+s}(N_{\l}^{\bullet})$ of $N_{\l+s}^{\bullet}$ which we can identify (via the insertion map, which we denote $\lambda_{\l+s}^{\bullet}$) with the subcomplex $\lambda_{\l}(N_{\l}^{\bullet})$ of $M \otimes \Omega_R^{\bullet}$.  Since every $N_{\l}^i$ is a finitely generated $R$-module and the insertion maps are $R$-linear, this is a subcomplex of $M \otimes \Omega_R^{\bullet}$ whose objects are finitely generated $R$-modules.

Now define $M_0 = \lambda_{\l}^0(N_{\l}^0)$, a finitely generated $R$-submodule of the holonomic $\D$-module $M$, and for all $j \geq 0$, let $M_j = \D^j(R) \cdot M_0$ where $\D^j(R)$ is the $R$-submodule of $\D$ consisting of differential operators of order at most $j$.  Define the complex $\mathcal{M}^{\bullet}_j$ in exactly the same way as in Definition \ref{partialcplxs}.  Since $\lambda_{\l}(N_{\l}^{\bullet})$ is a subcomplex of $M \otimes \Omega_R^{\bullet}$, we see that in fact $\lambda_{\l}(N_{\l}^{\bullet})$ is a subcomplex of $\mathcal{M}^{\bullet}_0$ (and hence of $\mathcal{M}^{\bullet}_j$ for all $j \geq 0$).  The proof of Proposition \ref{mitlef} still goes through assuming only that $M_0$ is a finitely generated $R$-submodule of the holonomic $\D$-module $M$ and that $M = \cup_j M_j$ (we do not need $M_0$ to be cyclic).  In our case, since $\varinjlim N_{\l}^{\bullet} = M \otimes \Omega_R^{\bullet}$, we also have $\varinjlim \mathcal{M}_j^{\bullet} = M \otimes \Omega_R^{\bullet}$, where this last direct limit is an ascending union of complexes.  (In particular, looking at the $0$th term of this complex, we see that the ascending union $\cup_j M_j$ is indeed the whole of $M$.)  Therefore we can invoke Proposition \ref{mitlef} (with this weaker hypothesis) to conclude that the images of $h^i(\mathcal{M}^{\bullet}_j)$ stabilize in the strong sense.  In particular, there exists $t$ large enough that for all $i$, the images of $h^i(\mathcal{M}_0^{\bullet})$ stabilize for $j \geq t$ with a finite-dimensional stable image.

The complex $\mathcal{M}^{\bullet}_t$ is still a complex of finitely generated $R$-modules, since $\D^t(R)$ is a finitely generated $R$-submodule of $\D$.  We now return to the stable image of the complex $N_{\l}^{\bullet}$, which is $\lambda_{\l,\l+s}^{\bullet}(N_{\l}^{\bullet})$.  Write $s_{\l}$ for $s$ to display its dependence on $\l$.  As this argument can be carried out for any $\l$, we obtain a direct system $\lambda_{\l,\l+s_{\l}}^{\bullet}(N_{\l}^{\bullet})$ of the stable images of $N_{\l}^{\bullet}$ as $\l$ varies.  We may assume the $s_{\l}$ have been chosen so that $\{\l+s_{\l}\}$ is strictly increasing; then the transition maps in this direct system are the restrictions, for all pairs $\l \leq \l'$, of $\lambda_{\l+s_{\l},\l'+s_{\l'}}^{\bullet}$.  Because the source and target complexes are complexes of stable images under the $\lambda$s, these restricted transition maps are injective, and the direct system so constructed also has $M \otimes \Omega_R^{\bullet}$ for its direct limit.  Thus $M \otimes \Omega_R^{\bullet}$ can be regarded as the ascending union of the complexes $\lambda_{\l',\l'+s_{\l'}}^{\bullet}(N_{\l'}^{\bullet})$, all of which have finitely generated $R$-modules as objects.  Any subcomplex of $M \otimes \Omega_R^{\bullet}$ whose objects are all finitely generated $R$-modules is thus a subcomplex of some $\lambda_{\l',\l'+s_{\l'}}^{\bullet}(N_{\l'}^{\bullet})$.  Let $\l'$ be so large that $\mathcal{M}^{\bullet}_t$ is a subcomplex of $\lambda_{\l',\l'+s_{\l'}}^{\bullet}(N_{\l'}^{\bullet})$, and consider the composite morphism of complexes
\[
N_{\l}^{\bullet} \xrightarrow{\lambda_{\l,\l+s_{\l}}^{\bullet}} N_{\l,\l+s_{\l}}^{\bullet} \xrightarrow{\lambda_{\l+s_{\l}}^{\bullet}} \mathcal{M}_0^{\bullet} \hookrightarrow \mathcal{M}_t^{\bullet} \hookrightarrow \lambda_{\l',\l'+s_{\l'}}^{\bullet}(N_{\l'}^{\bullet}) \subset N_{\l'+s_{\l'}}^{\bullet},
\]
in which all morphisms but possibly the first are injections.  Every step in this composite is a morphism of complexes, hence induces a morphism on cohomology.  If we choose $\l'' \leq \l'$ such that $\mathcal{M}^{\bullet}_0$ is a subcomplex of $\lambda_{\l'',\l''+s_{\l''}}^{\bullet}(N_{\l''}^{\bullet})$, the diagram
\[
\begin{CD}
\mathcal{M}^{\bullet}_0 @>>> \lambda_{\l'',\l''+s_{\l''}}^{\bullet}(N_{\l''}^{\bullet})\\
@VVV          @VVV\\
\mathcal{M}^{\bullet}_t @>>> \lambda_{\l',\l'+s_{\l'}}^{\bullet}(N_{\l'}^{\bullet})\\
\end{CD}
\]
is commutative and all arrows are injections, so we can regard the vertical inclusion $\mathcal{M}^{\bullet}_0 \hookrightarrow \mathcal{M}^{\bullet}_t$ as a restriction of $\lambda_{\l''+s_{\l''}, \l'+s_{\l'}}$.  This compatibility implies that the composite morphism above induces morphisms on cohomology through which the morphisms induced by $\lambda_{\l,\l'+s_{\l'}}^i$ factor for all $i$.  But we have seen already that partway through this composite morphism (at the $\mathcal{M}_0^{\bullet} \hookrightarrow \mathcal{M}_t^{\bullet}$ stage) the cohomology has attained a finite-dimensional stable image.  Therefore, for all $i$, the images of $h^i(N_{\l}^{\bullet})$ stabilize, with finite-dimensional stable image, at the $(\l' + s_{\l'})$th stage.
\end{proof}

We now return to our original case, and prove Proposition \ref{dualcoh1} (and hence Theorem \ref{dualcoh}). 

\begin{proof}[Proof of Proposition \ref{dualcoh1}]
For all $\l$, the differentials in the complex $\mathcal{M}^{\bullet}_{\l}$ are $\Sigma$-continuous, and therefore we can consider the Matlis dual $D(\mathcal{M}^{\bullet}_{\l})$, a complex whose $i$th object is the $R$-module $D(\mathcal{M}^{n-i}_{\l})$ (which is Artinian, since $\mathcal{M}^{n-i}_{\l}$ is finitely generated \cite[Thm. 18.6(v)]{matsumura}) and whose differentials are $k$-linear.  By Propositions \ref{artinian} and \ref{klindual}(a), the Matlis dual of this complex coincides with its $k$-linear dual.  Together with Proposition \ref{sigmabij}, this implies that for all $\l$, the complexes $\mathcal{M}^{\bullet}_{\l}$ and $(D(\mathcal{M}^{\bullet}_{\l}))^{\vee}$ are naturally isomorphic as complexes of $k$-spaces, where $\vee$ denotes $k$-linear dual.  Note that $\varprojlim D(\mathcal{M}_{\l}^{\bullet}) \simeq D(M \otimes \Omega_R^{\bullet})$ as complexes of $k$-spaces, by Remark \ref{cofinal}.

For every $i$, the inverse system $\{D(\mathcal{M}^i_{\l})\}_{\l}$ of $R$-modules (\emph{a fortiori}, of $k$-spaces) satisfies the Mittag-Leffler condition, since the transition maps are surjective (they are the Matlis duals of the $R$-linear inclusions $\mathcal{M}^i_{\l} \hookrightarrow \mathcal{M}^i_{\l+1}$).  Furthermore, the inverse system $\{h^i(D(\mathcal{M}^{\bullet}_{\l}))\}_{\l}$ of $k$-spaces also satisfies the Mittag-Leffler condition.  To see this, note that for all $i$, we have 
\[
h^{n-i}(\mathcal{M}^{\bullet}_{\l}) \simeq h^{n-i}(D(\mathcal{M}^{\bullet}_{\l})^{\vee}) \simeq (h^i(D(\mathcal{M}^{\bullet}_{\l})))^{\vee};
\]
the first isomorphism holds because $\mathcal{M}^{\bullet}_{\l} \simeq (D(\mathcal{M}^{\bullet}_{\l}))^{\vee}$ as complexes of $k$-spaces, and the second holds because $k$-linear dual is an exact contravariant functor.  By Proposition \ref{mitlef}, the images of $h^{n-i}(\mathcal{M}^{\bullet}_{\l}) \simeq (h^i(D(\mathcal{M}^{\bullet}_{\l})))^{\vee}$ stabilize in the strong sense, so by Lemma \ref{dualml}, the original system $\{h^i(D(\mathcal{M}^{\bullet}_{\l}))\}$ satisfies Mittag-Leffler.  The Mittag-Leffler conditions for $\{D(\mathcal{M}^i_{\l})\}$ and $\{h^i(D(\mathcal{M}^{\bullet}_{\l}))\}$ allow us to apply Proposition \ref{mlcohcomm}, which implies that there are isomorphisms
\[
h^i(\varprojlim D(\mathcal{M}^{\bullet}_{\l})) \xrightarrow{\sim} \varprojlim h^i(D(\mathcal{M}^{\bullet}_{\l})).
\]
of $k$-spaces.  Therefore, we have a chain of isomorphisms
\begin{align}
h^i(M \otimes \Omega_R^{\bullet}) &\simeq (h^i(M \otimes \Omega_R^{\bullet}))^{\vee \vee}\\
&\simeq (h^i(\varinjlim \mathcal{M}^{\bullet}_{\l}))^{\vee \vee}\\
&\simeq (\varinjlim h^i(\mathcal{M}_{\l}^{\bullet}))^{\vee \vee}\\
&\simeq (\varprojlim (h^i(\mathcal{M}_{\l}^{\bullet}))^{\vee})^{\vee}\\
&\simeq (\varprojlim (h^i((D(\mathcal{M}_{\l}^{\bullet}))^{\vee}))^{\vee})^{\vee}\\
&\simeq (\varprojlim (h^{n-i}(D(\mathcal{M}_{\l}^{\bullet})))^{\vee \vee})^{\vee}\\
&\simeq (\varprojlim h^{n-i}(D(\mathcal{M}_{\l}^{\bullet})))^{\vee}\\
&\simeq (h^{n-i}(\varprojlim D(\mathcal{M}_{\l}^{\bullet})))^{\vee}\\
&\simeq (h^{n-i}(D(M \otimes \Omega_R^{\bullet})))^{\vee},
\end{align}
where (1) holds because $h^i(M \otimes \Omega_R^{\bullet})$ is a finite-dimensional $k$-space, (3) since $\varinjlim$ is exact and thus commutes with cohomology, (4) because taking $k$-dual converts direct limits into inverse limits \cite[Prop. 5.26]{rotman}, (5) since $\mathcal{M}^{\bullet}_{\l}$ and $(D(\mathcal{M}_{\l}^{\bullet}))^{\vee}$ are isomorphic complexes of $k$-spaces, (6) since $k$-dual is an exact contravariant functor, (7) by Lemma \ref{finddual} applied to the inverse system $\{h^i(D(\mathcal{M}^{\bullet}_{\l}))\}$, (8) by Proposition \ref{mlcohcomm}, and (9) by Remark \ref{cofinal}.  Since $h^i(M \otimes \Omega_R^{\bullet})$ is a finite-dimensional $k$-space, all of the $k$-spaces appearing in the chain of isomorphisms are finite-dimensional as well.  Therefore $h^{n-i}(D(M \otimes \Omega_R^{\bullet}))$ is canonically isomorphic to its own double dual, so we can dualize the isomorphism obtained above to find $h^{n-i}(D(M \otimes \Omega_R^{\bullet})) \simeq (h^i(M \otimes \Omega_R^{\bullet}))^{\vee}$, as desired.
\end{proof}

\section{Local cohomology of formal schemes}\label{formal}

In this section, we recall the description of the Matlis duals of local cohomology modules (over any local Gorenstein ring) in terms of local cohomology on a formal scheme.  (A reference for this description is Ogus's thesis \cite{ogus}.)  Specializing to the case of a complete local ring with a coefficient field, we obtain a \emph{right} $\D$-module structure on the local cohomology of a formal scheme by applying the results of section \ref{dmod} to dualize the left $\D$-module structure on ordinary local cohomology.  We then recall a natural \emph{left} $\D$-module structure on the local cohomology of the formal scheme: the structure implicitly used to define its de Rham complex in \cite{derham}.  The main result of this section is that in the case where the complete local ring is regular and of characteristic zero, these left and right $\D$-module structures are \emph{transposes} of each other, so our theory and that of \cite{derham} are compatible.

Let $(R, \mathfrak{m})$ be a Gorenstein local ring.  The key ingredient in the identification of the Matlis duals of local cohomology modules over $R$ with local cohomology of a formal scheme is Grothendieck's \emph{local duality theorem}:

\begin{thm}\cite[Thm. 6.3]{hartshorneLC}\label{groth}
Let $(R, \mathfrak{m})$ be a Gorenstein local ring of dimension $n$, $E$ an injective hull of its residue field, and $D(-) = \Hom_R(-,E)$ the Matlis dual functor for $R$-modules.  If $M$ is a finitely generated $R$-module, there are isomorphisms $H^{n-i}_{\mathfrak{m}}(M) \simeq D(\Ext^i_R(M,R))$, for all $i$, that are functorial in $M$.
\end{thm}

\begin{remark}
If moreover $R$ is complete, we also have $\Ext^{n-i}_R(M,R) \simeq D(H^i_{\mathfrak{m}}(M))$, since the double Matlis dual and the identity functor are naturally isomorphic when restricted to the full subcategories of finitely generated or Artinian $R$-modules \cite[Thm. 18.6(v)]{matsumura}.
\end{remark}

Now recall the definition of local cohomology as a direct limit of Ext modules: for any module $M$ over any commutative Noetherian ring $R$, and any ideal $I \subset R$, we have isomorphisms $H^i_I(M) \simeq \varinjlim \Ext^i_R(R/I^t, M)$ \cite[Thm. 1.3.8]{brodmann}.  Now put $M = R$, a Gorenstein local ring of dimension $n$ with maximal ideal $\mathfrak{m}$.  Taking Matlis duals (and using the fact that any contravariant Hom functor converts direct limits into inverse limits) we have isomorphisms
\[
D(H^i_I(R)) \simeq D(\varinjlim \Ext^i_R(R/I^t, R)) \simeq \varprojlim D(\Ext^i_R(R/I^t, R)) \simeq \varprojlim H^{n-i}_{\mathfrak{m}}(R/I^t),
\]
where in the last step we have used local duality for every $t$ and passed to the inverse limit.  If we let $X = \Spec(R)$ and $Y$ the closed subscheme $V(I) \subset X$ defined by $I$, and if we write $\widehat{X}$ for the formal completion of $Y$ in $X$ (see sections \ref{intro} and \ref{cohom} for the definition of formal completion), $P$ for its closed point, and $\mathcal{O}_{\widehat{X}}$ for its structure sheaf, then the last object in the sequence of isomorphisms above is precisely the local cohomology $H^{n-i}_P(\widehat{X}, \mathcal{O}_{\widehat{X}})$ supported at the closed point \cite[Prop. 2.2]{ogus}.  (We may also write $X_{/Y}$ and $\mathcal{O}_{X_{/Y}}$ for $\widehat{X}$ and $\mathcal{O}_{\widehat{X}}$ when we need to record explicitly the dependence on $Y$.) It follows that for all $i$, the $R$-modules $D(H^i_I(R))$ and $H^{n-i}_P(\widehat{X}, \mathcal{O}_{\widehat{X}})$ are isomorphic.  We record this conclusion separately for future reference:

\begin{prop}\cite[Prop. 2.2]{ogus}\label{ogusiso}
Let $R$ be a Gorenstein local ring of dimension $n$ with maximal ideal $\mathfrak{m}$.  For all $i$, we have isomorphisms
\[
D(H^i_I(R)) \simeq \varprojlim H^{n-i}_{\mathfrak{m}}(R/I^t) \simeq H^{n-i}_P(\widehat{X}, \mathcal{O}_{\widehat{X}})
\]
of $R$-modules, where the rightmost object is the formal local cohomology defined in the previous paragraph.
\end{prop}

For the remainder of the section, we identify the $R$-modules $D(H^i_I(R))$ and $H^{n-i}_P(\widehat{X}, \mathcal{O}_{\widehat{X}})$ using the proposition above, suppressing any explicit mention of this isomorphism.

Now suppose that $(R, \mathfrak{m})$ is a \emph{complete} local ring with coefficient field $k$.  By Example \ref{farrago}, we know that if $I \subset R$ is an ideal, every $H^i_I(R)$ is a left $\D(R,k)$-module, and since $R$ is complete, $D(H^i_I(R))$ is a right $\D(R,k)$-module by Corollary \ref{dualright}.  If moreover $R$ is Gorenstein, we have $D(H^i_I(R)) = H^{n-i}_P(\widehat{X}, \mathcal{O}_{\widehat{X}})$ by Proposition \ref{ogusiso}, defining by transport of structure a right $\D(R,k)$-module structure on $H^{n-i}_P(\widehat{X}, \mathcal{O}_{\widehat{X}})$.

There is also a natural \emph{left} $\D(R,k)$-module structure on $H^{n-i}_P(\widehat{X}, \mathcal{O}_{\widehat{X}})$, defined without Matlis duality: see \cite[\S I.7 and \S III.2]{derham} and \cite[\S 7.2]{hellus}.  Let $d \in \D(R,k)$ be a differential operator of order $j$.  For all $t \geq 0$, $d(I^{t+j}) \subset I^t$, and so $d$ induces a $k$-linear map $d_t: R/I^{t+j} \rightarrow R/I^t$.  If $\mathcal{I} \subset \mathcal{O}_X$ is the sheaf of ideals on $X = \Spec(R)$ corresponding to $I$, the map $d_t$ induces a sheafified map $\tilde{d}_t: \mathcal{O}_X/\mathcal{I}^{t+j} \rightarrow \mathcal{O}_X/\mathcal{I}^t$ of sheaves of $k$-spaces on $X$; taking local cohomology, we obtain $k$-linear maps $\overline{d}_t: H^{n-i}_{\mathfrak{m}}(R/I^{t+j}) \rightarrow H^{n-i}_{\mathfrak{m}}(R/I^t)$, where we have identified $H^{n-i}_{\mathfrak{m}}(R/I^{\l})$ with $H^{n-i}_P(X, \mathcal{O}_X/\mathcal{I}^{\l})$ as $R$-modules for all $\l$ \cite[Exp. II, Cor. 4]{SGA2}.  The reason that sheaves were introduced here is to make clear that merely $k$-linear (or even additive) maps between $R$-modules still induce maps on local cohomology, because local cohomology, in its most general form, is a functor on sheaves of Abelian groups on a topological space \cite[Exp. I, D\'{e}f. 2.1]{SGA2}.

These maps $\overline{d}_t$ are compatible with the natural surjections $R/I^{t+\l} \rightarrow R/I^t$, in the sense that we have commutative diagrams
\[
\begin{CD}
H^{n-i}_{\mathfrak{m}}(R/I^{t+j+\l}) @>>\overline{d}_{t+\l}> H^{n-i}_{\mathfrak{m}}(R/I^{t+\l})\\
@VVV   @VVV \\
H^{n-i}_{\mathfrak{m}}(R/I^{t+j}) @>>\overline{d}_t> H^{n-i}_{\mathfrak{m}}(R/I^{t})\\
\end{CD}
\]
for all $t$, $j$, and $\l$, where the vertical arrows are the $R$-linear maps induced by the natural surjections.  Therefore, upon passing to inverse limits, the $\overline{d}_t$ define a $k$-linear map $\overline{d}: \varprojlim H^{n-i}_{\mathfrak{m}}(R/I^t) \rightarrow \varprojlim H^{n-i}_{\mathfrak{m}}(R/I^t)$.  If $\gamma \in \varprojlim H^{n-i}_{\mathfrak{m}}(R/I^t) = H^{n-i}_P(\widehat{X}, \mathcal{O}_{\widehat{X}})$, we set $d \cdot \gamma = \overline{d}(\gamma)$, and in this way define a left action of $\D(R,k)$ on $H^{n-i}_P(\widehat{X}, \mathcal{O}_{\widehat{X}})$.

\begin{definition}\label{twoduals}
If $(R, \mathfrak{m})$ is a complete Gorenstein local ring of dimension $n$ with coefficient field $k$ and $I \subset R$ is an ideal, the \emph{Matlis dual action} of $\D = \D(R,k)$ on $D(H^i_I(R))$ for any $i$ is the \emph{right} action defined by dualizing the natural structure of left $\D$-module on $H^i_I(R)$, and the \emph{inverse limit action} of $\D$ on $D(H^i_I(R)) = H^{n-i}_P(\widehat{X}, \mathcal{O}_{\widehat{X}})$ is the \emph{left} action defined in the previous paragraph.
\end{definition}

In the case of a characteristic-zero complete \emph{regular} local ring $(R, \mathfrak{m})$ containing its residue field $k$, we can state precisely how these two $\D$-module structures are related: they are \emph{transposes} of each other, as in Definition \ref{transpose}.  This is the main result of this section.

\begin{thm}\label{loccohtranspose}
Let $k$ be a field of characteristic zero, let $R = k[[x_1, \cdots, x_n]]$, and let $I \subset R$ be an ideal.  Denote by $\mathfrak{m}$ the maximal ideal of $R$, and by $\D = \D(R,k)$ the ring of $k$-linear differential operators on $R$.  Then for all $i$, the Matlis dual action of $\D$ on $D(H^i_I(R)) = H^{n-i}_P(\widehat{X}, \mathcal{O}_{\widehat{X}})$ is the \emph{transpose} of the inverse limit action.
\end{thm}

\begin{remark}\label{E1duals}
It follows from Theorem \ref{loccohtranspose} that the identification of Proposition \ref{ogusiso} can be extended to an identification of left $\D$-modules, regarding $D(H^i_I(R))$ as a \emph{left} $\D$-module by transposing the Matlis dual action as in Proposition \ref{regleft}.  We can therefore identify the de Rham complexes of both sides as well.
\end{remark}

\begin{proof}
Every element of $\D$ is a finite sum of terms of the form $\rho \partial_1^{a_1} \cdots \partial_n^{a_n}$ where $\rho \in R$, and Matlis duality respects composition of operators (reversing the order), so we need only check the statement of the theorem for the action of a single $\rho$ or $\partial_i$.  There is nothing to prove in the case of an element $\rho \in R$, since the Matlis dual of the $R$-linear multiplication by such an element $\rho$ is again multiplication by $\rho$.  Therefore, to prove Theorem \ref{loccohtranspose}, we need only to show that if $i$ is fixed ($1 \leq i \leq n$), the Matlis dual and inverse limit actions of the partial derivative $\partial_i \in \D$ on $D(H^i_I(R)) = H^{n-i}_P(\widehat{X}, \mathcal{O}_{\widehat{X}})$ differ by a sign (transposing the action of $\partial_i$ introduces a sign change, by Definition \ref{transpose}).  Without loss of generality, we may assume $i = 1$ and write $x$ for $x_1$, so that $R = k[[x, x_2, \ldots, x_n]]$ and $\partial = \partial_1$.  We denote the inverse limit action of $\partial$ on $D(H^i_I(R)) = H^{n-i}_P(\widehat{X}, \mathcal{O}_{\widehat{X}})$ by $\overline{\delta}$ and the Matlis dual action by $\partial^*$.  In order to prove the theorem, it is enough to prove the equality
\[
\overline{\delta} = -\partial^*.
\]

We will prove this equality first for a complete intersection ideal $I$ and then in general.  We prove the complete intersection case by induction on the number of generators of $I$.  Suppose that $f_1, \ldots, f_s$ is a regular sequence in $R$, and assume the theorem has been proved for complete intersection ideals with $s-1$ generators: the base case, $s = 0$, has already been established as Example \ref{singlevar}, since $H^0_{(0)}(R) = R$ and $H^i_{(0)}(R) = 0$ for $i > 0$.  We write $f$ for $f_s$. Let $I$ be the ideal $(f_1, \ldots, f_s)$ and $J \subset I$ the ideal $(f_1, \ldots, f_{s-1})$, so that $I = J + (f)$.  Let $Y$ (resp. $Z$) be the closed subscheme of $X = \Spec(R)$ defined by $I$ (resp. $J$), and let $X_t = \Spec(R/(f^t))$ for all $t \geq 1$.  We claim that the Matlis dual and inverse limit actions of $\partial$ on $D(H^i_I(R)) = H^{n-i}_P(X_{/Y}, \mathcal{O}_{X_{/Y}})$ differ by a sign for all $i$.  Since $I$ is a complete intersection ideal, $H^i_I(R) = 0$ unless $i = s$, so this is the only case we need to consider.  

Let $M = H^{s-1}_J(R)$.  Since $f_1, \ldots, f_s$ is a regular sequence, the composite local cohomology spectral sequence of Example \ref{lccomposite} (with respect to the ideals $J$ and $(f)$) degenerates at $E_2$, and consequently we have $H^s_I(R) \simeq H^1_{(f)}(M)$ as $R$-modules (indeed, as left $\D$-modules).  We write $\partial_M$ for the action of $\partial$ on the left $\D$-module $M$, $\partial_M^*$ for the corresponding Matlis dual action on $D(M)$, and $\overline{\delta}_Z$ for the inverse limit action of $\partial$ on $D(M) = H^{n-s+1}_P(X_{/Z}, \mathcal{O}_{X_{/Z}})$ (note that our induction hypothesis is that $\overline{\delta}_Z = -\partial_M^*$).

By the  \v{C}ech complex definition of local cohomology, $H^1_{(f)}(M)$ is the cokernel of the localization map $M \rightarrow M_f$.  Since $f_1, \ldots, f_s$ is a regular sequence, this localization map is injective: identifying $M$ with its image, we view $M$ as a submodule of $M_f$ and write $M_f/M$ for the cokernel, which we can express as the direct limit $\varinjlim M/f^tM$. Here, the transition map $M/f^tM \rightarrow M/f^{t+1}M$ carries the class of $m \in M$ to the class of $fm$, and the $R$-linear isomorphism $M_f/M \rightarrow \varinjlim M/f^tM$ carries the class of $\frac{m}{f^t} \in M_f$ to the class of $m \in M$ modulo $f^tM$.  The action of $\partial$ on the left $\D$-module $M$ induces an action of $\partial$ on $M_f/M$, defined by the quotient rule: we have a $k$-linear map $(\partial_M)_f: M_f \rightarrow M_f$ given by $(\partial_M)_f(\frac{m}{f^t}) = \frac{f \partial_M(m) - t \partial(f) m}{f^{t+1}}$ that carries $M \subset M_f$ into itself and therefore descends to the quotient $M_f/M$.  In terms of the description of $M_f/M$ as the direct limit $\varinjlim M/f^tM$, the map $(\partial_M)_f$ is defined by the direct limit of the $k$-linear maps $\partial_t: M/f^tM \rightarrow M/f^{t+1}M$ where 
\[
\partial_t(\mu) = (f \partial_M(m) - t \partial(f) m) + f^{t+1}M \in M/f^{t+1}M
\]
if $\mu$ is the class of $m \in M$ in $M/f^tM$.  It is straightforward to check that $\partial_t(\mu)$ is well-defined (that is, depends only on $\mu$ and not on $m$) and that $\partial_t$ satisfies the Leibniz rule: $\partial_t(r\mu)=\partial(r)\overline{\mu} + r\partial_t(\mu)$ for $r \in R$, where $\overline{\mu}$ is the image of $\mu$ under the transition map $M/f^tM \rightarrow M/f^{t+1}M$.

The maps $\partial_t$ defined in the previous paragraph fit into commutative diagrams
\[
\begin{CD}
0 @>>> M @>>f^t> M @>>> M/f^tM @>>> 0\\
@.  @VV \partial_M V @VV f\partial_M - t\partial(f)V @VV \partial_t V @. \\
0 @>>> M @>>f^{t+1}> M @>>> M/f^{t+1}M @>>> 0\\
\end{CD}
\]
for all $t$, where the rows are exact sequences of $R$-modules and the vertical arrows are $k$-linear.  The commutativity of the right square is clear, while the calculation
\[
(f\partial_M - t\partial(f))(f^t m) = f\partial_M(f^t m) - t\partial(f)f^t m = f(tf^{t-1}\partial(f) m + f^t \partial_M(m)) - t\partial(f)f^t m = f^{t+1}\partial_M(m)
\]
for $m \in M$ shows that the left square is commutative.  Fix $t$ and consider the Matlis dual of the above diagram, which takes the form
\begin{equation}\label{onestar}
\begin{CD}
0 @>>> D(M/f^{t+1}M) @>>> D(M) @>>f^{t+1}> D(M) @>>> 0\\
@.  @VV \partial_t^* V   @VV (f\partial_M-t\partial(f))^*V   @VV \partial_M^*V   @. \\
0 @>>> D(M/f^tM) @>>> D(M) @>>f^t> D(M) @>>> 0\\
\end{CD}\tag{*}
\end{equation}
Since $H^1_{(f)}(M)$ is the direct limit $\varinjlim M/f^tM$, the Matlis dual $D(H^1_{(f)}(M))$ is the \emph{inverse} limit $\varprojlim D(M/f^tM)$, and by Remark \ref{cofinal}, the Matlis dual action of $\partial$ on $D(H^1_{(f)}(M))$ is defined by $\varprojlim \partial_t^*$ (the Leibniz rule implies that all $\partial_t$ are $\Sigma$-continuous).   

We turn next to the inverse limit action of $\partial$ on $D(H^s_I(R)) = H^{n-s}_P(X_{/Y}, \mathcal{O}_{X_{/Y}})$.  Again we let $M = H^{s-1}_J(R)$.  From the long exact sequence of local cohomology supported at $J$ applied to the short exact sequence
\[
0 \rightarrow R \xrightarrow{f^t} R \rightarrow R/(f^t) \rightarrow 0
\]
of $R$-modules, it follows at once that $M/f^tM \simeq H^{s-1}_J(R/(f^t))$ as $R$-modules. By the change of ring principle \cite[Thm. 4.2.1]{brodmann}, we have $H^{s-1}_J(R/(f^t)) \simeq H^{s-1}_{\overline{J}}(R/(f^t))$ as $R$-modules, where $\overline{J} = (J + (f^t))/(f^t) \subset R/(f^t)$, the local cohomology module on the right-hand side is computed in the category of $R/(f^t)$-modules, and its $R$-module structure is defined using the natural surjection $R \rightarrow R/(f^t)$.  Therefore $D(M/f^tM) \simeq D(H^{s-1}_{\overline{J}}(R/(f^t)))$ as $R$-modules.  Since $M/f^tM$ (and hence $H^{s-1}_{\overline{J}}(R/(f^t))$) is annihilated by $f^t$, it does not matter whether the Matlis dual on the right-hand side is computed over the ring $R$ or $R/(f^t)$.  Taking the second point of view, we can apply Proposition \ref{ogusiso} to the $(n-1)$-dimensional Gorenstein local ring $R/(f^t)$, obtaining an isomorphism
\[
D(M/f^tM) \simeq D(H^{s-1}_{\overline{J}}(R/(f^t))) \simeq H^{n-s}_P((X_t)_{/Z}, \mathcal{O}_{(X_t)_{/Z}})
\]
of $R$-modules (the cohomological degree is $n-s = (n-1) - (s-1) = \dim(R/(f^t)) - (s-1)$). Here, we have abusively written $Z$ for the closed subscheme of $X_t$ defined by $\overline{J} = (J + (f^t))/(f^t)$. We already know the inverse limit $\varprojlim D(M/f^tM)$ is isomorphic as an $R$-module to 
\[
D(H^1_{(f)}(M)) \simeq D(H^s_I(R)) = H^{n-s}_P(X_{/Y}, \mathcal{O}_{X_{/Y}}),
\]
from which it follows (by passing to the inverse limit in $t$) that we must have
\[
\varprojlim H^{n-s}_P((X_t)_{/Z}, \mathcal{O}_{(X_t)_{/Z}}) \simeq H^{n-s}_P(X_{/Y}, \mathcal{O}_{X_{/Y}})
\]
as $R$-modules. In order to see how this last isomorphism interacts with the inverse limit action of $\partial$, we describe it more explicitly.  For all $t$, we have
\[
H^{n-s}_P((X_t)_{/Z}, \mathcal{O}_{(X_t)_{/Z}}) \simeq \varprojlim_{\l} H^{n-s}_{\mathfrak{m}}((R/(f^t))/\overline{J}^{\l}) \simeq \varprojlim_{\l} H^{n-s}_{\mathfrak{m}}(R/(J^{\l} + (f^t)))
\]
as $R$-modules. Passing to the inverse limit in $t$, we obtain
\[
\varprojlim_t H^{n-s}_P((X_t)_{/Z}, \mathcal{O}_{(X_t)_{/Z}}) \simeq \varprojlim_t \varprojlim_{\l} H^{n-s}_{\mathfrak{m}}(R/(J^{\l} + (f^t))).
\]
As the family $\{I^{\l}\} = \{(J + (f))^{\l}\}$ of ideals of $R$ is cofinal with the family $\{J^{\l} + (f^t)\}_{(\l, t)}$, this last inverse limit is isomorphic to
\[
\varprojlim_{\l} H^{n-s}_{\mathfrak{m}}(R/(J + (f))^{\l}) = H^{n-s}_P(X_{/Y}, \mathcal{O}_{X_{/Y}}),
\]
as claimed.  For all $t$ and $\l$, there are $k$-linear maps $\delta_{\l, t}: R/(J^{\l+1} + (f^{t+1})) \rightarrow R/(J^{\l} + (f^t))$ and $\delta_{\l}: R/J^{\l+1} \rightarrow R/J^{\l}$, induced on the quotients by $\partial: R \rightarrow R$, that induce $k$-linear maps
\[
\overline{\delta}_{\l,t}: H^{n-s}_{\mathfrak{m}}(R/(J^{\l+1} + (f^{t+1}))) \rightarrow H^{n-s}_{\mathfrak{m}}(R/(J^{\l} + (f^t)))
\]
and
\[
\overline{\delta}_{\l}: H^{n-s}_{\mathfrak{m}}(R/J^{\l+1}) \rightarrow H^{n-s}_{\mathfrak{m}}(R/J^{\l})
\]
on local cohomology, by viewing local cohomology as a functor on sheaves of $k$-spaces on the topological space $\Spec(R)$.  (Here we write, for instance, $\delta_{\l}$ rather than $\partial_{\l}$ to avoid confusion with the maps $\partial_{\l}$ defined earlier in the proof.) The inverse limit action of $\partial$ on $H^{n-s}_P(X_{/Y}, \mathcal{O}_{X_{/Y}})$, which we have denoted $\overline{\delta}$, is defined by $\varprojlim_{(\l, t)} \overline{\delta}_{\l, t}$.  If we pass to the inverse limit in $\l$ first, we find that $\overline{\delta} = \varprojlim_t \overline{\delta}_{Z,t}$, where for all $t$, the map
\[
\overline{\delta}_{Z,t}: H^{n-s}_P((X_{t+1})_{/Z}, \mathcal{O}_{(X_{t+1})_{/Z}}) \rightarrow H^{n-s}_P((X_t)_{/Z}, \mathcal{O}_{(X_t)_{/Z}})
\]
is simply $\varprojlim_{\l} \overline{\delta}_{\l, t}$.   

The maps $\delta_{\l, t}$ and $\delta_{\l}$ defined in the previous paragraph fit into commutative diagrams
\[
\begin{CD}
0 @>>> R/J^{\l+1} @>>f^{t+1}> R/J^{\l+1} @>>> R/(J^{\l+1} + (f^{t+1})) @>>> 0\\
@.  @VV \overline{\partial f + t\partial(f)}V @VV \delta_{\l} V  @VV \delta_{\l, t} V @.\\
0 @>>> R/J^{\l} @>>f^t> R/J^{\l} @>>> R/(J^{\l} + (f^t)) @>>> 0\\
\end{CD}
\]
for all $\l$ and $t$, where the rows are exact sequences of $R$-modules and the vertical arrows are $k$-linear.  It is clear that the map $\partial f + t\partial(f): R \rightarrow R$ carries $J^{\l+1}$ into $J^{\l}$ (the left vertical arrow is the $k$-linear map that this map induces on the quotients), and the calculation (at the level of elements of $R$)
\begin{align*}
f^t((\partial f + t\partial(f))(r)) = f^t(\partial(fr) + t\partial(f)(r)) &= f^t(\partial(f)r + f\partial(r) + t\partial(f)(r))\\ &= f^{t+1}\partial(r) + (t+1)f^t\partial(f)r\\ &= \partial(f^{t+1}r)
\end{align*}
for $r \in R$ shows that the left square is commutative (the commutativity of the right square is clear).  Fix $\l$ and $t$.  Since the rows of the above diagram are short exact sequences of $R$-modules, both induce long exact sequences of local cohomology supported at $\mathfrak{m}$, and the vertical arrows, which are $k$-linear, induce $k$-linear maps on local cohomology as defined in the previous paragraph.  Since $J$ is a complete intersection ideal, $R/J^{\l}$ is a Cohen-Macaulay ring (of dimension $n-s+1$) for all $\l$ by \cite[Ex. 17.4]{matsumura}, so $H^i_{\mathfrak{m}}(R/J^{\l}) = 0$ unless $i = n-s+1$.  It follows that $H^i_{\mathfrak{m}}(R/(J^{\l} + (f^t))) = 0$ unless $i = n-s$, since $f^t$ is a non-zerodivisor on $R/J^{\l}$ for all $\l$ and $t$.  Therefore, the only nonzero portion of the resulting diagram is
\[
\begin{CD}
0 @>>> H^{n-s}_{\mathfrak{m}}(R/(J^{\l+1} + (f^{t+1}))) @>>> H^{n-s+1}_{\mathfrak{m}}(R/J^{\l+1}) @>> f^{t+1} > H^{n-s+1}_{\mathfrak{m}}(R/J^{\l+1}) @>>> 0\\ 
@.  @VV \overline{\delta}_{\l, t}V @VVV  @VV\overline{\delta}_{\l}V @.\\
0 @>>> H^{n-s}_{\mathfrak{m}}(R/(J^{\l} + (f^t))) @>>> H^{n-s+1}_{\mathfrak{m}}(R/J^{\l}) @>> f^t > H^{n-s+1}_{\mathfrak{m}}(R/J^{\l}) @>>> 0\\
\end{CD}
\]
where the middle arrow is the map induced by $\overline{\partial f + t\partial(f)}$ on local cohomology.  We now pass to the inverse limit in $\l$ while keeping $t$ fixed.  By \cite[Thm. 7.1.3]{brodmann}, the local cohomology modules appearing in the above diagram are all Artinian, and if $t$ is fixed, the transition maps in the inverse systems $\{H^{n-s+1}_{\mathfrak{m}}(R/J^{\l})\}_{\l}$ and $\{H^{n-s}_{\mathfrak{m}}(R/(J^{\l} + (f^t)))\}_{\l}$ are $R$-linear. Consequently, the Mittag-Leffler hypotheses of Proposition \ref{mlcohcomm} are satisfied, and so the rows in the above diagram remain exact after passing to the inverse limit in $\l$.  The result is a commutative diagram
\begin{gather}\label{twostar}\raisetag{-0.5 in}
\begin{CD}
0 @>>> H^{n-s}_P((X_{t+1})_{/Z}, \mathcal{O}_{(X_{t+1})_{/Z}}) @>>> H^{n-s+1}_P(X_{/Z}, \mathcal{O}_{X_{/Z}}) @>>f^{t+1}> H^{n-s+1}_P(X_{/Z}, \mathcal{O}_{X_{/Z}}) @>>> 0\\
@. @VV \overline{\delta}_{Z,t} V  @VV \overline{\delta}_Z f + t\partial(f)V @VV \overline{\delta}_Z V @. \\
0 @>>> H^{n-s}_P((X_t)_{/Z}, \mathcal{O}_{(X_t)_{/Z}}) @>>> H^{n-s+1}_P(X_{/Z}, \mathcal{O}_{X_{/Z}}) @>>f^t> H^{n-s+1}_P(X_{/Z}, \mathcal{O}_{X_{/Z}}) @>>> 0\\
\end{CD}\tag{**}
\end{gather}
where the rows are short exact sequences of $R$-modules. Recall that $\overline{\delta}_Z$ defines the inverse limit action of $\partial$ on $H^{n-s+1}_P(X_{/Z}, \mathcal{O}_{X_{/Z}})$.  We have already established that the inverse limit action of $\partial$ on 
\[
H^{n-s}_P(X_{/Y}, \mathcal{O}_{X_{/Y}}) \simeq \varprojlim H^{n-s}_P((X_t)_{/Z}, \mathcal{O}_{(X_t)_{/Z}})
\]
is given by $\varprojlim \overline{\delta}_{Z,t}$.  We now have a diagram
\[
\begin{tikzcd}
{}& H^{n-s}_P((X_{t+1})_{/Z}, \mathcal{O}_{(X_{t+1})_{/Z}}) \arrow{dd}[near start]{\overline{\delta}_{Z,t}} \arrow{rr} & & H^{n-s+1}_P(X_{/Z}, \mathcal{O}_{X_{/Z}}) \arrow{dd}[near start]{\overline{\delta}_Zf + t\partial(f)} \\
D(M/f^{t+1}M) \arrow{dd}[near start]{\partial_t^*} \arrow{ur} \arrow[crossing over]{rr} & & D(M) \arrow{ur} \\
& H^{n-s}_P((X_t)_{/Z}, \mathcal{O}_{(X_t)_{/Z}}) \arrow{rr} & & H^{n-s+1}_P(X_{/Z}, \mathcal{O}_{X_{/Z}}) \\
D(M/f^tM) \arrow{ur} \arrow{rr} & & D(M) \arrow[leftarrow, crossing over]{uu}[near end]{(f\partial_M-t\partial(f))^*} \arrow{ur} \\
\end{tikzcd}
\]
in which the front face is the left square of diagram \eqref{onestar}, and the back face is the left square of diagram \eqref{twostar}, so these two faces commute.  Consider next the arrows from the front face to the back face.  The arrows on the right are the isomorphisms of Proposition \ref{ogusiso}.  These are inverse limits of the local duality isomorphisms of Theorem \ref{groth}, which are functorial.  The arrows on the left are composites of the isomorphisms of Proposition \ref{ogusiso} with inverse limits of the change-of-ring isomorphisms $H^{s-1}_J(R/(f^t)) \simeq H^{s-1}_{\overline{J}}(R/(f^t))$ \cite[Thm. 4.2.1]{brodmann}, which are also functorial.  Therefore the top and bottom faces also commute.

By the induction hypothesis, the Matlis dual and inverse limit actions of $\partial$ on 
\[
D(M) = H^{n-s+1}_P(X_{/Z}, \mathcal{O}_{X_{/Z}})
\] 
differ by a sign, that is, $\overline{\delta}_Z = -\partial_M^*$.  Therefore, the right vertical arrow in the front face is
\[
(f\partial_M-t\partial(f))^* = (f\partial_M)^* - t\partial(f) = \partial_M^*f - t\partial(f) = -\overline{\delta}_Zf - t\partial(f)
\]
under the identification $D(M) = H^{n-s+1}_P(X_{/Z}, \mathcal{O}_{X_{/Z}})$ (here we have used the fact that the Matlis dual of multiplication by an element of $R$ is again multiplication by that element).  We conclude that the right face of the diagram anticommutes.  Since the horizontal arrows of the front and back faces are injective, the left face must also anticommute.  Consequently, under the identification $D(H^s_I(R)) = H^{n-s}_P(X_{/Y}, \mathcal{O}_{X_{/Y}})$ and isomorphisms $D(H^s_I(R)) \simeq \varprojlim D(M/f^tM)$ and $H^{n-s}_P(X_{/Y}, \mathcal{O}_{X_{/Y}}) \simeq \varprojlim H^{n-s}_P((X_t)_{/Z}, \mathcal{O}_{(X_t)_{/Z}})$, we have
\[
\overline{\delta} = \varprojlim \overline{\delta}_{Z,t} = \varprojlim (-\partial_t^*)  = -\varprojlim \partial_t^* = -\partial^*;
\]
that is, the Matlis dual and inverse limit actions on $D(H^s_I(R)) = H^{n-s}_P(X_{/Y}, \mathcal{O}_{X_{/Y}})$ also differ by a sign, completing the proof by induction for complete intersection ideals $I$.

Before beginning the proof for arbitrary ideals $I$, we need a lemma:

\begin{lem}\label{inclformal}
Let $J \subset I$ be ideals of $R$, and let $Y$ (resp. $Z$) be the closed subscheme of $X = \Spec(R)$ defined by $I$ (resp. $J$).  
\begin{enumerate}[(a)]
\item If the natural map $H^j_I(R) \rightarrow H^j_J(R)$ is surjective, and the Matlis dual and inverse limit actions of $\partial$ on $D(H^j_I(R)) = H^{n-j}_P(X_{/Y}, \mathcal{O}_{X_{/Y}})$ differ by a sign, then the Matlis dual and inverse limit actions of $\partial$ on $D(H^j_J(R)) = H^{n-j}_P(X_{/Z}, \mathcal{O}_{X_{/Z}})$ also differ by a sign.
\item If the natural map $H^j_I(R) \rightarrow H^j_J(R)$ is injective, and the Matlis dual and inverse limit actions of $\partial$ on $D(H^j_J(R)) = H^{n-j}_P(X_{/Z}, \mathcal{O}_{X_{/Z}})$ differ by a sign, then the Matlis dual and inverse limit actions of $\partial$ on $D(H^j_I(R)) = H^{n-j}_P(X_{/Y}, \mathcal{O}_{X_{/Y}})$ also differ by a sign.
\end{enumerate}
\end{lem}

\begin{proof}[Proof of Lemma \ref{inclformal}]
There are canonical surjections $R/J^{\l} \rightarrow R/I^{\l}$ for all $\l$ which induce $R$-linear maps $H^{n-j}_{\mathfrak{m}}(R/J^{\l}) \rightarrow H^{n-j}_{\mathfrak{m}}(R/I^{\l})$ and $\Ext^j_R(R/I^{\l}, R) \rightarrow \Ext^j_R(R/J^{\l}, R)$.  The first of these maps and the Matlis dual of the second are the horizontal arrows in the diagram
\[
\begin{CD}
D(\Ext^j_R(R/J^{\l}, R)) @>>> D(\Ext^j_R(R/I^{\l}, R))\\
@VVV    @VVV\\
H^{n-j}_{\mathfrak{m}}(R/J^{\l}) @>>> H^{n-j}_{\mathfrak{m}}(R/I^{\l})
\end{CD}
\]
whose vertical arrows are the isomorphisms of Theorem \ref{groth}: the diagram commutes because the isomorphisms of Theorem \ref{groth} are functorial.  By Proposition \ref{ogusiso}, the inverse limit in $\l$ of these commutative diagrams is the commutative diagram
\[
\begin{CD}
D(H^j_J(R)) @>>> D(H^j_I(R))\\
@VVV     @VVV\\
H^{n-j}_P(X_{/Z}, \mathcal{O}_{X_{/Z}}) @>>> H^{n-j}_P(X_{/Y}, \mathcal{O}_{X_{/Y}})\\
\end{CD}
\]
of $R$-modules.  If the natural map $H^j_I(R) \rightarrow H^j_J(R)$ is surjective (resp. injective), then the top horizontal arrow of this diagram is injective (resp. surjective), as is the bottom horizontal arrow (because the vertical maps are isomorphisms).  

Both parts of the lemma will follow if we can show that the top (resp. bottom) horizontal arrow commutes with the Matlis dual (resp. inverse limit) action on both sides; for then in the case that the map $H^j_I(R) \rightarrow H^j_J(R)$ is surjective (so both horizontal maps in the preceding commutative diagram are injective), the Matlis dual (resp. inverse limit) action on $D(H^j_J(R)) = H^{n-j}_P(X_{/Z}, \mathcal{O}_{X_{/Z}})$ is induced by the Matlis dual (resp. inverse limit) action on $D(H^j_I(R)) = H^{n-j}_P(X_{/Y}, \mathcal{O}_{X_{/Y}})$, and in the case that the map $H^j_I(R) \rightarrow H^j_J(R)$ is injective (so both horizontal maps in the preceding commutative diagram are surjective), the Matlis dual (resp. inverse limit) action on $D(H^j_I(R)) = H^{n-j}_P(X_{/Y}, \mathcal{O}_{X_{/Y}})$ is induced by the Matlis dual (resp. inverse limit) action on $D(H^j_J(R)) = H^{n-j}_P(X_{/Z}, \mathcal{O}_{X_{/Z}})$.

The top arrow commutes with the Matlis dual action of $\D$ on both sides because it is the Matlis dual of an $\D$-linear map.  Finally, the bottom arrow 
\[
H^{n-j}_P(X_{/Z}, \mathcal{O}_{X_{/Z}}) \rightarrow H^{n-j}_P(X_{/Y}, \mathcal{O}_{X_{/Y}})
\] 
commutes with the inverse limit action of $\D$ on both sides, because if $d \in \D$ is a differential operator of order $j$, we have commutative diagrams
\[
\begin{CD}
R/J^{t+j} @>>> R/I^{t+j}\\
@VVV        @VVV\\
R/J^t @>>> R/I^t\\
\end{CD}
\]
for all $t$ where the vertical arrows are induced by $d$, and these diagrams remain commutative after applying the functor $H^j_{\mathfrak{m}}$ to all of their objects and maps (viewing local cohomology as a functor on sheaves of $k$-spaces on the topological space $\Spec(R)$).  
\end{proof}

Now let $I \subset R$ be an arbitrary ideal, and let $h = \hgt(I)$.  Choose a regular sequence $f_1, \ldots, f_h \in I$ and denote by $J$ the ideal generated by $f_1, \ldots, f_h$, so that $J \subset I$ and $J$ is a complete intersection ideal.  Let $J = \mathfrak{q}_1 \cap \cdots \cap \mathfrak{q}_r$ be a primary decomposition of $J$.  Reindexing if necessary, we may assume $I \subset \sqrt{\mathfrak{q}_1}, \cdots, I \subset \sqrt{\mathfrak{q}_s}, I \not\subset \sqrt{\mathfrak{q}_{s+1}}, \cdots, I \not\subset \sqrt{\mathfrak{q}_r}$.  We may assume $s < r$, as otherwise $\sqrt{I} = \sqrt{J}$ and there is nothing left to prove.  Put $I' = \mathfrak{q}_1 \cap \cdots \cap \mathfrak{q}_s$ and $I'' = \mathfrak{q}_{s+1} \cap \cdots \cap \mathfrak{q}_r$; then we have $\hgt(I') = h, \hgt(I' + I'') = h+1$, $J = I' \cap I''$, and $\sqrt{I} = \sqrt{I'}$.

The ideals $I'$ and $I''$ give rise to a Mayer-Vietoris sequence of local cohomology (Proposition \ref{mv})
\[
\cdots \rightarrow H^{h-1}_J(R) \rightarrow H^h_{I' + I''}(R) \rightarrow H^h_{I}(R) \oplus H^h_{I''}(R) \rightarrow H^h_J(R) \rightarrow H^{h+1}_{I' + I''}(R) \rightarrow \cdots
\]
where we identify $H^i_I(R)$ with $H^i_{I'}(R)$ for all $i$ because the ideals $I$ and $I'$ have the same radical \cite[Remark 1.2.3]{brodmann}.  Since $\hgt(I' + I'') > h$, we have $H^h_{I' + I''}(R) = 0$, so there is an injection $H^h_I(R) \oplus H^h_{I''}(R) \hookrightarrow H^h_J(R)$; composing with the inclusion of the first summand, we obtain an injection $H^h_I(R) \hookrightarrow H^h_J(R)$, which is just the natural map induced by the inclusion $J \subset I$ of ideals.  By the previous step of the proof, since $J$ is a complete intersection ideal, the Matlis dual and inverse limit actions of $\partial$ on $D(H^h_J(R)) = H^{n-h}_P(X_{/Z}, \mathcal{O}_{X_{/Z}})$ differ by a sign.  We conclude from Lemma \ref{inclformal}(b) that the Matlis dual and inverse limit actions of $\partial$ on $D(H^h_I(R)) = H^{n-h}_P(X_{/Y}, \mathcal{O}_{X_{/Y}})$ also differ by a sign.  This proves the theorem for an arbitrary ideal $I$ in cohomological degree equal to the height of $I$.

Finally, we prove the theorem in arbitrary cohomological degree $i$ by induction on $i - \hgt(I)$; the base case, where $i = \hgt(I)$, was proved in the previous paragraph.  Let $i > h = \hgt(I)$ be given.  Since $J$ is a complete intersection ideal of height $h$, $H^i_J(R) = 0$ unless $i = h$, so from the Mayer-Vietoris sequence we also conclude that $H^i_{I' + I''}(R)$ maps surjectively to $H^i_I(R) \oplus H^i_{I''}(R)$; composing with the projection onto the first summand, we obtain a surjection $H^i_{I' + I''}(R) \rightarrow H^i_I(R)$, which is just the natural map induced by the inclusion $I \subset I' + I''$ of ideals.  We have $\hgt(I' + I'') > \hgt(I)$, so $i - \hgt(I' + I'') < i - \hgt(I)$.  By the induction hypothesis, the Matlis dual and inverse limit actions of $\partial$ on $D(H^i_{I' + I''}(R)) = H^{n-i}_P(X_{/Z'}, \mathcal{O}_{X_{/Z'}})$ differ by a sign, where $Z'$ is the closed subscheme of $X = \Spec(R)$ defined by $I' + I''$.  We conclude from Lemma \ref{inclformal}(a) that the Matlis dual and inverse limit actions of $\partial$ on $D(H^i_I(R)) = H^{n-i}_P(X_{/Y}, \mathcal{O}_{X_{/Y}})$ also differ by a sign.
\end{proof}

\section{The de Rham cohomology of a complete local ring}\label{cohom}

In this section, we turn our attention to de Rham cohomology, and prove Theorems \ref{mainthmB} and \ref{dualE2}.  We prove the latter theorem (and Theorem \ref{mainthmB}(b) as an immediate corollary) first, by synthesizing our results from the previous sections on Matlis duality and local cohomology of formal schemes.  The proof of Theorem \ref{mainthmB}(a), which we give next, proceeds similarly to that of Theorem \ref{mainthmA}(a).  As in section \ref{homology}, we will first give the proof of Theorem \ref{mainthmB}(a) for the $E_2$-terms of the spectral sequences, to bring out the main ideas of the proof.  There is a technical complication that will make the $E_2$ case more difficult than the $E_2$ case of Theorem \ref{mainthmA}(a).  (We will need to verify various Mittag-Leffler hypotheses for inverse systems, for which Corollary \ref{fgstabdR} will be necessary.)

We begin by recalling notation used in section \ref{homology} and giving Hartshorne's definition of de Rham cohomology.  Thus $A$ is again a complete local ring with a coefficient field $k$ of characteristic zero and $\pi: R \rightarrow A$ is a surjection of $k$-algebras with $R = k[[x_1, \ldots, x_n]]$ for some $n$.  The surjection $\pi$ induces a closed immersion of spectra $\Spec(A) = Y \hookrightarrow X = \Spec(R)$, defined by the coherent sheaf of ideals $\mathcal{I} \subset \mathcal{O}_X$ associated with the ideal $I = \ker \pi \subset R$.  Let $\widehat{X}$ be the formal completion of $Y$ in $X$, that is, the topological space $Y$ equipped with the structure of a locally ringed space via the sheaf of rings $\mathcal{O}_{\widehat{X}} = \varprojlim \mathcal{O}_X/\mathcal{I}^{\l}$, an inverse limit of sheaves of rings supported at $Y$.  For any coherent sheaf $\mathcal{F}$ of $\mathcal{O}_X$-modules, we can define its $\mathcal{I}$-adic completion $\widehat{\mathcal{F}} = \varprojlim \mathcal{F}/\mathcal{I}^{\l}\mathcal{F}$, which is a sheaf of $\mathcal{O}_{\widehat{X}}$-modules.

Now consider again the (continuous) de Rham complex $\Omega_X^{\bullet}$.  By the Leibniz rule, the $k$-linear differentials in this complex of sheaves on $X$ are $\mathcal{I}$-adically continuous, and therefore pass to the $\mathcal{I}$-adic completions of the coherent $\mathcal{O}_X$-modules $\Omega^i_X$.  We obtain, therefore, a complex $\widehat{\Omega}_X^{\bullet}$, whose objects are sheaves of $\mathcal{O}_{\widehat{X}}$-modules but whose differentials are merely $k$-linear.  Since $X$ is the spectrum of a complete regular local ring and so the sheaves $\Omega_X^i$ are finite free $\mathcal{O}_X$-modules, there is a simpler description of the complex $\widehat{\Omega}_X^{\bullet}$.  As formal completion commutes with \emph{finite} direct sums, the sheaf $\widehat{\Omega}^i_X$ is a direct sum of copies of $\mathcal{O}_{\widehat{X}}$.  All of the derivations $\partial_{x_j}: R \rightarrow R$ induce $\mathcal{I}$-adically continuous maps $\mathcal{O}_{\widehat{X}} \rightarrow \mathcal{O}_{\widehat{X}}$, and if we form the de Rham complex of $\mathcal{O}_{\widehat{X}}$ with respect to these derivations, we recover precisely the complex $\widehat{\Omega}_X^{\bullet}$.

In \cite{derham}, the \emph{(local) de Rham cohomology} of the local scheme $Y$ is defined as $H^i_{P, dR}(Y) = \mathbf{H}^i_P(\widehat{X}, \widehat{\Omega}_X^{\bullet})$, the hypercohomology (supported at the closed point $P$ of $Y$) of the completed de Rham complex.  These $k$-spaces are known to be independent of the choice of $R$ and $\pi$ \cite[Prop. III.1.1]{derham} and finite-dimensional \cite[Thm. III.2.1]{derham}.  Since the local de Rham cohomology is defined as the hypercohomology of a complex, we have, as in the case of homology, a Hodge-de Rham spectral sequence that begins $\tilde{E}_1^{p,q} = H^q_P(\widehat{X}, \widehat{\Omega}^p_X)$ and abuts to $H^{p+q}_{P, dR}(Y)$.  

We also recall from section \ref{homology} that the Hodge-de Rham spectral sequence for homology has $E_1$-term given by $E_1^{n-p, n-q} = H^{n-q}_Y(X, \Omega_X^{n-p})$ and abuts to $H_{p+q}^{dR}(Y)$.  The assertion of Theorem \ref{dualE2} is that, for all $p$ and $q$, the $k$-spaces $E_2^{n-p, n-q}$ and $\tilde{E}_2^{p,q}$ are dual to each other.  Since we know by Theorem \ref{mainthmA}(b) that the former $k$-space is finite-dimensional, this duality implies that the latter is as well, which will prove Theorem \ref{mainthmB}(b).  Our work in sections \ref{mittleff} and \ref{formal} is nearly enough to establish this duality; all that remains is to identify the rows of the $E_1$-term of the \emph{cohomology} spectral sequence with de Rham complexes of $\D$-modules.

We recall again Convention \ref{tensorsubscript}, which remains in force in this section: if $M$ is a left $\D(R,k)$-module, we write $M \otimes \Omega^{\bullet}_R$ for its de Rham complex, where the subscript $R$ indicates the ring over which the complex is being computed: the objects of this complex are tensor products of $R$-modules taken over $R$, but the maps are not $R$-linear.

\begin{lem}\label{E2cohshape}
Let the surjection $\pi: R = k[[x_1, \ldots, x_n]] \rightarrow A$ (and the associated objects $I, X, \widehat{X}, Y$) be as above, and $\{\tilde{E}_{\bullet}^{\bullet,\bullet}\}$ the corresponding Hodge-de Rham spectral sequence for cohomology.  For all $q$, the $q$th row $\{\tilde{E}_1^{\bullet, q}\}$ of the $E_1$-term $\{\tilde{E}_1^{\bullet,\bullet}\}$ is isomorphic, as a complex of $k$-spaces, to the de Rham complex $H^q_P(\widehat{X}, \mathcal{O}_{\widehat{X}}) \otimes \Omega_R^{\bullet}$, where the left $\D(R,k)$-structure on $H^q_P(\widehat{X}, \mathcal{O}_{\widehat{X}})$ is given by the \emph{inverse limit action} of Definition \ref{twoduals}.
\end{lem}

\begin{proof}
Let $q$ be fixed.  Both formal completion along $Y$ and local cohomology $H^q_P$ commute with finite direct sums, so for all $p$, $\tilde{E}_1^{p,q} = H^q_P(\widehat{X}, \widehat{\Omega}^p_X)$ is a direct sum of copies of the $R$-module $H^q_P(\widehat{X}, \mathcal{O}_{\widehat{X}})$, and the complex $\{\tilde{E}_1^{\bullet, q}\}$ has differentials induced by the differentials in the complex $\Omega_X^{\bullet}$ by first passing to $\mathcal{I}$-adic completions and then applying the functor $H^q_P$.  This is exactly the de Rham complex of $H^q_P(\widehat{X}, \mathcal{O}_{\widehat{X}})$ with respect to the inverse limit action, since by Proposition \ref{ogusiso} we have an isomorphism $H^q_P(\widehat{X}, \mathcal{O}_{\widehat{X}}) \simeq \varprojlim H^q_P(\widehat{X}, \mathcal{O}_X/\mathcal{I}^{\l})$.
\end{proof} 

\begin{proof}[Proof of Theorems \ref{dualE2} and \ref{mainthmB}(b)]
Fix $q$.  By Lemma \ref{E2cohshape}, the $q$th row of the $E_1$-term of the cohomology spectral sequence, $\tilde{E}_1^{\bullet, q}$, is the de Rham complex of the left $\D(R,k)$-module $H^q_P(\widehat{X}, \mathcal{O}_{\widehat{X}})$.  By Theorem \ref{loccohtranspose} and Remark \ref{E1duals}, this complex is isomorphic (as a complex of $k$-spaces) to the de Rham complex of the left $\D(R,k)$-module $D(H^{n-q}_I(R))$.  We obtain the $E_2$-term by taking cohomology, so for all $p$, $\tilde{E}_2^{p,q} = H^p_{dR}(D(H^{n-q}_I(R)))$.  Since $H^{n-q}_I(R)$ is a holonomic $\D(R,k)$-module, Theorem \ref{dualcoh} applies: we have an isomorphism 
\[
H^p_{dR}(D(H^{n-q}_I(R))) \simeq (H^{n-p}_{dR}(H^{n-q}_I(R)))^{\vee}.
\]  
The right-hand side of this isomorphism is nothing but $(E_2^{n-p,n-q})^{\vee}$ (by Lemma \ref{E2shape}), completing the proof.
\end{proof}

We will now begin working toward the proof of Theorem \ref{mainthmB}(a).  The reductions immediately preceding Lemma \ref{specialcase} are equally valid here, so for the remainder of the section, we assume that $R=k[[x_1, \ldots, x_n]]$ and $R' = R[[z]]$, and it suffices to compare the Hodge-de Rham spectral sequences for cohomology corresponding to an arbitrary ideal $I \subset R$ (defining a closed immersion $Y = \Spec(R/I) \hookrightarrow X = \Spec(R)$) and the ideal $I' = IR' + (z) \subset R'$ (defining a closed immersion $Y \hookrightarrow X' = \Spec(R')$).

In the proof of Theorem \ref{mainthmA}(a), we made use of an operation (Definition \ref{plusop}) on $k$-spaces.  Its replacement in this section is the following operation:

\begin{definition}\label{upperplus}
Let $M$ be any $k$-space.  We define $M^+ = M[[z]]$, the $\D(k[[z]],k)$-module of formal power series with coefficients in $M$.  If $M$ is an $R$-module (resp. a $\D(R,k)$-module), then $M^+$ defined in this way is an $R'$-module (resp. a $\D(R',k)$-module), with $\partial_z$-action defined by $\partial_z(\sum m_{\l}z^{\l}) = \sum (\l+1)m_{\l+1}z^{\l}$.
\end{definition}

If $M$ is an $R$-module, the $R'$-module $M^+$ is always $z$-adically complete.  What is more, $k$-linear maps between $k$-spaces and $R$-linear maps between $R$-modules extend to the corresponding formal power series objects: if $f: M \rightarrow N$ is a $k$-linear map between $k$-spaces (or an $R$-linear map between $R$-modules), $f^+: M^+ \rightarrow N^+$ is defined by $f^+(\sum m_{\l}z^{\l}) = \sum f(m_{\l})z^{\l}$.

\begin{remark}
If $M$ is an $R$-module, the $R'$-module $M^+ = M[[z]]$ is usually not isomorphic to $M \otimes_R R'$.  This is an example of the failure of inverse limits to commute with tensor products.  We do have 
\[
M^+ \simeq \varprojlim M[[z]]/z^{\l} \simeq \varprojlim (M \otimes_R R'/z^{\l}),
\]
where every $M \otimes_R R'/z^{\l}$ is regarded as an $R'$-module via the natural surjection $R' \rightarrow R'/z^{\l}$, but this inverse limit need not be isomorphic to $M \otimes_R (\varprojlim R'/z^{\l})$.
\end{remark}

In the proofs below, we will often take two inverse limits simultaneously in the process of forming the module $M^+$.  The general principle is the following:

\begin{lem}\label{doublelim}
Let $\{M_{\l}\}$ be an inverse system of $R$-modules, indexed by $\mathbb{N}$, and let $M = \varprojlim M_{\l}$.  Then $\varprojlim (M_{\l} \otimes_R R'/z^{\l}) \simeq M^+$ as $R'$-modules.
\end{lem}

\begin{proof}
As an $R'$-module, $M^+ \simeq \varprojlim_{\l} ((\varprojlim_s M_s) \otimes_R R'/z^{\l})$.  Consider the inverse system $\{M_s \otimes_R R'/z^{\l}\}$, indexed by $\mathbb{N} \times \mathbb{N}$ where $(s, \l) \leq (s', \l')$ if and only if $s \leq s'$ and $\l \leq \l'$.  Then
\[ 
\varprojlim_{\l} ((\varprojlim_s M_s) \otimes_R R'/z^{\l}) \simeq \varprojlim_{(s, \l)} M_s \otimes_R R'/z^{\l}.
\]  
As the ``diagonal'' inverse system $\{M_{\l} \otimes_R R'/z^{\l}\}_{\l}$ is cofinal with $\{M_s \otimes_R R'/z^{\l}\}_{(s,\l)}$, their inverse limits are isomorphic, as desired.
\end{proof}

We have defined the formal power series operation both for $k$-spaces and for $R$-modules.  In what follows, we will frequently apply the operation to an entire complex whose objects are $R$-modules but whose differentials are merely $k$-linear.  We will still (abusively) use the notation $\otimes \, R'/z^{\l}$ for the $\l$th truncation of the formal power series operation, a convention which we record here:

\begin{definition}\label{otimescplx}
Let $C^{\bullet}$ be a complex whose objects are $R$-modules and whose differentials are merely $k$-linear.  The complex $C^{\bullet} \otimes R'/z^{\l}$ is the direct sum of $\l$ copies of $C^{\bullet}$, indexed by $z^i$ for $i = 0, \ldots, \l - 1$.
\end{definition}

Note that the complex $(C^{\bullet})^+$, obtained by applying the formal power series operation to all objects and differentials of $C^{\bullet}$, is the inverse limit (in the category of complexes of $k$-spaces) of $C^{\bullet} \otimes R'/z^{\l}$.

\begin{remark}\label{ab4}
We note one important difference between the formal power series operation defined above and the operation of Definition \ref{plusop}.  The latter operation is defined using an infinite direct sum.  Therefore, the question of whether it commutes with cohomology reduces to the question of whether the underlying category satisfies Grothendieck's axiom $AB4$ \cite{tohoku}, that is, whether direct sums are exact.  This is true for the categories of $R$-modules, of $k$-spaces, and of sheaves of Abelian groups on a topological space \cite[p. 80]{weibel}.  By contrast, the formal power series operation is defined using an infinite direct product, which commutes with cohomology if the underlying category satisfies axiom $AB4^*$.  This is true for the categories of $R$-modules and $k$-spaces but \emph{not} sheaves \cite[p. 80]{weibel}.
\end{remark}

The main technical preliminary result we need concerns the interaction of the formal power series operation with the de Rham complexes of local cohomology modules on formal schemes.  For reference, we repeat in the statement of this proposition the notation we will use for the remainder of this section.

\begin{lem}\label{dRlcplus}
Let $R = k[[x_1, \ldots, x_n]]$ and let $I$ be an ideal of $R$.  Let $R' = R[[z]]$ and $I' = IR' + (z)$.  Let $\mathcal{I}$ (resp. $\mathcal{I}'$) be the associated sheaf of ideals of $\mathcal{O}_X$ (resp. $\mathcal{O}_{X'}$) where $X = \Spec(R)$ and $X' = \Spec(R')$.  Let $Y$ be the closed subscheme of $X$ defined by $\mathcal{I}$; via the natural closed immersion $X \hookrightarrow X'$, we identify $Y$ with the closed subscheme of $X'$ defined by $\mathcal{I}'$, which we also denote $Y$.  Let $P \in Y$ be the closed point.  Then for all $q$, we have 
\[
H^q_P(\widehat{X'}, \mathcal{O}_{\widehat{X'}}) \otimes \Omega_R^{\bullet} \simeq (H^q_P(\widehat{X}, \mathcal{O}_{\widehat{X}}) \otimes \Omega_R^{\bullet})^+
\] 
as complexes of $k$-spaces, where the $\D(R,k)$-structure on both local cohomology modules is given by the \emph{inverse limit action} of Definition \ref{twoduals}.
\end{lem}

The proof of Lemma \ref{dRlcplus} involves several ideas, so we begin with a lemma focusing on a single local cohomology module, with no reference to its de Rham complex.  We will appeal below not only to this lemma but also to its proof.

\begin{lem}\label{lcplus}
All notation is the same as in Lemma \ref{dRlcplus}.  For all $q$, we have 
\[
H^q_P(\widehat{X'}, \mathcal{O}_{\widehat{X'}}) \simeq (H^q_P(\widehat{X}, \mathcal{O}_{\widehat{X}}))^+
\] 
as $R'$-modules.
\end{lem}

\begin{proof}
For all $\l$, let $J_{\l}$ be the ideal $I^{\l}R' + (z^{\l})$.  The families $\{J_{\l}\}$ and $\{(I')^{\l}\}$ of ideals of $R'$ are cofinal, and we have isomorphisms
\[
R'/J_{\l} = R'/(I^{\l}R' + (z^{\l})) \simeq R/I^{\l} \otimes_R R'/z^{\l}
\]
as $R'$-modules for all $\l$ (the $R'$-module structure on $R/I^{\l} \otimes_R R'/z^{\l}$ being defined via the natural surjection $R' \rightarrow R'/z^{\l}$).

Denote by $\mathfrak{n}$ (resp. $\mathfrak{m}$) the maximal ideal of $R'$ (resp. $R$).  For all $q$, we have isomorphisms
\[
H^q_P(\widehat{X'}, \mathcal{O}_{\widehat{X'}}) \simeq \varprojlim H^q_{\mathfrak{n}}(R'/(I')^{\l}) \simeq \varprojlim H^q_{\mathfrak{n}}(R'/J_{\l}),
\]
the first isomorphism by Proposition \ref{ogusiso} and the second by the cofinality of $\{J_{\l}\}$ and $\{(I')^{\l}\}$.  We saw above that $R'/J_{\l} \simeq R/I^{\l} \otimes_R R'/z^{\l}$ as $R'$-modules, and therefore 
\[
H^q_{\mathfrak{n}}(R'/J_{\l}) \simeq H^q_{\mathfrak{n}}(R/I^{\l} \otimes_R R'/z^{\l})
\]
as $R'$-modules.  We claim that the right-hand side is isomorphic to $H^q_{\mathfrak{m}}(R/I^{\l}) \otimes_R R'/z^{\l}$ as an $R'$-module.  

The $R'$-module $R/I^{\l} \otimes_R R'/z^{\l}$ is annihilated by a power of $z$, and so
\[
H^i_{(z)}(R/I^{\l} \otimes_R R'/z^{\l}) = R/I^{\l} \otimes_R R'/z^{\l}
\] 
if $i = 0$, and is zero otherwise.  The spectral sequence of Example \ref{lccomposite} corresponding to the composite functor $\Gamma_{\mathfrak{n}} = \Gamma_{\mathfrak{m}R'} \circ \Gamma_{(z)}$ therefore degenerates at $E_2$, and we have isomorphisms 
\[
H^q_{\mathfrak{n}}(R/I^{\l} \otimes_R R'/z^{\l}) \simeq H^q_{\mathfrak{m}R'}(R/I^{\l} \otimes_R R'/z^{\l})
\] 
as $R'$-modules.  By the change of ring principle \cite[Thm. 4.2.1]{brodmann}, it does not matter whether we compute this last local cohomology module over $R'$ or over $R'/z^{\l}$, so in fact 
\[
H^q_{\mathfrak{m}R'}(R/I^{\l} \otimes_R R'/z^{\l}) \simeq H^q_{\mathfrak{m}(R'/z^{\l})}(R/I^{\l} \otimes_R R'/z^{\l})
\] 
as $R'/z^{\l}$-modules.  Finally, since $R'/z^{\l}$ is flat over $R$, the flat base change theorem \cite[Thm. 4.3.2]{brodmann} implies that 
\[
H^q_{\mathfrak{m}(R'/z^{\l})}(R/I^{\l} \otimes_R R'/z^{\l}) \simeq H^q_{\mathfrak{m}}(R/I^{\l}) \otimes_R R'/z^{\l}
\] 
as $R'/z^{\l}$-modules.  The previous two isomorphisms of $R'/z^{\l}$-modules are isomorphisms of $R'$-modules, as well, since the $R'$-structures are defined using the natural surjection $R' \rightarrow R'/z^{\l}$.  Putting these isomorphisms together, we see that
\[
H^q_{\mathfrak{n}}(R'/J_{\l}) \simeq H^q_{\mathfrak{n}}(R/I^{\l} \otimes_R R'/z^{\l}) \simeq H^q_{\mathfrak{m}}(R/I^{\l}) \otimes_R R'/z^{\l}
\]
as $R'$-modules, for all $\l$.  As the isomorphisms of \cite[Thm. 4.2.1, Thm. 4.3.2]{brodmann} are functorial, the isomorphisms above form a compatible system, and passing to the inverse limit, we have
\[
H^q_P(\widehat{X'}, \mathcal{O}_{\widehat{X'}}) \simeq \varprojlim H^q_{\mathfrak{n}}(R'/J_{\l}) \simeq \varprojlim H^q_{\mathfrak{m}}(R/I^{\l}) \otimes_R R'/z^{\l},
\]
and since $\varprojlim H^q_{\mathfrak{m}}(R/I^{\l}) \simeq H^q_P(\widehat{X}, \mathcal{O}_{\widehat{X}})$ (again by Proposition \ref{ogusiso}), the rightmost module above is isomorphic as an $R'$-module to $(H^q_P(\widehat{X}, \mathcal{O}_{\widehat{X}}))^+$ by Lemma \ref{doublelim}, as desired.
\end{proof}

We now consider the de Rham complexes as well, and prove Lemma \ref{dRlcplus}.

\begin{proof}[Proof of Lemma \ref{dRlcplus}]
We retain the notation introduced in the proof of Lemma \ref{lcplus}.  For all $\l$ and $s$, let $J_{\l,s} = I^{\l+s}R' + (z^{\l})$.  (Note that $J_{\l,0} = J_{\l}$.) The families $\{J_{\l,s}\}$ (with $s$ fixed) and $\{(I')^{\l}\}$ of ideals of $R'$ are cofinal.  For all $\l$ and $s$, the derivations $\partial_1, \ldots, \partial_n$ induce (by the Leibniz rule) $k$-linear maps $R'/J_{\l,s} \rightarrow R'/J_{\l,s-1}$, as all of these derivations are $z$-linear.  In turn, these maps induce $k$-linear maps on local cohomology as described in section \ref{formal}.  We can therefore construct, for every $\l$, a ``de Rham-like'' complex
\[
0 \rightarrow H^q_{\mathfrak{n}}(R'/J_{\l,n}) \rightarrow \oplus_{1 \leq i \leq n} H^q_{\mathfrak{n}}(R'/J_{\l,n-1}) \rightarrow \cdots \rightarrow H^q_{\mathfrak{n}}(R'/J_{\l,0}) \rightarrow 0
\]
using the derivations $\partial_1, \ldots, \partial_n$.  We write $H^q_{\mathfrak{n}}(\mathcal{C}_{\l}^{\bullet})$ for this complex (\emph{cf.} Definition \ref{omnibus} below).  The argument in the proof of Lemma \ref{lcplus} applies to all terms of this complex, and using the fact that the differentials in this complex do not involve $z$ or $dz$, we see that this complex is isomorphic to the complex
\[
(0 \rightarrow H^q_{\mathfrak{m}}(R/I^{\l+n}) \rightarrow \oplus_{1 \leq i \leq n} H^q_{\mathfrak{m}}(R/I^{\l+n-1}) \rightarrow \cdots \rightarrow H^q_{\mathfrak{m}}(R/I^{\l}) \rightarrow 0) \otimes R'/z^{\l},
\]
which we write $H^q_{\mathfrak{m}}(C_{\l}^{\bullet}) \otimes R'/z^{\l}$, again anticipating Definition \ref{omnibus}.  We now pass to the inverse limit in $\l$ of both systems of complexes.  For all $s$, we have 
\[
\varprojlim_{\l} H^q_{\mathfrak{n}}(R'/J_{\l,s}) \simeq H^q_P(\widehat{X'}, \mathcal{O}_{\widehat{X'}})
\]
as $R'$-modules, by Proposition \ref{ogusiso} and the cofinality of $\{J_{\l,s}\}$ and $\{(I')^{\l}\}$.  Moreover, by the definition of the inverse limit action, the differentials in the complex $H^q_P(\widehat{X'}, \mathcal{O}_{\widehat{X'}}) \otimes \Omega_R^{\bullet}$ are given by taking the inverse limit of the differentials in the complexes $H^q_{\mathfrak{n}}(\mathcal{C}_{\l}^{\bullet})$.  That is, we have
\[
H^q_P(\widehat{X'}, \mathcal{O}_{\widehat{X'}}) \otimes \Omega_R^{\bullet} \simeq \varprojlim H^q_{\mathfrak{n}}(\mathcal{C}_{\l}^{\bullet}),
\]
for all $\l$, as complexes of $k$-spaces.  On the other hand, again using Proposition \ref{ogusiso}, we have 
\[
\varprojlim H^q_{\mathfrak{m}}(R/I^{\l}) \simeq H^q_P(\widehat{X}, \mathcal{O}_{\widehat{X}})
\]
as $R$-modules, and by the definition of the inverse limit action,
\[
H^q_P(\widehat{X}, \mathcal{O}_{\widehat{X}}) \otimes \Omega_R^{\bullet} \simeq \varprojlim H^q_{\mathfrak{m}}(C_{\l}^{\bullet}),
\]
for all $\l$, as complexes of $k$-spaces.  To conclude the lemma, it suffices to show that
\[
\varprojlim H^q_{\mathfrak{n}}(\mathcal{C}_{\l}^{\bullet}) \simeq (\varprojlim H^q_{\mathfrak{m}}(C_{\l}^{\bullet}))^+
\]
as complexes of $k$-spaces.  We have already shown that we have isomorphisms $H^q_{\mathfrak{n}}(\mathcal{C}_{\l}^{\bullet}) \simeq H^q_{\mathfrak{m}}(C_{\l}^{\bullet}) \otimes R'/z^{\l}$ that clearly form a compatible system, which implies
\[
\varprojlim H^q_{\mathfrak{n}}(\mathcal{C}_{\l}^{\bullet}) \simeq \varprojlim H^q_{\mathfrak{m}}(C_{\l}^{\bullet}) \otimes R'/z^{\l},
\]
and since $\varprojlim H^q_{\mathfrak{m}}(C_{\l}^{\bullet}) \otimes R'/z^{\l} \simeq (\varprojlim H^q_{\mathfrak{m}}(C_{\l}^{\bullet}))^+$ (by applying Lemma \ref{doublelim} to all objects in the complex), the lemma follows.
\end{proof}

We can now begin the proof of Theorem \ref{mainthmB}(a).  We have already reduced ourselves to the case where $R$ and $R'$ are as in the statement of Lemma \ref{dRlcplus}.  With $R$ and $R'$ as in the statement of that lemma, our goal is to compare the spectral sequence $\{\tilde{E}_{r,R}^{p,q}\}$ arising from the surjection $R \rightarrow A$, to the spectral sequence $\{\tilde{\mathbf{E}}_{r,R'}^{p,q}\}$ arising from the surjection $R' \rightarrow A$.  We will first prove that the $E_2$-objects of these spectral sequences are isomorphic.  As we have identified these $E_2$-objects in Lemma \ref{E2cohshape}, this claim is equivalent to the following proposition. As with Lemma \ref{specialcase}, we will later need not only the statement of this proposition, but the specific isomorphisms appearing in its proof.  

\begin{prop}\label{cohspecialcase}
All notation is the same as in Lemma \ref{dRlcplus}.  For all $p$ and $q$, we have
\[
\tilde{E}_{2,R}^{p,q} = H^p_{dR}(H^q_P(\widehat{X}, \mathcal{O}_{\widehat{X}})) \simeq H^p_{dR}(H^q_P(\widehat{X'}, \mathcal{O}_{\widehat{X'}})) = \tilde{\mathbf{E}}_{2,R'}^{p,q}
\]
as $k$-spaces, where the de Rham cohomology is computed by regarding $H^q_P(\widehat{X}, \mathcal{O}_{\widehat{X}})$ as a $\D(R,k)$-module and $H^q_P(\widehat{X'}, \mathcal{O}_{\widehat{X'}})$ as a $\D(R',k)$-module. 
\end{prop}

\begin{proof}
The argument closely parallels that of Lemma \ref{specialcase}.  Consider the short exact sequence
\[
0 \rightarrow \mathcal{O}_{\widehat{X'}} \otimes \widehat{\Omega}_X^{\bullet}[-1] \xrightarrow{\iota} \widehat{\Omega}_{X'}^{\bullet} \xrightarrow{\pi} \mathcal{O}_{\widehat{X'}} \otimes \widehat{\Omega}_X^{\bullet} \rightarrow 0
\]
of complexes of sheaves of $k$-spaces on $\widehat{X'}$, where $\iota$ is simply $\wedge \, dz$.  This sequence is the analogue, for completed sheaves, of the short exact sequence given in Definition \ref{partialdR}, and consists of split exact sequences of finite free $\mathcal{O}_{\widehat{X'}}$-modules.  Apply $H^q_P$ to the entire sequence of complexes: as formal completion and local cohomology commute with finite direct sums, we obtain a short exact sequence
\[
0 \rightarrow H^q_P(\widehat{X'}, \mathcal{O}_{\widehat{X'}}) \otimes \Omega_R^{\bullet}[-1] \rightarrow H^q_P(\widehat{X'}, \mathcal{O}_{\widehat{X'}}) \otimes \Omega_{R'}^{\bullet} \xrightarrow{\pi^{\bullet}_q} H^q_P(\widehat{X'}, \mathcal{O}_{\widehat{X'}}) \otimes \Omega_R^{\bullet} \rightarrow 0
\]
of complexes of $k$-spaces.  The corresponding long exact sequence in cohomology (accounting for the shift of $-1$) is 
\begin{align*}
\cdots \rightarrow h^{p-1}(H^q_P(\widehat{X'}, \mathcal{O}_{\widehat{X'}}) \otimes \Omega_R^{\bullet}) \xrightarrow{\partial_z} h^{p-1}(H^q_P(\widehat{X'}, \mathcal{O}_{\widehat{X'}}) \otimes \Omega_R^{\bullet})
&\rightarrow \\
h^p(H^q_P(\widehat{X'}, \mathcal{O}_{\widehat{X'}}) \otimes \Omega_{R'}^{\bullet}) \xrightarrow{\overline{\pi}^p_q} h^p(H^q_P(\widehat{X'}, \mathcal{O}_{\widehat{X'}}) \otimes \Omega_R^{\bullet}) 
&\xrightarrow{\partial_z} h^p(H^q_P(\widehat{X'}, \mathcal{O}_{\widehat{X'}}) \otimes \Omega_R^{\bullet}) \rightarrow \cdots,
\end{align*}
where we know by Lemma \ref{connhom} that, up to a sign, the connecting homomorphism is $\partial_z$.  By Lemma \ref{dRlcplus}, we have  
\[
H^q_P(\widehat{X'}, \mathcal{O}_{\widehat{X'}}) \otimes \Omega_R^{\bullet} \simeq (H^q_P(\widehat{X}, \mathcal{O}_{\widehat{X}}) \otimes \Omega_R^{\bullet})^+
\]
as complexes of $k$-spaces.  By Remark \ref{ab4}, the formal power series operation on the right-hand side commutes with cohomology (since we are now applying this operation to $k$-spaces, not sheaves).  Taking the cohomology of both sides, we find
\[
h^p(H^q_P(\widehat{X'}, \mathcal{O}_{\widehat{X'}}) \otimes \Omega_R^{\bullet}) \simeq (h^p(H^q_P(\widehat{X}, \mathcal{O}_{\widehat{X}}) \otimes \Omega_R^{\bullet}))^+
\] 
as $k$-spaces for all $p$.  For any $k$-space $M$, the action of $\partial_z$ on $M^+$ is given in Definition \ref{upperplus}, and it is clear from this definition (since $\ch(k) = 0$) that $\coker(\partial_z: M^+ \rightarrow M^+) = 0$ and $\ker(\partial_z: M^+ \rightarrow M^+) \simeq M$, the latter corresponding to the ``constant term'' component of $M^+$.  Returning to the displayed portion of the long exact sequence (with $M = h^p(H^q_P(\widehat{X}, \mathcal{O}_{\widehat{X}}) \otimes \Omega_R^{\bullet})$), the first $\partial_z$ is surjective, and so by exactness the unlabeled arrow is the zero map; this implies that $\overline{\pi}^p_q$ is injective, inducing an isomorphism between the kernel of the second $\partial_z$ and 
\[
h^p(H^q_P(\widehat{X'}, \mathcal{O}_{\widehat{X'}}) \otimes \Omega_{R'}^{\bullet}) = H^p_{dR}(H^q_P(\widehat{X'}, \mathcal{O}_{\widehat{X'}})).
\]
Since this kernel is isomorphic to $h^p(H^q_P(\widehat{X'}, \mathcal{O}_{\widehat{X}}) \otimes \Omega_R^{\bullet}) = H^p_{dR}(H^q_P(\widehat{X}, \mathcal{O}_{\widehat{X}}))$, the proof is complete.
\end{proof}

We next work toward the general case of Theorem \ref{mainthmB}(a).  Our goal is to construct a morphism between the Hodge-de Rham spectral sequences $\{\tilde{E}_{r,R}^{p,q}\}$ and $\{\tilde{\mathbf{E}}_{r,R'}^{p,q}\}$ arising from the two surjections $R \rightarrow A$ and $R' \rightarrow A$ which, at the level of $E_2$-objects, consists of the isomorphisms of Proposition \ref{cohspecialcase}: by Proposition \ref{E1isos}, this is enough.  As in the proof of Theorem \ref{mainthmA}(a), we will construct an ``intermediate'' spectral sequence $\{\tilde{\mathcal{E}}_r^{p,q}\}$.  The analogue of Lemma \ref{ssplus}, however, will be significantly harder to prove, since we are working with inverse limits and must therefore check various Mittag-Leffler conditions.

We begin with definitions of several complexes, collecting pieces of notation introduced in the course of the preceding proofs together with some obvious variations.

\begin{definition}\label{omnibus}
All notation is the same as in Lemma \ref{dRlcplus}.
\begin{enumerate}[(a)]
\item Let $C^{\bullet}_{\l}$ be the complex
\[
0 \rightarrow R/I^{\l+n} \rightarrow \oplus_{1 \leq 1 \leq n} R/I^{\l+n-1} \rightarrow \cdots \rightarrow R/I^{\l} \rightarrow 0
\]
defined using the derivations $\partial_1, \ldots, \partial_n$.  Note that this is a complex of $R$-modules with $k$-linear differentials.
\item Let $\mathcal{C}^{\bullet}_{\l}$ be the complex
\[
0 \rightarrow R'/J_{\l,n} \rightarrow \oplus_{1 \leq 1 \leq n} R'/J_{\l,n-1} \rightarrow \cdots \rightarrow R'/J_{\l,0} \rightarrow 0
\]
defined using the derivations $\partial_1, \ldots, \partial_n$, where for all $\l$ and $s$, $J_{\l,s} = I^{\l+s}R' + (z^{\l})$.  This is a complex of $R'$-modules with $k$-linear differentials, and we have $\mathcal{C}^{\bullet}_{\l} \simeq C^{\bullet}_{\l} \otimes R'/z^{\l}$.
\item Let $H^q_{\mathfrak{m}}(C_{\l}^{\bullet})$ (resp. $H^q_{\mathfrak{n}}(\mathcal{C}^{\bullet}_{\l})$) be the complex obtained by applying local cohomology functors to the previous two complexes.  In the course of the proof of Lemma \ref{dRlcplus}, we remarked that $H^q_{\mathfrak{n}}(\mathcal{C}^{\bullet}_{\l}) \simeq H^q_{\mathfrak{m}}(C_{\l}^{\bullet}) \otimes R'/z^{\l}$.
\item Let
\[
\widetilde{C_{\l}^{\bullet}} = (0 \rightarrow \mathcal{O}_X/\mathcal{I}^{\l+n} \rightarrow \oplus_{1 \leq 1 \leq n} \mathcal{O}_X/\mathcal{I}^{\l+n-1} \rightarrow \cdots \rightarrow \mathcal{O}_X/\mathcal{I}^{\l} \rightarrow 0)
\]
and 
\[
\widetilde{\mathcal{C}_{\l}^{\bullet}} = (0 \rightarrow \mathcal{O}_{X'}/\mathcal{J}_{\l,n} \rightarrow \oplus_{1 \leq 1 \leq n} \mathcal{O}_{X'}/\mathcal{J}_{\l,n-1} \rightarrow \cdots \rightarrow \mathcal{O}_{X'}/\mathcal{J}_{\l,0} \rightarrow 0)
\]
be the sheafified versions of the first two complexes, where $\mathcal{J}_{\l,s} = \widetilde{J_{\l,s}}$ for all $\l$ and $s$.  These can be viewed as complexes of sheaves of $k$-spaces on $\widehat{X}$ (resp. $\widehat{X'}$).  We have $\widehat{\Omega}_X^{\bullet} \simeq \varprojlim \widetilde{C_{\l}^{\bullet}}$ and $\mathcal{O}_{\widehat{X'}} \otimes \widehat{\Omega}_X^{\bullet} \simeq \varprojlim \widetilde{\mathcal{C}_{\l}^{\bullet}}$ in the respective categories of complexes of sheaves.  
\item Finally, we consider a sheaf-theoretic variant of Definition \ref{otimescplx}.  If $\mathcal{F}^{\bullet}$ is a complex whose objects are sheaves of $\mathcal{O}_X$-modules and whose differentials are merely $k$-linear, the complex $\mathcal{F}^{\bullet} \otimes \mathcal{O}_{X'}/z^{\l}$ is the direct sum of $\l$ copies of $\mathcal{F}^{\bullet}$, indexed by $z^i$ for $i = 0, \ldots, \l-1$.  (At the level of objects, $\mathcal{F} \otimes \mathcal{O}_{X'}/z^{\l}$ is shorthand for the $\mathcal{O}_X$-module $\mathcal{F} \otimes_{\mathcal{O}_X} i^{*}(\mathcal{O}_{X'}/\mathcal{Z}^{\l})$, where $i$ is the closed immersion $X \hookrightarrow X'$ and $\mathcal{Z}$ is the sheaf of ideals defining this immersion.)  As an example, we have $\widetilde{\mathcal{C}_{\l}^{\bullet}} \simeq \widetilde{C_{\l}^{\bullet}} \otimes \mathcal{O}_{X'}/z^{\l}$.
\end{enumerate}
\end{definition}

For the complexes of \emph{sheaves} in Definition \ref{omnibus}, we have corresponding spectral sequences for local hypercohomology, and we will need to work with all of these.

\begin{definition}\label{ssomnibus}
If $\mathcal{F}^{\bullet}$ is a complex of sheaves of $k$-spaces on $\widehat{X}$ (or $\widehat{X'}$, which has the same underlying space), the \emph{local hypercohomology spectral sequence} for $\mathcal{F}^{\bullet}$ is the spectral sequence defined in subsection \ref{specseq} with respect to the functor $\Gamma_P$ of sections supported at the closed point.  If $\mathcal{L}^{\bullet, \bullet}$ is any Cartan-Eilenberg resolution of $\mathcal{F}^{\bullet}$ (or, more generally, a double complex resolution satisfying the conditions of Lemma \ref{sscomp}), this spectral sequence is the column-filtered spectral sequence associated with the double complex $\Gamma_P(\widehat{X}, \mathcal{L}^{\bullet, \bullet})$.  It begins $E_1^{p,q} = H^q_P(\widehat{X}, \mathcal{F}^p)$ and has abutment $\mathbf{H}^{p+q}_P(\widehat{X}, \mathcal{F}^{\bullet})$.  We introduce the following notation for the specific hypercohomology spectral sequences we will consider below:
\begin{enumerate}[(a)]
\item $\tilde{E}^{\bullet, \bullet}_{\bullet, R}$ is the local hypercohomology spectral sequence for $\widehat{\Omega}^{\bullet}_X$. (This is precisely the Hodge-de Rham spectral sequence arising from the surjection $R \rightarrow A$.)
\item $\tilde{\mathbf{E}}_{\bullet, R'}^{\bullet, \bullet}$ is the local hypercohomology spectral sequence for $\widehat{\Omega}^{\bullet}_{X'}$.  (This is precisely the Hodge-de Rham spectral sequence arising from the surjection $R' \rightarrow A$.)
\item $\tilde{\mathcal{E}}_{\bullet}^{\bullet, \bullet}$ is the local hypercohomology spectral sequence for $\mathcal{O}_{\widehat{X'}} \otimes \widehat{\Omega}_X^{\bullet}$.
\item For all $\l$, $(\tilde{E_{\l}})_{\bullet, R}^{\bullet, \bullet}$ is the local hypercohomology spectral sequence for $\widetilde{C_{\l}^{\bullet}}$.
\item For all $\l$, $(\tilde{\mathcal{E}_{\l}})_{\bullet}^{\bullet, \bullet}$ is the local hypercohomology spectral sequence for $\widetilde{\mathcal{C}_{\l}^{\bullet}}$.
\end{enumerate}
\end{definition}

The key to the proof of Theorem \ref{mainthmB}(a) is the following analogue of Lemma \ref{ssplus}.  Once this lemma is established, the rest of the proof will closely parallel our work in section \ref{homology}.

\begin{lem}\label{cohssplus}
Let $\tilde{E}_{\bullet, R}^{\bullet, \bullet}$ and $\tilde{\mathcal{E}}_{\bullet}^{\bullet, \bullet}$ be the spectral sequences of Definition \ref{ssomnibus}.  There is an isomorphism
\[
(\tilde{E}_{\bullet, R}^{\bullet, \bullet})^+ \xrightarrow{\sim} \tilde{\mathcal{E}}_{\bullet}^{\bullet, \bullet},
\]
where the object on the left-hand side is obtained by applying the formal power series operation to all the objects and differentials of the spectral sequence $\tilde{E}_{\bullet, R}^{\bullet, \bullet}$.
\end{lem}

Lemma \ref{ssplus} was a consequence of the fact that the double complexes giving rise to the two spectral sequences considered there were related by the $+$-operation of Definition \ref{plusop}.  We were therefore able to prove that lemma by working entirely at the level of double complexes.  The analogous reasoning fails here for reasons alluded to in Remark \ref{ab4}: the formal power series operation is an infinite direct \emph{product}.  The obvious extension of this operation to sheaves need not commute with the functor $\Gamma_P$, because taking stalks of sheaves (a direct limit) and direct products of sheaves (an inverse limit) need not commute.  The proof of Lemma \ref{cohssplus} will take place at the level of spectral sequences, not merely double complexes.

We will give the proof of Lemma \ref{cohssplus} in steps, as follows:

\begin{enumerate}
\item $\varprojlim (\tilde{E_{\l}})_{\bullet, R}^{\bullet, \bullet}$, defined by taking the ``term-by-term'' inverse limit of all objects and differentials in the spectral sequences, is again a spectral sequence: each term is derived from its predecessor by taking cohomology.  (An identical proof shows that $\varprojlim (\tilde{\mathcal{E}_{\l}})_{\bullet}^{\bullet, \bullet}$ is a spectral sequence.)
\item The spectral sequences $\tilde{E}^{\bullet, \bullet}_{\bullet, R}$ and $\varprojlim (\tilde{E_{\l}})_{\bullet, R}^{\bullet, \bullet}$ are isomorphic.  (An identical proof works for $\tilde{\mathcal{E}}_{\bullet}^{\bullet, \bullet}$ and $\varprojlim (\tilde{\mathcal{E}_{\l}})_{\bullet}^{\bullet, \bullet}$.)
\item For all $\l$, the spectral sequences $(\tilde{\mathcal{E}_{\l}})_{\bullet}^{\bullet, \bullet}$ and $(\tilde{E_{\l}})_{\bullet, R}^{\bullet, \bullet} \otimes R'/z^{\l}$ are isomorphic, via isomorphisms compatible with the transition maps for varying $\l$ (the latter spectral sequence, defined by the natural extension of Definition \ref{otimescplx} to spectral sequences, is a direct sum of $\l$ copies of $(\tilde{E_{\l}})_{\bullet, R}^{\bullet, \bullet}$).
\item There is an isomorphism $(\tilde{E}_{\bullet, R}^{\bullet, \bullet})^+ \xrightarrow{\sim} \tilde{\mathcal{E}}_{\bullet}^{\bullet, \bullet}$ of spectral sequences (the general statement).
\end{enumerate}

Step (4) follows immediately from step (3) by applying Lemma \ref{doublelim} to all objects of the spectral sequences, so we will not give it a separate proof below.  Steps (2) and (3) are not difficult.  Step (1), which is necessary in order for the later steps to make sense, is more difficult and depends crucially on our work in section \ref{mittleff}.

\begin{proof}[Proof of step (1)]
We show first that the $E_2$-term $\varprojlim (\tilde{E_{\l}})_{2, R}^{\bullet, \bullet}$ of the ``term-by-term'' inverse limit is obtained from the $E_1$-term by taking cohomology: that is, we show that for all $p$ and $q$, we have isomorphisms
\[
\varprojlim (\tilde{E_{\l}})_{2, R}^{p,q} = \varprojlim h^p(H^q_{\mathfrak{m}}(C_{\l}^{\bullet})) \simeq h^p(\varprojlim H^q_{\mathfrak{m}}(C_{\l}^{\bullet})) = h^p(\varprojlim (\tilde{E_{\l}})_{1, R}^{\bullet,q})
\]
of $k$-spaces.  This is the assertion that for the inverse system $\{H^q_{\mathfrak{m}}(C_{\l}^{\bullet})\}$ of complexes of $k$-spaces, taking cohomology commutes with inverse limits.  By Proposition \ref{mlcohcomm}, it suffices to check that for all $p$ and $q$, the inverse systems $\{H^q_{\mathfrak{m}}(C_{\l}^p)\}$ and $\{h^p(H^q_{\mathfrak{m}}(C_{\l}^{\bullet}))\}$ both satisfy the Mittag-Leffler condition.  By \cite[Thm. 7.1.3]{brodmann}, each $H^q_{\mathfrak{m}}(C_{\l}^p)$ is an Artinian $R$-module; as the transition maps in this inverse system are $R$-linear, induced by the canonical $R$-linear maps $C_{\l+1}^p \rightarrow C_{\l}^p$, the Mittag-Leffler condition for this first system is immediate.

In order to verify the Mittag-Leffler condition for the second system, we consider the Matlis dual of the first system.  For all $q$ and $\l$, the differentials in the complex $H^q_{\mathfrak{m}}(C_{\l}^{\bullet})$ have Matlis duals by Proposition \ref{klindual}(a), since they are $k$-linear maps between Artinian $R$-modules.  The transition maps in the \emph{direct} system $\{D(H^q_{\mathfrak{m}}(C_{\l}^{\bullet}))\}$ of complexes are $R$-linear in each degree, and each $D(H^q_{\mathfrak{m}}(C_{\l}^p))$ is a finitely generated $R$-module (since it is the Matlis dual of an Artinian $R$-module).  In fact, by Theorem \ref{groth}, $D(H^q_{\mathfrak{m}}(C_{\l}^p))$ is a direct sum of $n \choose p$ copies of the $R$-module $\Ext_R^{n-q}(R/I^{\l+n-p},R)$, and so $\varinjlim D(H^q_{\mathfrak{m}}(C_{\l}^p))$ is a direct sum of $n \choose p$ copies of $\varinjlim \Ext_R^{n-q}(R/I^{\l+n-p},R) \simeq H^{n-q}_I(R)$.  Finally, the de Rham complex of the $\D(R,k)$-module $H^{n-q}_I(R)$ is the direct limit of the Matlis dual complexes $D(H^q_{\mathfrak{m}}(C_{\l}^{\bullet}))$, since the inverse limit of their Matlis duals is the de Rham complex of $D(H^{n-q}_I(R)) = H^q_P(\widehat{X}, \mathcal{O}_{\widehat{X}})$ by Theorem \ref{loccohtranspose} and the definition of the inverse limit action on $H^q_P(\widehat{X}, \mathcal{O}_{\widehat{X}})$.

The direct system $\{D(H^q_{\mathfrak{m}}(C_{\l}^{\bullet}))\}$ thus satisfies the hypotheses of Corollary \ref{fgstabdR}: it is a direct system of complexes with $k$-linear differentials whose objects are finitely generated $R$-modules, the transition maps are $R$-linear in each degree, and the direct limit is the de Rham complex $H^{n-q}_I(R) \otimes \Omega_R^{\bullet}$ of a holonomic $\D(R,k)$-module.  We conclude from that corollary that for all $p$ and $\l$, the images of $h^{n-p}(D(H^q_{\mathfrak{m}}(C_{\l}^{\bullet})))$ in $h^{n-p}(D(H^q_{\mathfrak{m}}(C_{\l+s}^{\bullet})))$ stabilize in the strong sense as $s$ varies, with finite-dimensional stable image.  By Corollary \ref{artinian} and Proposition \ref{klindual}, all the Matlis duals (of objects and differentials) in the direct system $\{D(H^q_{\mathfrak{m}}(C_{\l}^{\bullet}))\}$ coincide with $k$-linear duals.  Since $k$-linear dual is a contravariant, exact functor, we thus have $h^{n-p}(D(H^q_{\mathfrak{m}}(C_{\l}^{\bullet}))) \simeq (h^p(H^q_{\mathfrak{m}}(C_{\l}^{\bullet})))^{\vee}$, and by Lemma \ref{dualml}, the inverse system $\{h^p(H^q_{\mathfrak{m}}(C_{\l}^{\bullet}))\}$ satisfies the Mittag-Leffler condition, as desired.

Examining the proof of Lemma \ref{dualml}, we see that in fact we can conclude something stronger than the Mittag-Leffler condition, namely the following: for all $p$ and $q$, given $\l$, there exists $s$ such that the image of $h^p(H^q_{\mathfrak{m}}(C_{\l+s}^{\bullet}))$ in $h^p(H^q_{\mathfrak{m}}(C_{\l}^{\bullet}))$ is a finite-dimensional $k$-space.  That is, the inverse system $\{h^p(H^q_{\mathfrak{m}}(C_{\l}^{\bullet}))\}$ is \emph{eventually finite}.  Since any descending chain of $k$-subspaces of a finite-dimensional $k$-space must terminate, it is clear that eventual finiteness implies the Mittag-Leffler condition.  But what is more, eventual finiteness is inherited by cohomology: if, for any $r$, the inverse system $\{(\tilde{E_{\l}})_{r, R}^{p,q}\}$ is eventually finite, so is the inverse system $\{(\tilde{E_{\l}})_{r+1, R}^{p,q}\}$, since the objects $(\tilde{E_{\l}})_{r+1, R}^{p,q}$ are subquotients of the objects $(\tilde{E_{\l}})_{r, R}^{p,q}$.  By induction on $r$, we conclude that the $E_{r+1}$-term of the ``term-by-term'' inverse limit, $\varprojlim (\tilde{E_{\l}})_{r+1, R}^{\bullet, \bullet}$, is obtained from the $E_r$-term by taking cohomology, and so $\varprojlim (\tilde{E_{\l}})_{\bullet, R}^{\bullet, \bullet}$ is a well-defined spectral sequence.
\end{proof}

\begin{proof}[Proof of step (2)]
Since $\widehat{\Omega}_X^{\bullet} \simeq \varprojlim \widetilde{C_{\l}^{\bullet}}$ as complexes of sheaves on $\widehat{X}$, we have projection maps $\pi_{\l}: \widehat{\Omega}_X^{\bullet} \rightarrow \widetilde{C_{\l}^{\bullet}}$ for all $\l$.  As described in subsection \ref{specseq}, each such projection map induces a morphism between the corresponding local hypercohomology spectral sequences.  By the universal property of inverse limits, this family of projection maps induces a morphism of spectral sequences $\tilde{E}^{\bullet, \bullet}_{\bullet, R} \rightarrow \varprojlim (\tilde{E_{\l}})_{\bullet, R}^{\bullet, \bullet}$ (where the right-hand side is a well-defined spectral sequence by step (1)).  We claim this is an isomorphism; by Proposition \ref{E1isos}, it suffices to check this on the $E_1$-objects of both sides.  By definition, for all $p$ and $q$, we have $\tilde{E}_{1,R}^{p,q} = H^q_P(\widehat{X}, \widehat{\Omega}^p_X)$, a direct sum of $n \choose p$ copies of $H^q_P(\widehat{X}, \mathcal{O}_{\widehat{X}})$, and for all $\l$, we have $(\tilde{E}_{\l})_{1,R}^{p,q} = H^q_P(\widehat{X}, \widetilde{C_{\l}^p})$, a direct sum of $n \choose p$ copies of $H^q_{\mathfrak{m}}(R/I^{\l+n-p})$.  The induced map $\tilde{E}_{1,R}^{p,q} \rightarrow \varprojlim (\tilde{E}_{\l})_{1,R}^{p,q}$ is therefore a finite direct sum of copies of the canonical map $H^q_P(\widehat{X}, \mathcal{O}_{\widehat{X}}) \rightarrow \varprojlim H^q_{\mathfrak{m}}(R/I^{\l+n-p})$, which we already know to be an isomorphism by Proposition \ref{ogusiso}.
\end{proof}

\begin{proof}[Proof of step (3)]
This is the only step in the proof where we can work entirely at the level of the double complexes giving rise to the spectral sequences.  Let $\l$ be given.  Choose a Cartan-Eilenberg resolution $\widetilde{C_{\l}^{\bullet}} \rightarrow \mathcal{L}_{\l}^{\bullet, \bullet}$ in the category of sheaves of $k$-spaces on $\widehat{X}$.  By Remark \ref{ab4}, this category satisfies axiom $AB4$, and so direct sums are exact; furthermore, direct sums of injective sheaves are again injective, and the functor $\Gamma_P$ commutes with direct sums.  We conclude that the double complex $\mathcal{L}_{\l}^{\bullet, \bullet} \otimes \mathcal{O}_{X'}/z^{\l}$ (which is simply a finite direct sum of copies of $\mathcal{L}_{\l}^{\bullet, \bullet}$) satisfies the conditions of Lemma \ref{sscomp}, and so the double complex $\Gamma_P(\widehat{X}, \mathcal{L}_{\l}^{\bullet, \bullet} \otimes \mathcal{O}_{X'}/z^{\l})$ gives rise to the local hypercohomology spectral sequence for the complex $\widetilde{C_{\l}^{\bullet}} \otimes \mathcal{O}_{X'}/z^{\l} \simeq \widetilde{\mathcal{C}_{\l}^{\bullet}}$.  On the one hand, this spectral sequence is $(\tilde{\mathcal{E}_{\l}})_{\bullet}^{\bullet, \bullet}$ by definition.  On the other hand, since $\Gamma_P$ commutes with direct sums, we have $\Gamma_P(\widehat{X}, \mathcal{L}_{\l}^{\bullet, \bullet} \otimes \mathcal{O}_{X'}/z^{\l}) \simeq \Gamma_P(\widehat{X}, \mathcal{L}_{\l}^{\bullet, \bullet}) \otimes R'/z^{\l}$ as double complexes of $k$-spaces, and since cohomology of $k$-spaces commutes with direct sums, this isomorphism induces an isomorphism $(\tilde{\mathcal{E}_{\l}})_{\bullet}^{\bullet, \bullet} \simeq (\tilde{E_{\l}})_{\bullet, R}^{\bullet, \bullet} \otimes R'/z^{\l}$ at the level of spectral sequences, as desired.  The fact that these isomorphisms are compatible with the transition maps for varying $\l$ follows from the fact that the same is true for the isomorphisms $\widetilde{C_{\l}^{\bullet}} \otimes \mathcal{O}_{X'}/z^{\l} \simeq \widetilde{\mathcal{C}_{\l}^{\bullet}}$, since as discussed in subsection \ref{specseq}, the association of a local hypercohomology spectral sequence with a complex of sheaves is functorial in the complex.  This completes the proof of Lemma \ref{cohssplus}.
\end{proof}

We are now ready to complete the proof of Theorem \ref{mainthmB}.

\begin{proof}[Proof of Theorem \ref{mainthmB}(a)]
Consider again the short exact sequence from the proof of Proposition \ref{cohspecialcase}:
\[
0 \rightarrow \mathcal{O}_{\widehat{X'}} \otimes \widehat{\Omega}_X^{\bullet}[-1] \xrightarrow{\iota} \widehat{\Omega}_{X'}^{\bullet} \xrightarrow{\pi} \mathcal{O}_{\widehat{X'}} \otimes \widehat{\Omega}_X^{\bullet} \rightarrow 0.
\]
As described in subsection \ref{specseq}, the morphism of complexes $\pi$ induces a morphism between the corresponding spectral sequences for local hypercohomology, given in Definition \ref{ssomnibus}.  That is, there is an induced morphism
\[
\pi_{\bullet}^{\bullet, \bullet}: \tilde{\mathbf{E}}_{\bullet, R'}^{\bullet, \bullet} \rightarrow \tilde{\mathcal{E}}_{\bullet}^{\bullet, \bullet}.
\]
Identifying first $\tilde{\mathcal{E}}_{\bullet}^{\bullet, \bullet}$ with $(\tilde{E}_{\bullet, R}^{\bullet, \bullet})^+$ (by Lemma \ref{cohssplus}) and then $\tilde{E}_{\bullet, R}^{\bullet, \bullet}$ with the ``constant term'' component of $(\tilde{E}_{\bullet, R}^{\bullet, \bullet})^+$, we see that this further induces a morphism
\[ 
\psi_{\bullet}^{\bullet,\bullet}: \tilde{\mathbf{E}}_{\bullet, R'}^{\bullet, \bullet} \rightarrow \tilde{E}_{\bullet, R}^{\bullet, \bullet}
\]
given in every degree by $\pi_{\bullet}^{\bullet, \bullet}$ followed by the projection of $\tilde{\mathcal{E}}_{\bullet}^{\bullet, \bullet} \simeq (\tilde{E}_{\bullet, R}^{\bullet, \bullet})^+$ on its constant term component.  If $r=2$, the maps $\psi_2^{p,q}$ are precisely the inverses of the isomorphisms $H^p_{dR}(H^q_P(\widehat{X}, \mathcal{O}_{\widehat{X}})) \simeq H^p_{dR}(H^q_P(\widehat{X'}, \mathcal{O}_{\widehat{X'}}))$ appearing in the proof of Proposition \ref{cohspecialcase}, which were induced by the morphism of complexes $\pi$ and the projection of $(\tilde{E}_{2,R}^{p,q})^+ = (H^p_{dR}(H^q_P(\widehat{X}, \mathcal{O}_{\widehat{X}})))^+$ on its constant term component.  Therefore the morphism $\psi^{\bullet,\bullet}_{\bullet}$ of spectral sequences is an isomorphism at the $E_2$-level.  By Proposition \ref{E1isos}, it follows that $\psi$ is an isomorphism at all later levels, including the abutments.  The proof is complete.
\end{proof}

\section{Some open questions}\label{openqs}

In this final section, we pose two questions and state a conjecture left open by our work above.

Our first question concerns coefficient fields.  \emph{A priori}, the isomorphism classes of the spectral sequences $\{E_{r,R}^{p,q}\}$ and $\{\tilde{E}_{r,R}^{p,q}\}$, as well as the integers $\rho_{p,q}$ of Definition \ref{rho}, depend on the choice of coefficient field $k \subset A$.  

\begin{question}\label{cohfield}
Are the isomorphism classes of the spectral sequences $\{E_{r,R}^{p,q}\}$ and $\{\tilde{E}_{r,R}^{p,q}\}$, as well as the integers $\rho_{p,q}$ of Definition \ref{rho}, independent of the choice of coefficient field $k \subset A$?
\end{question}

In the case where the local cohomology modules $H^q_I(R)$ are supported only at the maximal ideal $\mathfrak{m}$, the $\rho_{p,q}$ are indeed independent of this choice by Remark \ref{lyunos} (since the $\lambda_{p,q}$ are known to be independent), which provides supporting evidence for a positive answer to Question \ref{cohfield}.

Our second question concerns degeneration of the spectral sequences.  In general, the $E_1$-objects $E_1^{p,q} = H^q_Y(X, \Omega_X^p)$ in the \emph{local} Hodge-de Rham homology spectral sequence (for a complete local ring $A$ with coefficient field $k$ of characteristic zero) are not even finitely generated as $R$-modules (where $R \twoheadrightarrow A$ is as in section \ref{homology}), so the spectral sequence need not degenerate at $E_1$.  However, this obstruction does not exist for the $E_2$-term since, by Theorem \ref{mainthmA}(b), the $E_2$-objects are finite-dimensional $k$-spaces.  Therefore, the question of possible degeneration at $E_2$ reduces to a comparison of dimensions: do we have equalities $\dim_k \mathbf{H}^l_Y(X, \Omega_X^{\bullet}) = \sum_{p+q=l} \dim_k E_2^{p,q} = \sum_{p+q=l} \dim_k H^p_{dR}(H^q_I(R))$ for all $\l$?

\begin{question}\label{degen}
Does the local Hodge-de Rham homology spectral sequence (for a complete local ring $A$ with coefficient field $k$ of characteristic zero) degenerate at $E_2$? 
\end{question}

Note that by Theorems \ref{hartind}(c) and \ref{dualE2}, degeneration of the homology spectral sequence at $E_2$ is equivalent to degeneration of the cohomology spectral sequence at $E_2$, so it is enough to answer the question for the homology sequence.  Even if the answer to Question \ref{degen} is negative in general, it would be interesting to find conditions on the complete local ring $A$ under which such degeneration occurs.

Finally, we restate a conjecture which appeared already in Section \ref{intro}:

\begin{conjecture}\label{dualss}
Let $A$, $k$, and $R$ be as in the statement of Theorem \ref{mainthmA}.  Beginning with the $E_2$-terms, the homology and cohomology spectral sequences are $k$-dual to each other: for all $r \geq 2$, and for all $p$ and $q$, $E_r^{n-p,n-q} \simeq (\tilde{E}_r^{pq})^{\vee}$, and similarly for the differentials.
\end{conjecture}

So far, we have only proved the $k$-duality asserted in Conjecture \ref{dualss} for the $E_2$-\emph{objects}.  We remark that although our proof of Theorem \ref{mainthmB}(a) is independent of Theorem \ref{mainthmA}, the former becomes an immediate corollary of the latter if Conjecture \ref{dualss} is known.   

\bibliographystyle{plain}

\bibliography{switalaDR}

\end{document}